\documentclass[12pt]{article}
\usepackage{mathtools,amsthm,amssymb}
\usepackage[usenames,dvipsnames]{xcolor}
\usepackage[pdftitle={Active spanning trees with bending energy on planar maps and SLE-decorated Liouville quantum gravity for kappa>8},
  pdfauthor={Ewain Gwynne, Adrien Kassel, Jason Miller, and David B. Wilson},
  colorlinks=true,linkcolor=NavyBlue,urlcolor=RoyalBlue,citecolor=PineGreen,bookmarks=true,bookmarksopen=true,bookmarksopenlevel=2,unicode=true,linktocpage]{hyperref}
\usepackage{enumerate}
\usepackage{graphicx}
\usepackage{cite}
\usepackage{comment}
\usepackage{oands}
\usepackage{tikz}
\usepackage{changepage}
\usepackage{bbm}
\usepackage[margin=1in]{geometry}
\usepackage[pagewise,mathlines]{lineno}
\usepackage{appendix}
\usepackage{colonequals}
\usepackage[normalem]{ulem}
\usepackage[usenames,dvipsnames]{xcolor}

\usepackage{microtype}
\usepackage[position=b]{subcaption}

\newcommand{\old}[1]{}
\newcommand{\SLE}{\operatorname{SLE}}

\theoremstyle{plain}
\newtheorem{thm}{Theorem}[section]
\newtheorem{cor}[thm]{Corollary}
\newtheorem{lem}[thm]{Lemma}
\newtheorem{prop}[thm]{Proposition}
\newtheorem{conj}[thm]{Conjecture}

\def\@rst #1 #2other{#1}
\newcommand\MR[1]{\relax\ifhmode\unskip\spacefactor3000 \space\fi
  \MRhref{\expandafter\@rst #1 other}{#1}}
\newcommand{\MRhref}[2]{\href{http://www.ams.org/mathscinet-getitem?mr=#1}{MR#2}}

\newcommand{\arXiv}[1]{\href{http://arxiv.org/abs/#1}{arXiv:#1}}
\newcommand{\arxiv}[1]{\href{http://arxiv.org/abs/#1}{#1}}

\theoremstyle{definition}
\newtheorem{defn}[thm]{Definition}
\newtheorem{remark}[thm]{Remark}

\numberwithin{equation}{section}

\newcommand{\dsb}{\begin{adjustwidth}{2.5em}{0pt}
\begin{footnotesize}}
\newcommand{\dse}{\end{footnotesize}
\end{adjustwidth}}

\newcommand{\ssb}{\begin{adjustwidth}{2.5em}{0pt}}
\newcommand{\sse}{\end{adjustwidth}}

\newcommand{\aryb}{\begin{eqnarray*}}
\newcommand{\arye}{\end{eqnarray*}}
\def\alb#1\ale{\begin{align*}#1\end{align*}}
\newcommand{\eqb}{\begin{equation}}
\newcommand{\eqe}{\end{equation}}
\newcommand{\eqbn}{\begin{equation*}}
\newcommand{\eqen}{\end{equation*}}

\makeatletter
 \newcommand{\hypertop}[1]{\Hy@raisedlink{\hypertarget{#1}{}}}
\makeatother
\newcommand{\cI}{\protect\hyperlink{def-identification}{\mcl I}}
\newcommand{\cN}{\protect\hyperlink{def-theta-count}{\mcl N\!}}
\newcommand{\cB}{\protect\hyperlink{def-theta-count}{\mcl B}}
\newcommand{\cO}{\protect\hyperlink{def-theta-count}{\mcl O}}
\newcommand{\cC}{\protect\hyperlink{def-theta-count}{\mcl C}}
\newcommand{\cQ}{\protect\hyperlink{def-theta-count}{\mcl D}}
\newcommand{\dd}{\protect\hyperlink{def-theta-count}{\vec{d}}}
\newcommand{\dpt}{\protect\hyperlink{def-theta-count}{d}}
\newcommand{\ddt}{\protect\hyperlink{def-theta-count}{d^*}}
\newcommand{\dipt}{\protect\hyperlink{def-d-z}{d}}
\newcommand{\didt}{\protect\hyperlink{def-d-z}{d^*}}
\newcommand{\ddi}{\protect\hyperlink{def-d-z}{\vec{d}}}
\newcommand{\cR}{\protect\hyperlink{def-reduce}{\mcl R}}
\newcommand{\J}{\protect\hyperlink{def-J}{J}}
\newcommand{\CHI}{\protect\hyperlink{def-J}{\chi}}
\newcommand{\Jh}{J^\hb}

\newcommand{\BB}{\mathbbm}
\newcommand{\one}{\BB 1}
\newcommand{\E}{\BB E}
\newcommand{\CC}{\BB C}
\newcommand{\R}{\BB R}
\newcommand{\Z}{\BB Z}
\newcommand{\N}{\BB N}
\newcommand{\Q}{\BB Q}
\newcommand{\pvec}{\vec p}
\newcommand{\ol}{\overline}

\newcommand{\op}{\operatorname}

\newcommand{\eqD}{\overset{d}{=}}
\newcommand{\ep}{\varepsilon}
\newcommand{\rta}{\rightarrow}

\newcommand{\wt}{\widetilde}
\newcommand{\wh}{\widehat}
\newcommand{\mcl}{\mathcal}

\newcommand{\act}{\operatorname{a}}
\newcommand{\dup}{\operatorname{d}}

\newcommand*\tb[1]{{\mathsf{#1}}}
\newcommand\cb{\tb c}
\newcommand\hb{\tb h}
\newcommand\db{\tb d}
\newcommand\eb{\tb e}
\newcommand\co{\tb C}
\newcommand\ho{\tb H}
\newcommand\fo{\tb F}
\newcommand\so{\tb S}

\newcommand\Hc{{\ho\cb}}
\newcommand\Ch{{\co\hb}}

\newcommand\PP{\BB P}

\let\originalleft\left
\let\originalright\right
\renewcommand{\left}{\mathopen{}\mathclose\bgroup\originalleft}
\renewcommand{\right}{\aftergroup\egroup\originalright}

\title{\vspace*{-75pt}\mbox{Active~spanning~trees~with~bending~energy~on~planar~maps} and SLE-decorated Liouville quantum gravity for $\kappa > 8$}
\date{}
\author{
\begin{tabular}{c}Ewain Gwynne\\[-5pt]\small MIT\end{tabular}\quad
\begin{tabular}{c}Adrien Kassel\\[-5pt]\small ENS Lyon\end{tabular}\quad
\begin{tabular}{c}Jason Miller\\[-5pt]\small Cambridge\end{tabular}\quad
\begin{tabular}{c}David B.\! Wilson\\[-5pt]\small Microsoft Research\end{tabular}
}

\begin{document}

\maketitle
\thispagestyle{empty}

\begin{abstract}
  We introduce a two-parameter family of probability measures on
  spanning trees of a planar map. One of the parameters controls the
  activity of the spanning tree and the other is a measure of its
  bending energy.  When the bending parameter is 1, we recover the
  active spanning tree model, which is closely related to the critical
  Fortuin--Kasteleyn model.  A random planar map decorated by a
  spanning tree sampled from our model can be encoded by means of a
  generalized version of Sheffield's hamburger-cheeseburger bijection.
  Using this encoding, we prove that for a range of parameter values
  (including the ones corresponding to maps decorated by an active
  spanning tree), the infinite-volume limit of spanning-tree-decorated
  planar maps sampled from our model converges in the peanosphere
  sense, upon rescaling, to an $\SLE_\kappa$-decorated
  $\gamma$-Liouville quantum cone with $\kappa > 8$
  and $\gamma = 4/\sqrt\kappa \in (0,\sqrt 2)$.
\end{abstract}

\enlargethispage{50pt}
\tableofcontents


\section{Introduction}
\label{sec-intro}
\subsection{Overview}

We study a family of probability measures on spanning-tree-decorated rooted planar maps, which we
define in Section~\ref{sec-burger}, using a generalization of the Sheffield hamburger-cheeseburger model~\cite{shef-burger}.
This family includes as special cases maps decorated by a uniform spanning tree \cite{mullin-maps},
planar maps together with a critical Fortuin--Kasteleyn (FK) configuration~\cite{shef-burger},
and maps decorated by an active spanning tree~\cite{kassel-wilson-active}.
These models converge in a certain sense (described below) to
Liouville quantum gravity (LQG) surfaces decorated by Schramm--Loewner evolution ($\SLE_\kappa$)~\cite{schramm0},
and any value of~$\kappa>4$ corresponds to some measure in the family.
Although our results are motivated by SLE and LQG, our proofs are entirely self-contained,
requiring no knowledge beyond elementary probability theory.

Consider a spanning-tree-decorated rooted planar map $(M, e_0, T)$, where $M$
is a planar map, $e_0$ is an oriented root edge for $M$, and $T$ is a
spanning tree of $M$. Let $M^*$ be the dual map of $M$ and let
$T^*$ be the dual spanning tree, which consists of the edges of $M^*$
which do not cross edges of $T$.
Let~$Q$ be the quadrangulation whose vertex set is the union of the vertex sets of $M$ and $M^*$, obtained by identifying each vertex of $M^*$ with a point in the corresponding face of $M$, then connecting it by an edge (in $Q$) to each vertex of $M$ on the boundary of that face. Each face of~$Q$ is bisected by either an edge of $T$ or an edge of $T^*$ (but not both).
Let $\BB e_0$ be the oriented edge of~$Q$ with the same initial endpoint as~$e_0$ and which is the first edge in the clockwise direction from $e_0$ among all such edges.
As explained in, e.g.,~\cite[\S~4.1]{shef-burger}, there is a path $\lambda$ consisting of edges of (the dual of) $Q$
which snakes between the primal tree $T$ and dual tree $T^*$, starts and ends at $\BB e_0$, and hits each edge of $Q$ exactly once.
This path $\lambda$ is called the \textit{Peano curve\/} of $(M,e_0,T)$. See Figure~\ref{fig:map} for an illustration.

For Euclidean lattices, Lawler, Schramm, and Werner \cite{lsw-lerw-ust} showed that the uniform spanning tree Peano curve converges to $\SLE_8$.  For random tree-decorated planar maps with suitable weights coming from the critical Fortuin--Kasteleyn model, Sheffield \cite{shef-burger} proved a convergence result which, when combined with the continuum results of \cite{wedges}, implies that the Peano curve converges in a certain sense to a space-filling version of $\SLE_\kappa$ with $4<\kappa\leq 8$ on an LQG surface.  The measures on tree-decorated planar maps we consider generalize these, and converge in this same sense to $\SLE_\kappa$ with $4<\kappa<\infty$.

\begin{figure}[htb!]
\centering
\vspace{-9pt}
\includegraphics[width=.9\textwidth]{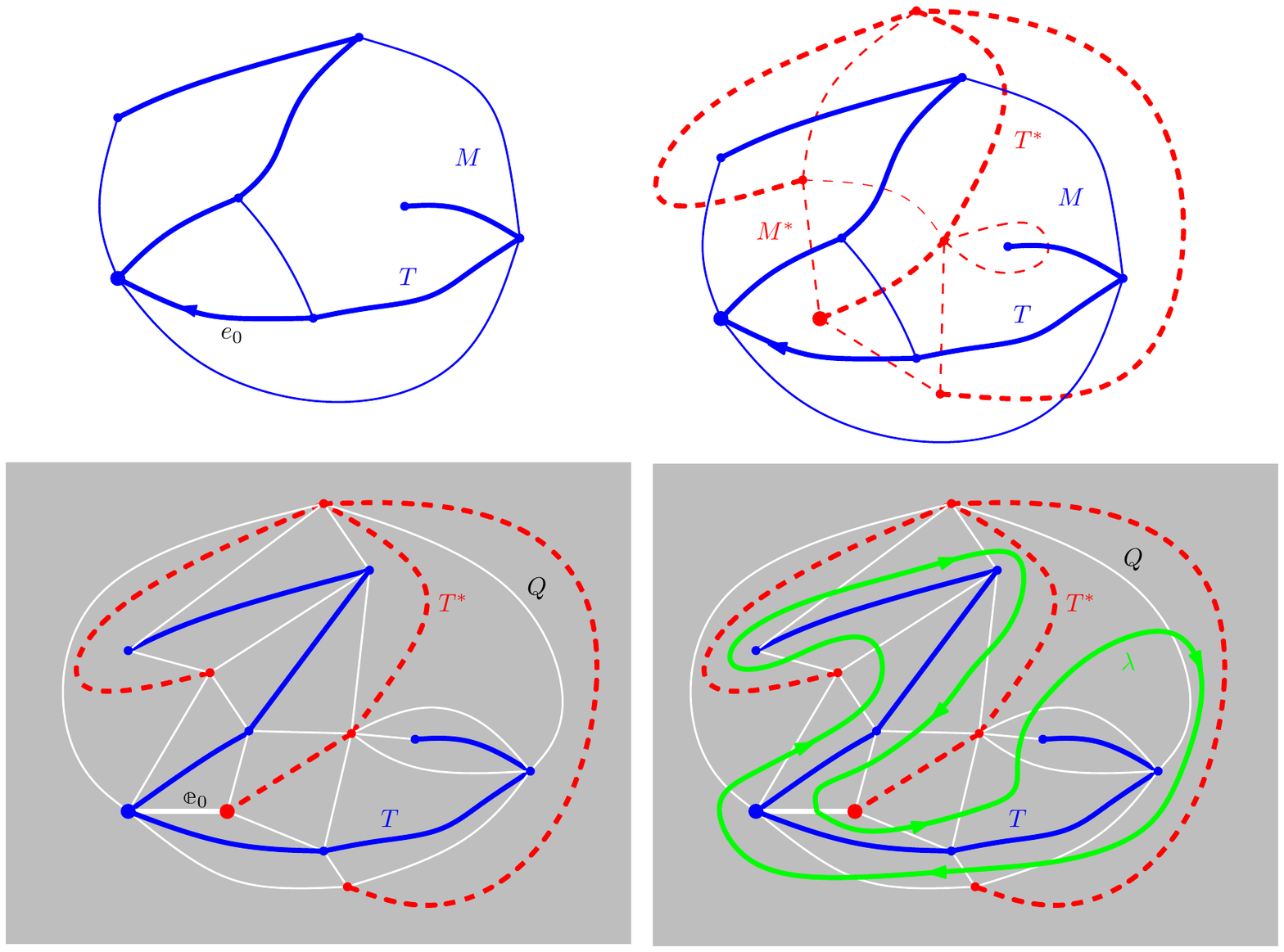}
\caption{Top left: a rooted map $(M, e_0)$ (in blue) with a spanning tree $T$ (heavier blue lines). Top right: the dual map $M^*$ (dashed red) with the dual spanning tree $T^*$ (heavier dashed red lines). Bottom left: the quadrangulation $Q$ (in white) whose vertices are the vertices of $M$ and $M^*$. Bottom right: the Peano curve $\lambda$ (in green), exploring clockwise. Formally, $\lambda$ is a cyclic ordering of the edges of $Q$ with the property that successive edges share an endpoint. The triple $(M, e_0, T)$ can be encoded by means of a two-dimensional simple walk excursion in the first quadrant with $2n$ steps, equivalently a word consisting of elements of the set $\Theta_0$ defined below which reduces to the empty word; see Figure~\ref{fig:word}. }\label{fig:map}
\end{figure}

For the measures on tree-decorated planar maps which we consider in this paper, the conjectured scaling limit of the Peano curve $\lambda$ is a \textit{whole-plane space-filling $\SLE_\kappa$ from $\infty$ to $\infty$\/} for an appropriate value of $\kappa > 4$. In the case when $\kappa \geq 8$, $\SLE_\kappa$ is space-filling \cite{schramm-sle}, and whole-plane space-filling $\SLE_\kappa$ from $\infty$ to $\infty$ is just a whole-plane variant of chordal $\SLE_\kappa$ (see~\cite[footnote~9]{wedges}). It is characterized by the property that for any stopping time $\tau$ for the curve, the conditional law of the part of the curve which has not yet been traced is that of a chordal SLE$_{\kappa}$ from the tip of the curve to $\infty$.  Ordinary $\SLE_\kappa$ for $\kappa \in (4,8)$ is not space-filling \cite{schramm-sle}.  In this case, whole-plane space-filling $\SLE_\kappa$ from $\infty$ to $\infty$ is obtained from a whole-plane variant of ordinary chordal $\SLE_\kappa$ by iteratively filling in the ``bubbles'' disconnected from $\infty$ by the curve. The construction of space-filling $\SLE_\kappa$ in this case is explained in~\cite[\S~1.2.3 and 4.3]{ig4}. For $\kappa > 4$, whole-plane space-filling $\SLE_\kappa$ is the Peano curve of a certain tree of $\SLE_{16/\kappa}$-type curves, namely the set of all flow lines (in the sense of~\cite{ig1,ig2,ig3,ig4}) of a whole-plane Gaussian free field (GFF) started from different points but with a common angle.

There are various ways to formulate the convergence of
spanning-tree-decorated planar maps toward space-filling
$\SLE_\kappa$-decorated LQG surfaces. One can embed the map $M$
into~$\CC$ (e.g.\ via circle packing or Riemann uniformization) and
conjecture that the Peano curve of $T$ (resp.\ the measure
which assigns mass $1/n$ to each vertex of $M$) converges in the
Skorokhod metric (resp.\ the weak topology) to the space-filling
$\SLE_\kappa$ (resp.\ the volume measure associated with the
$\gamma$-LQG surface). Alternatively, one can first try to define a
metric on an LQG surface (which has so far been accomplished only in
the case when $\gamma = \sqrt{8/3}$~\cite{qle,sphere-constructions,tbm-characterization,lqg-tbm1,lqg-tbm2,lqg-tbm3},
in which case it is isometric to some variant of the Brownian
map~\cite{legall-sphere-survey,miermont-survey}), and then try to show
that the graph metric on $M$ (suitably rescaled) converges in the
Hausdorff metric to an LQG surface.  Convergence in the former sense
has only recently been shown for ``mated-CRT maps'' using the
  Tutte (harmonic or barycentric) embedding \cite{gms-tutte}.
  It has not yet been proved for any other random planar map model, and convergence
in the latter (metric) sense has been established only for
uniform planar maps and slight variants thereof (which correspond to
$\gamma=\sqrt{8/3}$)~\cite{legall-uniqueness,miermont-brownian-map}.

\begin{figure}[ht!]
\begin{center}
\vspace{-30pt}
\begin{minipage}[t]{0.43\textwidth}
\includegraphics[width=\textwidth]{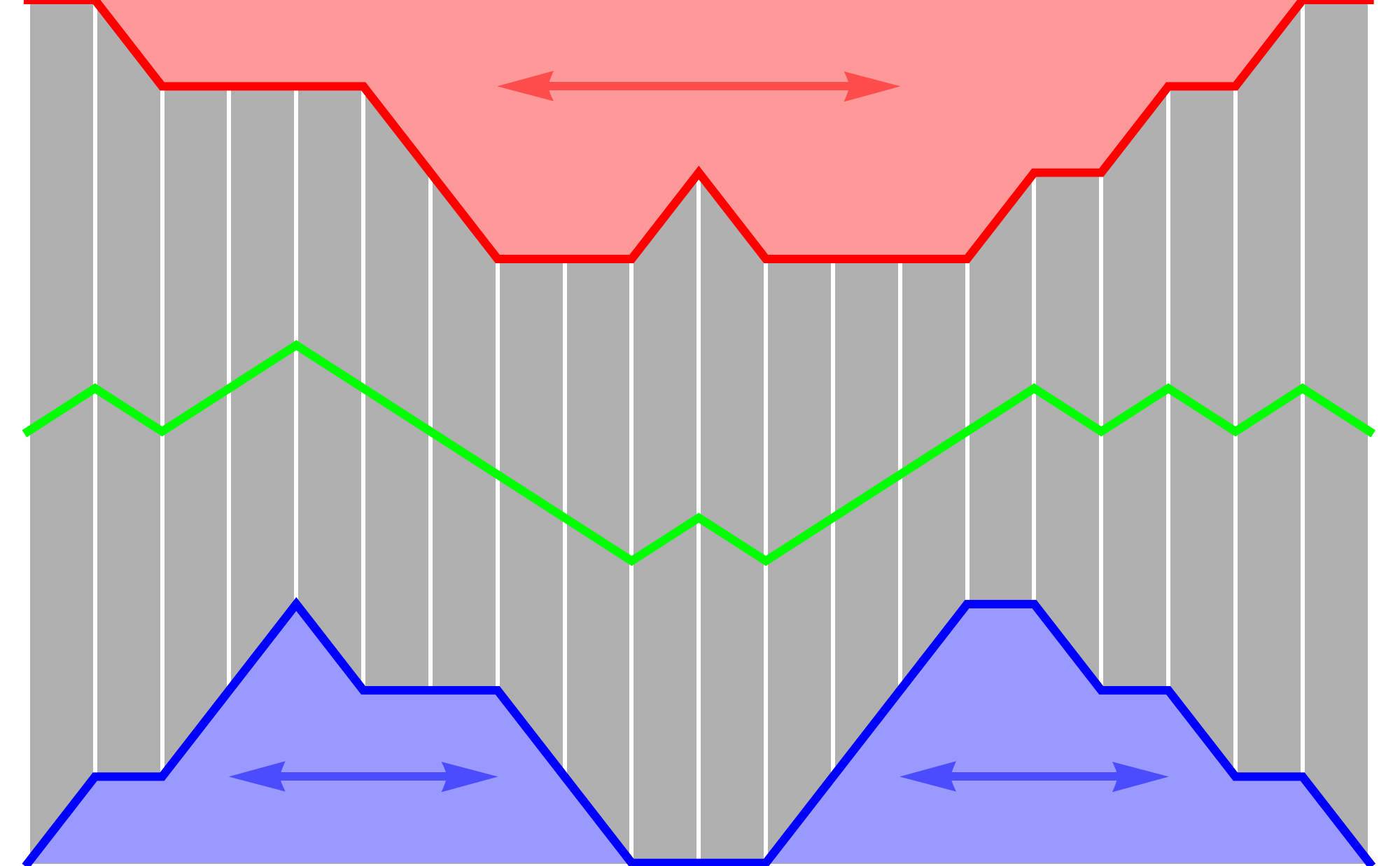}
\end{minipage}
\hspace{-8pt}
\begin{minipage}[t]{0.12\textwidth}
\vspace{-0.13\textheight}
\includegraphics[scale=0.9,trim=0mm 0mm 10mm 0mm]{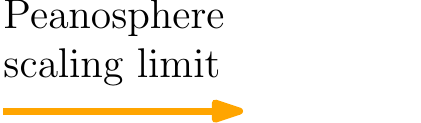}
\end{minipage}
\hspace{1pt}
\begin{minipage}[t]{0.43\textwidth}
\includegraphics[width=\textwidth]{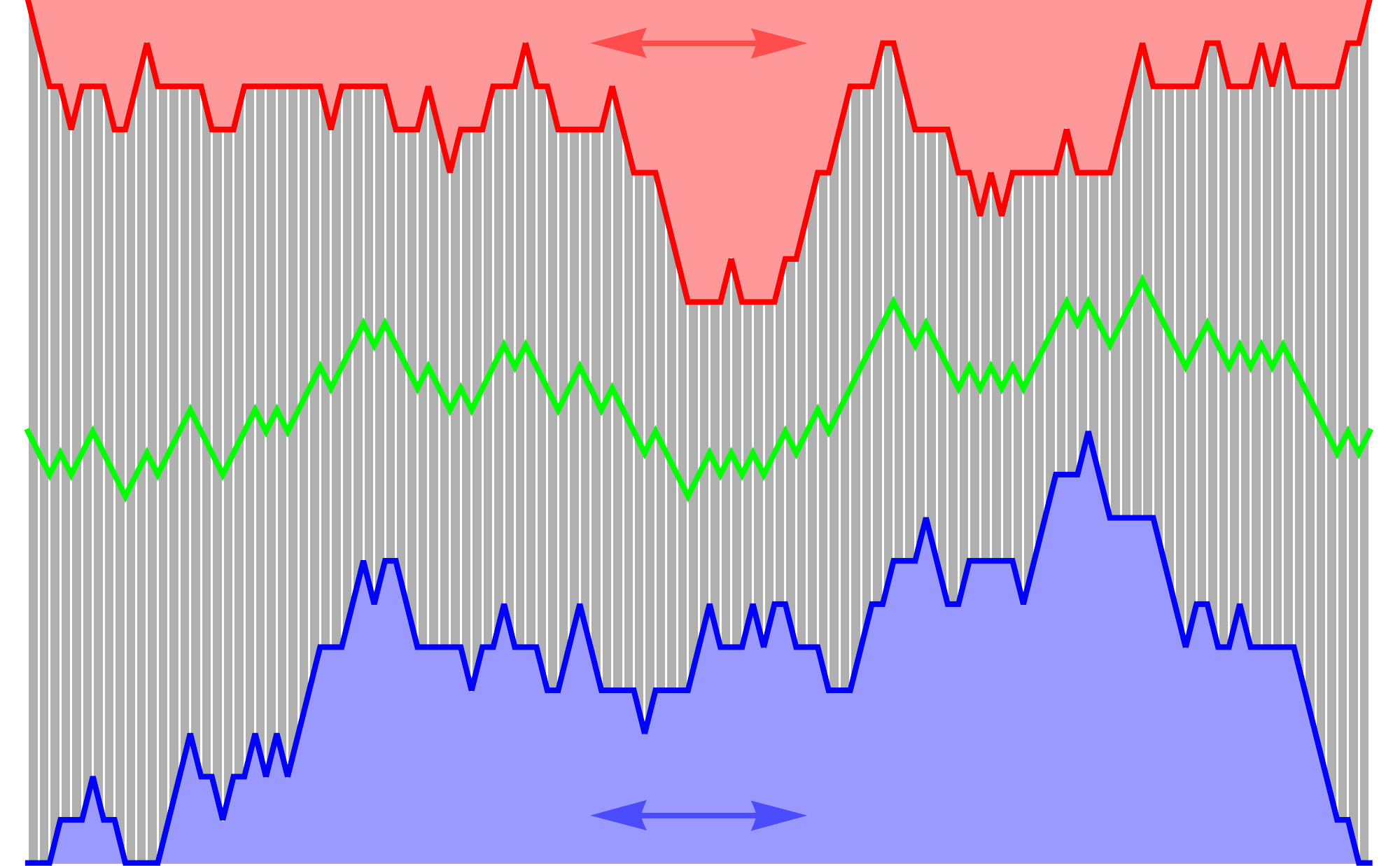}
\end{minipage}
\begin{minipage}[t]{0.39\textwidth}
\includegraphics[scale=0.85,page=1]{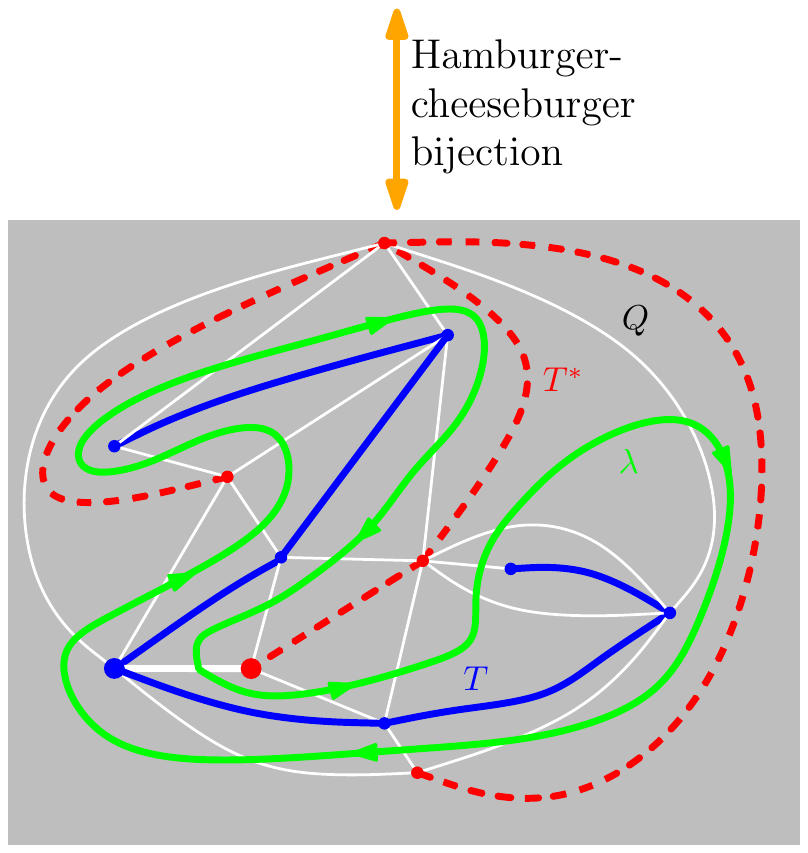}
\end{minipage}
\begin{minipage}[t]{0.12\textwidth}
\phantom{\includegraphics[scale=0.9,trim=0mm 0mm 10mm 0mm]{peanosphere-scaling-limit}}	
\end{minipage}
\begin{minipage}[t]{0.39\textwidth}
\hspace{-0.1\textwidth}
\includegraphics[scale=0.85,page=2]{tree-gluing}
\end{minipage}
\end{center}
\caption{
Shown on the top left are the contour functions for the discrete primal tree (blue) and dual tree (red) for the tree-decorated planar map on the bottom left using Sheffield's hamburger-cheeseburger bijection \cite{shef-burger}.
Vertices of the tree and dual tree correspond to blue and red horizontal segments in the contour representation; edges of the tree and dual tree correspond to matching up and down steps.
The white boundaries between quadrangles in the quadrangulation correspond to the white vertical segments between the blue and red contour functions; the bold white boundary in the quadrangulation, which marks the starting point of the Peano curve, corresponds to the left and right edges in the contour diagram.
The main contribution of the current paper is to establish an infinite volume version of the scaling limit result indicated by the orange horizontal arrow on the top.  That is, if one first takes a limit as the size of the map tends to infinity, then the contour functions for the infinite discrete pair of trees converge to a two-dimensional correlated Brownian motion which encode a pair of infinite continuum random trees (CRTs) --- this is convergence in the so-called \textit{peanosphere sense}.  The main result of \cite{wedges} implies that these two infinite CRTs glued together as shown (i.e.\ contracting the vertical white segments in addition to gluing along the horizontal arrows) determine their embedding into an $\SLE$-decorated LQG surface.  That is, if one observes the two contour functions on the top right, then one can measurably recover the LQG surface decorated with an $\SLE$ indicated on the bottom right and conversely if one observes the $\SLE$-decorated LQG surface on the bottom right then one can measurably recover the contour functions on the top right.  This allows us to interpret our scaling limit result as a convergence result to $\SLE$-decorated LQG.}
\label{fig:peanosphere}
\end{figure}

Here we consider a different notion of convergence, called convergence in the \textit{peanosphere sense}, which we now describe (see Figure~\ref{fig:peanosphere}).
This notion of convergence is based on the work~\cite{wedges}, which shows how to encode a $\gamma$-quantum cone (a certain type of LQG surface parametrized by $\CC$, obtained by zooming in near a point sampled from the $\gamma$-LQG measure induced by a GFF~\cite[\S~4.3]{wedges}) decorated by an independent whole-plane space-filling $\SLE_\kappa$ curve $\eta$ with $\kappa = 16/\gamma^2$ in terms of a correlated two-dimensional Brownian motion $Z$, with correlation $-\cos(4\pi/\kappa)$.  Recall that the contour function of a discrete, rooted plane tree is the function one obtains by tracing along the boundary of the tree starting at the root and proceeding in a clockwise manner and recording the distance to the root from the present vertex.  The two coordinates of the Brownian motion $Z$ are the contour functions of the $\SLE_{16/\kappa}$ tree (whose Peano curve is $\eta$) and that of the corresponding dual tree (consisting of GFF flow lines whose angles differ from the angles of the flow lines in the original tree by $\pi$).  Here, the distance to the root is measured using $\gamma$-LQG length.  On the discrete side, the entire random planar map is determined by the pair of trees.  One non-obvious fact established in \cite{wedges} is that the corresponding statement is true in the continuum: the entire $\gamma$-quantum cone and space-filling $\SLE$ turn out to be almost surely determined by the Brownian motion $Z$.  We say that the triple $(M, e_0, T)$ converges in the scaling limit (in the peanosphere sense) to a $\gamma$-quantum cone decorated by an independent whole-plane space-filling $\SLE_\kappa$ if the joint law of the contour functions (or some slight variant thereof) of the primal and dual trees~$T$ and~$T^*$ converges in the scaling limit to the joint law of the two coordinates of $Z$.

The present paper is a generalization of~\cite{shef-burger}, which was the first work to study peanosphere convergence. The paper~\cite{shef-burger} considered rooted critical FK planar maps. For $n\in\N$ and $q\geq 0$, a \textit{rooted critical FK planar map with parameter $q$ and size $n$\/} is a triple $(M,e_0,S)$ consisting of a planar map $M$ with~$n$ edges, a distinguished oriented root edge $e_0$ for $M$, and a set $S$ of edges of $M$, sampled with probability proportional to $q^{K(S)/2}$, where $K(S)$ is the number of connected components of $S$ plus the number of complementary connected components of $S$.  The conditional law of $S$ given $M$ is that of the self-dual FK model on~$M$~\cite{fk-cluster}. An \textit{infinite-volume rooted critical FK planar map with parameter~$q$\/} is the infinite-volume limit of rooted critical FK planar maps of size $n$ in the sense of Benjamini--Schramm~\cite{benjamini-schramm-topology}.

There is a natural (but not bijective) means of obtaining a spanning tree $T$ of $M$ from the FK edge set~$S$, which depends on the choice of $e_0$; see~\cite{bernardi-sandpile,shef-burger}.
It is conjectured~\cite{shef-burger,wedges} that the triple $(M, e_0, T)$ converges in the scaling limit to an LQG sphere with parameter $\gamma$ decorated by an independent whole-plane space-filling $\SLE_\kappa$ with parameters satisfying
\eqb \label{eqn-kappa-q}
\sqrt q = - 2\cos\left(\frac{4\pi}{\kappa} \right),\qquad \gamma = \frac{4}{\sqrt\kappa} \,.
\eqe

In~\cite[Thm.~2.5]{shef-burger}, this convergence is proven in the peanosphere sense in the case of infinite-volume FK planar maps. This is accomplished by means of a bijection, called the \textit{Sheffield hamburger-cheeseburger bijection}, between triples $(M, e_0, S)$ consisting of a rooted planar map of size $n$ and a distinguished edge set~$S$; and certain words in an alphabet of five symbols (representing two types of ``burgers'' and three types of ``orders''). This bijection is essentially equivalent for a fixed choice of $M$ to the bijection of~\cite{bernardi-sandpile}. The word associated with a triple $(M, e_0, S)$ gives rise to a walk on $\Z^2$ whose coordinates are (roughly speaking) the contour function of the spanning tree $T$ of $M$ naturally associated with $S$ (under the mapping mentioned in the previous paragraph) and the contour function of the dual spanning tree $T^*$ of the dual map $M^*$.  There is also an infinite-volume version of Sheffield's bijection which is a.s.\ well defined for infinite-volume FK planar maps. See~\cite{chen-fk} for a detailed exposition of this version of the bijection.

Various strengthenings of Sheffield's scaling limit result (including an analogous scaling limit result for finite-volume FK planar maps) are proven in~\cite{gms-burger-cone,gms-burger-local,gms-burger-finite}. See also~\cite{chen-fk,blr-exponents} for additional results on FK planar maps and~\cite{gwynne-miller-cle} for a scaling limit result in a stronger topology which is proven using the above peanosphere scaling limit results.

In~\cite{kassel-wilson-active}, a new family of probability measures on spanning-trees of (deterministic) rooted planar maps, which generalizes the law arising from the self-dual FK model, was introduced. As explained in that paper, the law on trees $T$ of a rooted map $(M,e_0)$ arising from a self-dual FK model is given by the distribution on all spanning trees of $M$ weighted by~$y^{\act(T)}$, where $y = \sqrt q +1$ and $\act(T) = \act(T,e_0) \in \N$ is the ``embedding activity'' of $T$ (which depends on the choice of root $e_0$; we will remind the reader of the definition later). It also makes sense to consider the probability measure on trees $T$ weighted by~$y^{\act(T)}$ for $y \in (0,1)$, so that trees with a lower embedding activity are more likely. The (unifying) discrete model corresponding to any $y\ge 0$ is called a $y$-active spanning tree.

In the context of the current paper, it is natural to look at a joint law on the triple $(M,e_0,T)$ such that the marginal on $(M,e_0)$ is the measure which weights a rooted planar map by the partition function of active spanning trees.  Indeed, as we explain later, with this choice of law, exploring the tree respects the Markovian structure of the map. We call a random triple sampled from this law a \textit{random rooted active-tree-decorated planar map with parameter $y\geq 0$ and size $n\in\N$}. The limiting case $y = 0$ corresponds to a spanning tree conditioned to have the minimum possible embedding activity, which is equivalent to a bipolar orientation on~$M$ for which the source and sink are adjacent~\cite{bernardi-polynomial} (see~\cite{kmsw-bipolar} for more on random bipolar-oriented planar maps). 

It is conjectured in~\cite{kassel-wilson-active} that for $y \in [0,1)$ the scaling limit of a random spanning tree~$T$ on large subgraphs of a two-dimensional lattice sampled with probability proportional to~$y^{\act(T)}$ is an $\SLE_\kappa$ with $\kappa \in (8,12]$ determined by
\eqb \label{eqn-kappa-y}
\frac{y-1}{2} = -\cos\left(\frac{4\pi}{\kappa} \right) \,.
\eqe
It is therefore natural to expect that the scaling limit of a rooted active-tree-decorated planar map is a $\gamma$-LQG surface decorated by an independent space-filling $\SLE_\kappa$ with $\kappa \in (8,12] $ as in~\eqref{eqn-kappa-y} and $\gamma = 4/\sqrt\kappa$.

We introduce in Section~\ref{sec-burger} a two-parameter family of probability measures on words in an alphabet of 8 symbols which generalizes the hamburger-cheeseburger model of~\cite{shef-burger}.
Under the bijection of~\cite{shef-burger}, each of these models corresponds to a probability measure on spanning-tree-decorated planar maps. One parameter in our model corresponds to the parameter $y$ of the active spanning tree, and the other, which we call~$z$, controls the extent to which the tree $T$ and its corresponding dual tree $T^*$ are ``tangled together''.  This second parameter can also be interpreted in terms of some form of \textit{bending energy\/} of the Peano curve which separates the two trees, in the sense of~\cite{bbg-bending,DiFrancesco}; see Remark~\ref{remark-bending}.  We prove an analogue of~\cite[Thm.~2.5]{shef-burger} for our model which in particular implies that active-tree-decorated planar maps for $0 \leq y < 1$ converge in the scaling limit to $\gamma$-quantum cones decorated by $\SLE_\kappa$ in the peanosphere sense for $\kappa \in (8,12]$ as in~\eqref{eqn-kappa-y} and $\gamma = 4/\sqrt\kappa$. If we also vary~$z$, the other parameter of our model, we obtain tree-decorated random planar maps which converge in the peanosphere sense to $4/\sqrt\kappa$-quantum cones decorated by space-filling $\SLE_\kappa$ for any value of $\kappa > 8$.

\begin{remark} \label{remark-bipolar}
 When $y = 0$, an active-tree-decorated planar map is equivalent to a uniformly random bipolar-oriented planar map~\cite{bernardi-sandpile}.
  In~\cite{kmsw-bipolar}, the authors use a bijective encoding of bipolar-oriented planar maps, which is not equivalent to the one used in this paper, to show that random bipolar-oriented planar maps with certain face degree distributions converge in the peanosphere sense to an $\SLE_{12}$-decorated $\sqrt{4/3}$-LQG surface, both in the finite-volume and infinite-volume cases (see also~\cite{ghs-bipolar} for a stronger convergence result). In the special case when $z = 1$, our Theorem~\ref{thm-all-S} implies convergence of infinite-volume uniform bipolar-oriented planar maps in the peanosphere sense, but with respect to a different encoding of the map than the one used in~\cite{kmsw-bipolar}.
More precisely, bipolar-oriented maps are encoded in~\cite{kmsw-bipolar} by a random walk in $\Z^2$ with a certain step distribution.  The encoding of bipolar-oriented maps by the generalized hamburger-cheeseburger bijection corresponds to a random walk in $\Z^2 \times \{0,1\}$ with a certain step distribution. Both of these walks converge in law to a correlated Brownian motion (ignoring the extra bit in the hamburger-cheeseburger bijection), and the correlations are the same, so we say that they both converge in the peanosphere sense. 
\end{remark}

\subsection{Basic notation}
\label{sec-basic}

We write $\N$ for the set of positive integers.
\vspace{6pt}

\noindent
For $a < b \in \R$, we define the discrete intervals $[a,b]_\Z \colonequals [a, b]\cap \Z$ and $(a,b)_\Z \colonequals (a,b)\cap \Z$.
\vspace{6pt}

\noindent
If $a$ and $b$ are two quantities, we write $a\preceq b$ (resp.\ $a \succeq b$) if there is a constant $C$ (independent of the parameters of interest) such that $a \leq C b$ (resp.\ $a \geq C b$). We write $a \asymp b$ if $a\preceq b$ and $a \succeq b$.

\subsection{Generalized burger model}
\label{sec-burger}

We now describe the family of words of interest to us in this paper. These are (finite or infinite) words which we read from left to right and which consist of letters representing burgers and orders which are matched to one another following certain rules. Several basic properties of this model are proved in Appendix~\ref{sec-prelim}.
Let
\begin{equation}
\Theta_0 \colonequals \left\{\hb,\cb,\ho,\co\right\},
\end{equation}
and let $\mcl W(\Theta_0)$ be the set of all finite words consisting of elements of $\Theta_0$.
The alphabet $\Theta_0$ generates a semigroup whose elements are words in $\mcl W(\Theta_0)$ modulo the relations
\begin{equation} \label{eqn-theta-relations}
\begin{split}
&\cb \co = \hb \ho = \emptyset \quad\quad\quad\text{(order fulfillment)} \\
&\cb \ho = \ho \cb,\quad \hb \co = \co \hb  .
\end{split}
\end{equation} 
Following Sheffield~\cite{shef-burger}, we think of $\hb,\cb,\ho,\co$ as representing a hamburger, a cheeseburger, a hamburger order, and a cheeseburger order, respectively. A hamburger order is fulfilled by the freshest available hamburger (i.e., the rightmost hamburger which has not already fulfilled an order) and similarly for cheeseburger orders.
We say that an order and a burger which cancel out via the first relation of~\eqref{eqn-theta-relations} have been \textit{matched}, and that the order has \textit{consumed\/} the burger. See Fig.~\ref{fig:word}~(a) for a diagram representing matchings in an example. 

We enlarge the alphabet by defining
\begin{equation}
\Theta \colonequals \Theta_0 \cup \left\{ \db, \eb, \fo, \so \right\},
\end{equation}
and let $\mcl W(\Theta)$ be the set of all finite words consisting of elements of $\Theta$. The alphabet $\Theta$ generates a semigroup whose elements are finite words consisting of elements of $\Theta$ modulo the relations~\eqref{eqn-theta-relations} and the additional relations
\begin{equation} \label{eqn-theta-relations'}
\begin{aligned}
 \hb \fo & = \hb \ho = \emptyset & \quad \cb \fo & = \cb \co = \emptyset \\
\hb \so &= \hb \co& \cb \so &= \cb \ho \\
\hb \db &= \hb \hb& \cb \db &= \cb \cb \\
\hb \eb &= \hb \cb& \cb \eb &= \cb \hb.
\end{aligned}
\end{equation}
In the language of burgers, the symbol $\fo$ represents a ``flexible order'' which requests the freshest available burger. The symbol $\so$ represents a ``stale order'' which requests the freshest available burger of the type \textit{opposite\/} the freshest available burger. The symbol $\db$ represents a ``duplicate burger'' which acts like a burger of the same type as the freshest available burger. The symbol $\eb$ represents an ``opposite burger'' which acts like a burger of the type opposite the freshest available burger.  The model of~\cite{shef-burger} includes the flexible order~$\fo$ but no other elements of $\Theta \setminus \Theta_0$.

If a symbol in $\left\{ \db, \eb, \fo, \so \right\}$ has been replaced by a symbol in $\Theta_0$ via one of the relations in~\eqref{eqn-theta-relations'}, we say that this symbol is \textit{identified by\/} the earlier symbol in the relation; and \textit{identified as\/} the symbol in $\Theta_0$ with which it has been replaced. 

Given a word $x = x_1 \cdots x_n \in \mcl W(\Theta)$, we write $|x| = n$ for the number of symbols in~$x$.

\begin{defn}\hypertop{def-reduce} \label{def-reduce}
A word in $\mcl W(\Theta)$ is called \textit{reduced\/} if all of its orders, $\db$'s, and $\eb$'s lie to the left of all of its $\hb$'s and $\cb$'s.  
In Lemma~\ref{prop-reduction} we show that for any finite word $x$, there is a unique reduced word which can be obtained from $x$ by applying the relations~\eqref{eqn-theta-relations} and~\eqref{eqn-theta-relations'}, which we call the \text{reduction\/} of $x$, and denote by $\cR(x)$.
\end{defn} 

An important property of the reduction operation (proved in Lemma~\ref{prop-associative}) is  
\[\cR(xy) =\cR(\cR(x) \cR(y)).\]
Note that for any $x\in \mcl W(\Theta)$, we have $|\cR(x)|\le |x|$.

\begin{defn}\hypertop{def-identification} \label{def-identification}
We write $x'=\cI(x)$ (the \textit{identification\/} of $x$) for the word with $|x'| = |x|$ obtained from $x$ as follows.
For each $i\in \{1,\ldots,|x|\}$,
if $x_i \in \Theta_0$, we set $x_i' = x_i$.
If $x_i \in \{\fo, \so\}$ and $x_i$  
is replaced by a hamburger order (resp.\ cheeseburger order) via~\eqref{eqn-theta-relations'} when we pass to the reduced word $\cR(x)$, we set $x_i' = \ho$ (resp.\ $x_i' = \co$).
If $x_i \in \{\db, \eb\}$ and $x_i$ is replaced with a hamburger (resp.\ cheeseburger) via~\eqref{eqn-theta-relations'} when we pass to the reduced word, we set $x_i' = \hb$ (resp.\ $x_i' = \cb$).  Otherwise, we set $x_i'=x_i$.
We say that a symbol $x_i$ is \textit{identified in the word $x$\/} if $x_i'$ is an element of $\Theta_0$, and \textit{unidentified in the word $x$\/} otherwise.
\end{defn}
 
For example,
\begin{equation*}  
\begin{aligned}
\cR\left(\cb \fo \db   \hb \so   \right) &=  \db \co  \hb  \\
\cI\left( \cb \fo \db   \hb \so  \right) &= \cb \co \db \hb \co .
\end{aligned}
\end{equation*}

Note that $\cR(\cI(x)) = \cR(x)$. 
Note also that any symbol $x_i$ which has a match when we pass to $\cR(x)$ is necessarily identified, but identified symbols are not necessarily matched.  Indeed, symbols in
$\Theta_0$ are always identified, and there may be $\so$, $\db$, and/or $\eb$ symbols in $x$ which are identified, but do not have a match.

\begin{defn} \label{def-theta-count} \hypertop{def-theta-count}
For $\theta\in \Theta$ and a finite word $x$ consisting of elements of $\Theta$, we write
\begin{align*}
\cN_{\theta}(x) &\colonequals \text{number of $\theta$-symbols in $x$}\\
\cN_{\theta_1|\cdots|\theta_k}(x) &\colonequals \cN_{\theta_1}(x)+\cdots+\cN_{\theta_k}(x)\\
\intertext{We also define}
\cB(x)&\colonequals \cN_{\hb|\cb|\db|\eb}(x)=\text{number of burgers in $x$}\\
\cO(x)&\colonequals \cN_{\ho|\co|\fo|\so}(x)=\text{number of orders in $x$}\\
\cC(x)&\colonequals \cB(x)-\cO(x)
\intertext{and}
\dpt(x)  &\colonequals \cN_{\hb}(x) - \cN_{\ho}(x) \\
\ddt(x) &\colonequals \cN_{\cb}(x) - \cN_{\co}(x)\\
\dd(x) &\colonequals \left(\dpt(x),\, \ddt(x)\right) \\
\cQ(x) &\colonequals \dpt(x)-\ddt(x)\,.
\end{align*}
\end{defn}

The reason for the notation $\dpt$ and $\ddt$ is that these quantities represent distances to the root edge in the primal and dual trees, respectively, in the construction of~\cite[\S~4.1]{shef-burger} (see the discussion just below). Note that these quantities are still defined even if $x$ has some symbols in $\left\{ \db, \eb, \fo, \so \right\}$.

Fig.~\ref{fig:word}~(b) shows a random-walk representation of $\dd$ computed on increasing prefixes of a finite (identified) word. This process will later be our main object of study.

If $x$ is a finite word consisting of elements of $\Theta$ with $\cR(x) = \emptyset$, then the bijection described in~\cite[\S~4.1]{shef-burger} applied to $\cI(x)$ uniquely determines a rooted spanning-tree-decorated map $(M, e_0, T)$ associated with $x$.

\begin{figure}[b!]
\captionsetup[subfigure]{position=below,justification=justified,singlelinecheck=false,labelfont=bf}
\begin{subfigure}{.48\textwidth}
\centering
\includegraphics[width=.9\textwidth]{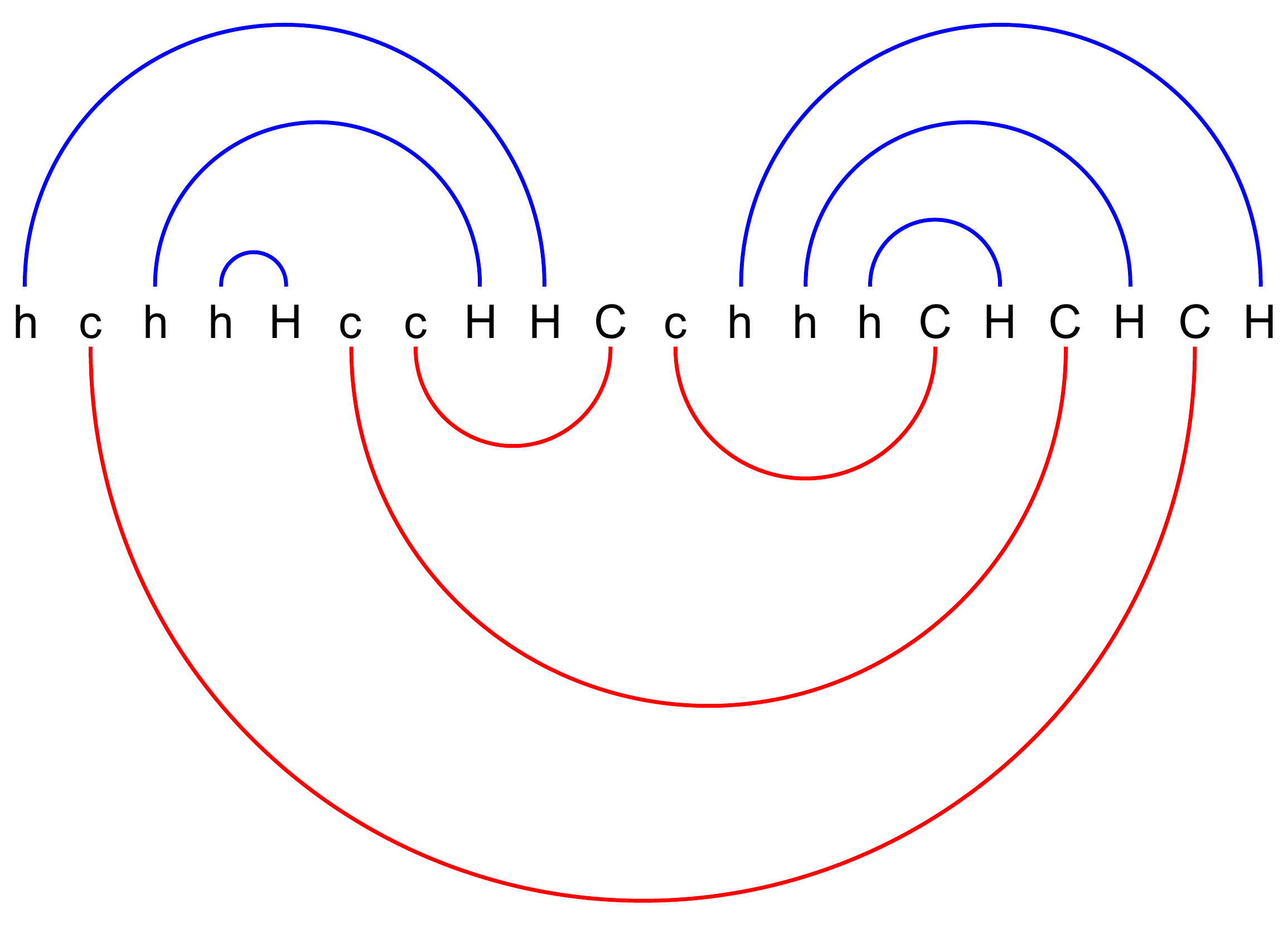}
\caption{The word associated to the decorated map of Fig.~\ref{fig:map}. The chords represent the matchings between orders and burgers that fulfill them.}
\end{subfigure}
\hfill
\begin{subfigure}{.48\textwidth}
\centering
\includegraphics[width=.68\textwidth]{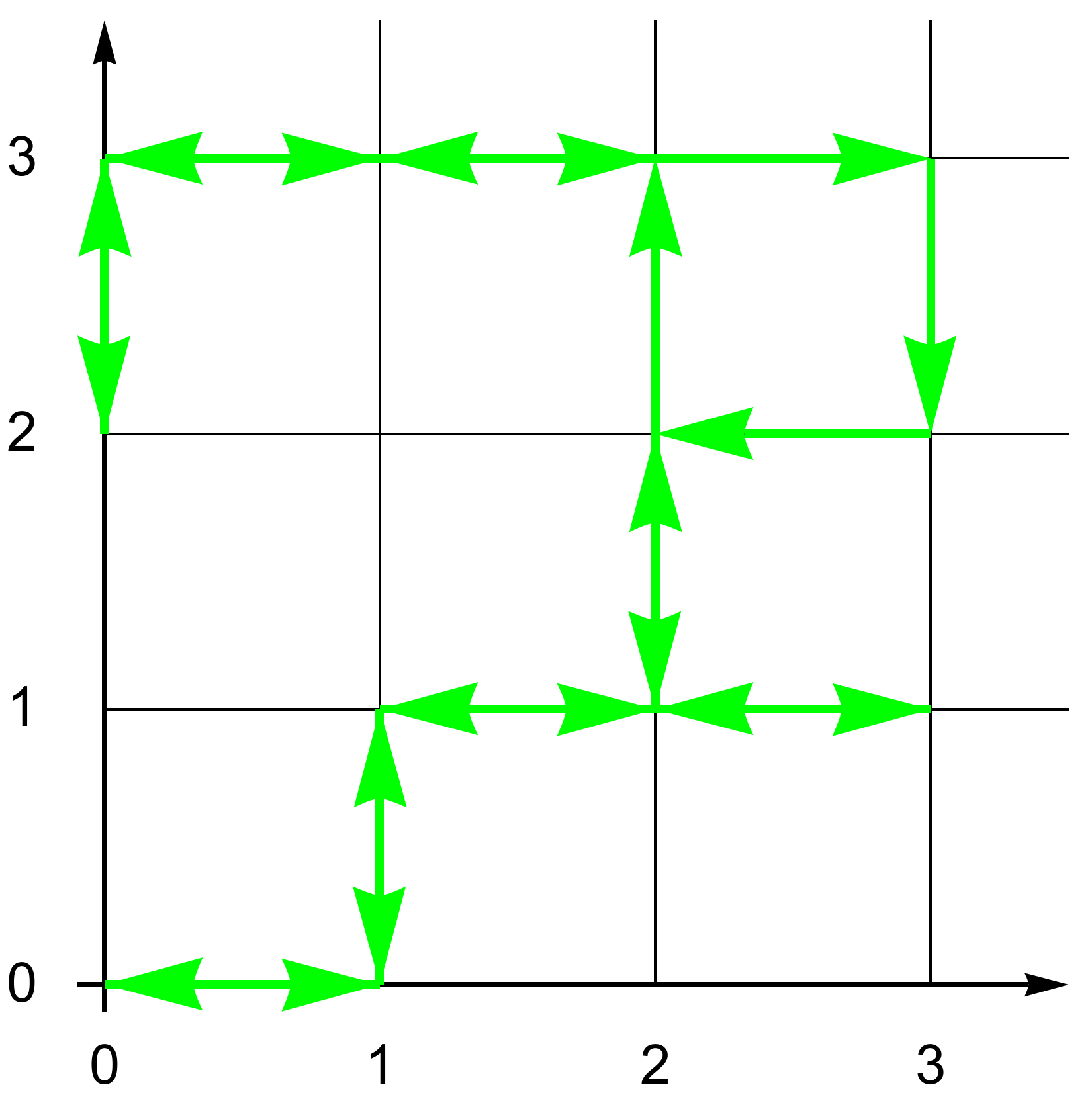}
\caption{The trace of the walk $(\dd_i)_{0\le i \le |x|}$ corresponding to the $\dd$ vector of increasing prefixes of the word $x=\hb\cb\hb\hb\ho\cb\cb\ho\ho\co\cb\hb\hb\hb\co\ho\co\ho\co\ho$.  The walk gives the number of available hamburgers and cheeseburgers as a function of time.}
\end{subfigure}
\caption{}\label{fig:word}
\end{figure}

We now describe the probability measure on words which gives rise to the law on spanning-tree-decorated planar maps which we are interested in.
Let
\alb
\mcl P \colonequals \left\{ (p_\fo, p_\so, p_\db, p_\eb ) \in [0,1]^4 \,:\, p_\fo + p_\so \leq 1 \quad \op{and} \quad p_\db + p_\eb < 1 \right\}.
\ale
For a vector $\pvec = (p_\fo, p_\so, p_\db,p_\eb ) \in \mcl P$, we define a probability measure $\PP = \PP_{\pvec}$ on $\Theta$ by
\begin{equation} \label{eqn-theta-prob}
\begin{aligned}
\PP\!\left(\fo\right) &= \frac{p_\fo}{2}, & \PP\!\left(\so\right) &= \frac{p_\so}{2}, & \PP\!\left(\ho \right) = \PP\!\left(\co\right) &= \frac{1-p_\fo-p_\so}{4} \\
\PP\!\left(\db\right) &= \frac{p_\db}{2}, & \PP\!\left(\eb \right) &= \frac{p_\eb}{2}, & \PP\!\left(\hb \right) = \PP\!\left(\cb\right) &= \frac{1-p_\db - p_\eb}{4}.
\end{aligned}
\end{equation}
Let $X= \cdots X_{-1} X_0 X_1 \cdots$ be a bi-infinite word whose symbols are i.i.d.\ samples from the probability measure~\eqref{eqn-theta-prob}. 
The identification procedure extends naturally to bi-infinite words, and we
show in Appendix~\ref{sec-prelim} that a.s.\ the bi-infinite identified word $X' = \cI(X)$ exists and contains only elements of $\Theta_0$.
 Furthermore, a.s.\ each order in $X$ consumes a burger and each burger in $X$ is consumed by an order. That is, each symbol $X_i$ in $X$ has a match~$X_{\phi(i)}$ which cancels it out, so that in effect the  reduced bi-infinite word $\cR(X)$ is a.s.\ empty.

\begin{defn} \label{def-X-identification}
We write $X' = \cdots X_{-1}' X_0' X_1' \cdots$ for the identification of the bi-infinite word~$X$.
\end{defn}

\begin{defn} \label{def-match}
For $i\in \Z$, we write $\phi(i) \in \Z$ for the index of the symbol matched to $X_i$ in the word $X$.
(From the above property, a.s.\ $\phi$ is an involution of $\Z$.)
\end{defn}

For $a < b \in \R$, we write
\eqb \label{eqn-X(a,b)}
X(a,b) \colonequals \cR(X_{\lfloor a \rfloor} \cdots X_{\lfloor b \rfloor}) \quad \op{and} \quad X'(a,b) \colonequals \cR(X_{\lfloor a \rfloor}' \cdots X_{\lfloor b \rfloor}').
\eqe

The aforementioned results of Appendix~\ref{sec-prelim} allow us to use the infinite-volume version of Sheffield's bijection~\cite{shef-burger} (which is described in full detail in~\cite{chen-fk}) to construct an infinite-volume rooted spanning-tree-decorated planar map $(M^\infty, e_0, T^\infty)$ from the identified word $X'$ of Definition~\ref{def-X-identification}. 

The set $\mcl P$ describes a four-parameter family of probability measures on $\Theta$, and hence a four-parameter family of probability measures on triples $(M^\infty, e_0, T^\infty )$. However, as we will see in Corollary~\ref{prop-identification-law} below, the law of $X'$ (and hence also the law of $(M^\infty, e_0,  T^\infty )$) depends only on the two parameters $p_\fo - p_\so$ and $p_\db - p_\eb$ (equivalently the parameters $y$ and $z$ defined in~\eqref{eqn-y-z}).

\begin{remark} \label{remark-generality}
The model described above includes three special symbols which are natural generalizations of the special order $\fo$ included in~\cite{shef-burger}: the order $\so$ has the opposite behavior as the order $\fo$, and the burgers $\db$ and $\eb$ behave in the same way as $\so$ and $\fo$ but with burgers in place of orders. As we will see in Section~\ref{sec-peano-model}, each of these symbols has a natural topological interpretation in terms of the spanning-tree-decorated rooted planar maps encoded by words consisting of elements of $\Theta$.
\end{remark}
 
\begin{remark} \label{remark-difference}
As we will see, the words we consider in this paper can behave in very different ways from the words considered in~\cite{shef-burger}, which do not include the symbols $\so,\db,$ or $\eb$. For example, in the setting of Section~\ref{sec-variable-SD}, where we allow $\so$'s and $\db$'s but not $\fo$'s or $\eb$'s, the net hamburger/cheeseburger counts $\dpt(X(1,n))$ and $\ddt(X(1,n)) $ in a reduced word tend to be \emph{negatively} correlated (Theorem~\ref{thm-variable-SD}) and the reduced word $X(1,n)$ tends to have \emph{more} symbols than the corresponding reduced word in the case when $p_\fo = p_\so =p_\db = p_\eb = 0$ (Lemma~\ref{prop-mean-mono}). The opposite is true in the setting of~\cite{shef-burger}. As another example, in the setting of Section~\ref{sec-variable-SD} we expect,  but do not prove, that the infinite reduced word $X(1,\infty)$ a.s.\ contains only finitely many unidentified $\so$'s and $\db$'s, whereas $X(1,\infty)$ a.s.\ contains infinitely many unidentified $\fo$'s in the setting of~\cite{shef-burger}  (Remark~\ref{remark-I-infinite}).
\end{remark} 

\subsection{Active spanning trees with bending energy}
\label{sec-peano-model}

Let $(M, e_0)$ be a (deterministic) planar map with $n$ edges with oriented root edge $e_0$. Let $M^*$ be the dual map of $M$ and let $(Q, \BB e_0)$ be the associated rooted quadrangulation (as described at the beginning of the introduction). In this subsection we introduce a probability measure on spanning trees of $M$ which is encoded by the model of Section~\ref{sec-burger}. 

There is a bijection between spanning trees on $M$ and \textit{noncrossing Eulerian cycles\/}
on the \textit{medial graph\/} of $M$, which is the planar dual graph of $Q$.
(An Eulerian cycle is a cycle which traverses each edge exactly once, vertices may be repeated.)
To describe this bijection, let $\lambda$ be a noncrossing Eulerian cycle on the dual of $Q$ starting and ending at $\BB e_0$. By identifying an edge of $Q^*$ with the edge of $Q$ which crosses it, we view $\lambda$ as a function from $[1,2n]_\Z$ to the edge set of $Q$. Each quadrilateral of $Q$ is bisected by one edge of $M$ and one edge of $M^*$, and $\lambda$ crosses each such quadrilateral exactly twice (one such quadrilateral is shown in gray in Figure~\ref{fig:active-sketch}). Hence $\lambda$ crosses each edge of $M$ and each edge of $M^*$ either 0 or 2 times. The set $T$ of edges of $M$ which are not crossed by $\lambda$ is a spanning tree of $M$ whose discrete Peano curve is $\lambda$ and the set $T^*$ of edges of $M^*$ not crossed by $\lambda$ is the corresponding dual spanning tree of $M^*$. Each quadrilateral of $Q$ is bisected by an edge of either $T$ or $T^*$ (but not both). This establishes a one-to-one correspondence between noncrossing Eulerian cycles on the dual of $Q$ starting and ending at $\BB e_0$ and spanning trees of $M$.

Now fix a noncrossing Eulerian cycle $\lambda$ as above. For $i\in [1,2n]_\Z$ we let $\ol e_i$ be the edge of $T \cup T^*$ which bisects the last quadrilateral of $Q$ crossed by $\lambda$ exactly once at or before time $i$,
if such a quadrilateral exists.
Let $e$ be an edge of $T\cup T^*$, and let $j,k\in[1,2n]_\Z$ be the first and second times respectively that $\lambda$ crosses the quadrilateral of $Q$ bisected by $e$.  Observe that if $e$ and $\ol e_{k-1}$ both belong to $M$ or both belong to $M^*$,
then in fact $e = \ol e_{k-1}$.  In this case,
we say that $e$ is of \textit{active type}; this definition coincides with ``embedding activity'',
as illustrated in Figure~\ref{fig:active-sketch}.
If $\ol e_{j-1}$ exists and $e$ and $\ol e_{j-1}$ either both belong to $M$ or both belong to $M^*$, then
we say that $e$ is of \textit{duplicate type}; duplicate edges are illustrated in Figure~\ref{fig-duplicate}, and Remark~\ref{remark-bending} below discusses their relevance.
Figure~\ref{fig:active-duplicate-map} shows the active and duplicate edges from Figure~\ref{fig:map}.
An edge can be of both active and duplicate type, or of neither active nor duplicate type. 

\begin{figure}[h!]
\centering
\hfill\includegraphics[width=.35\textwidth]{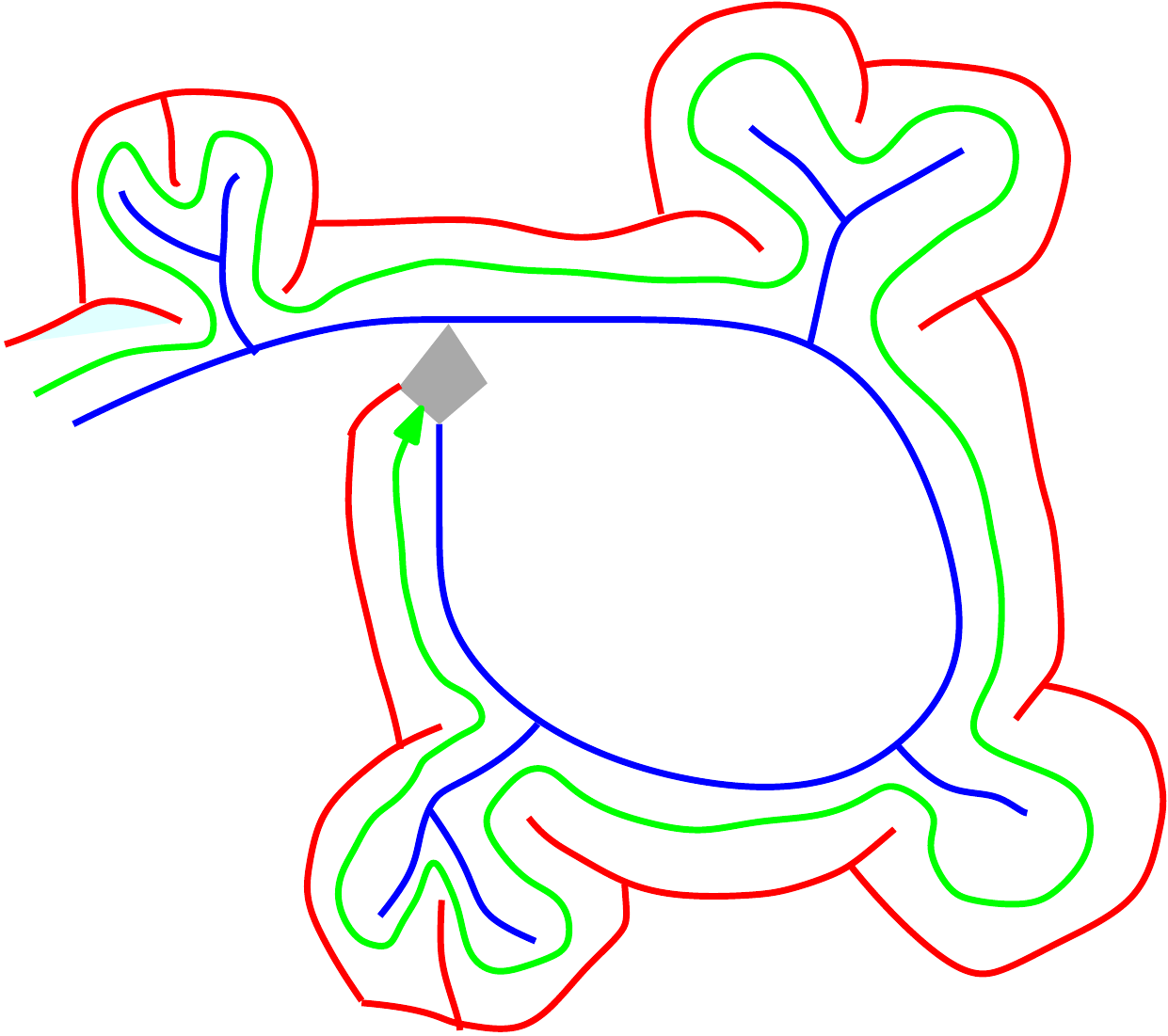}\hfill\raisebox{60pt}{$\rightarrow$}\hfill
\includegraphics[width=.35\textwidth]{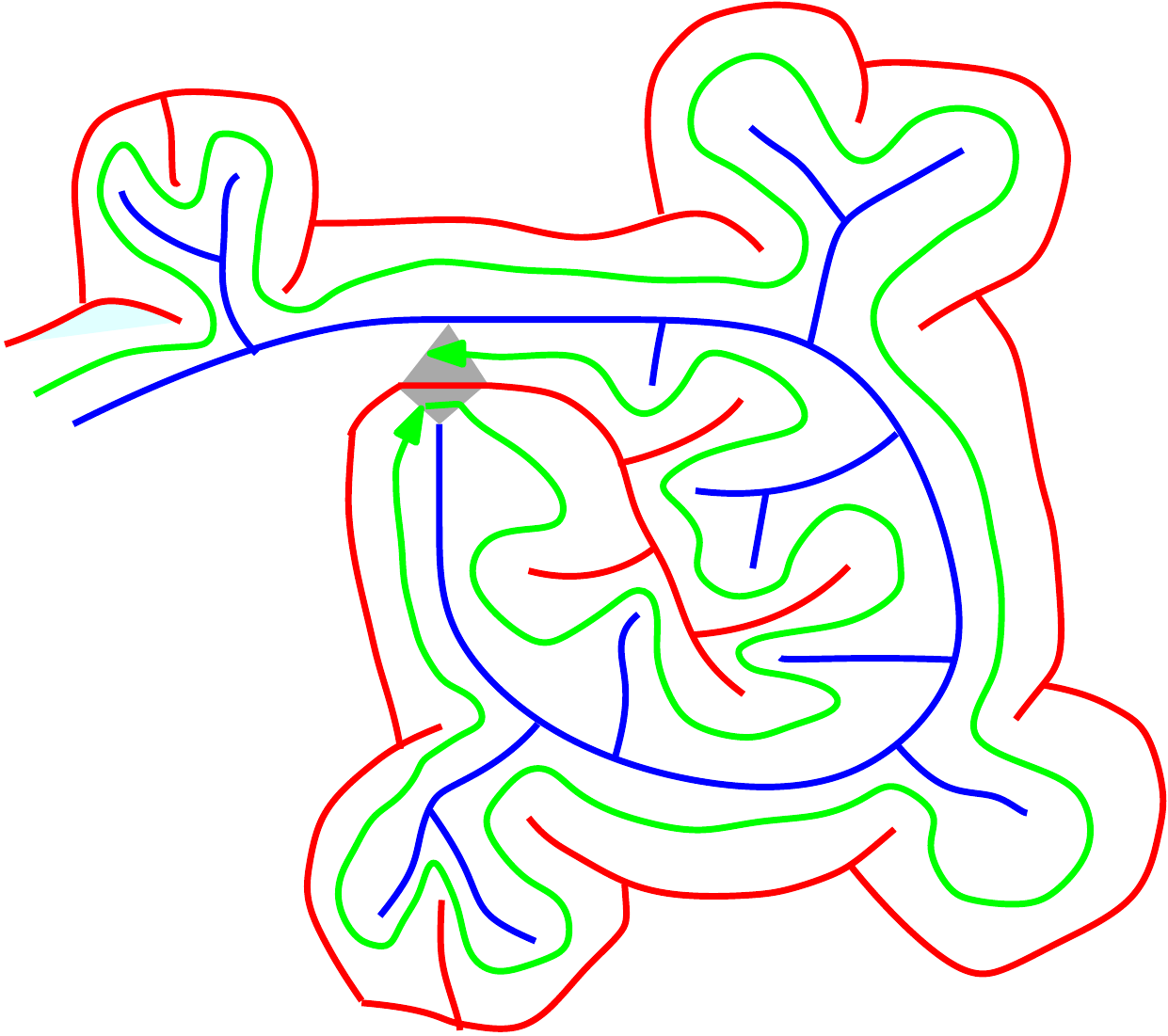}\hfill{}
\caption{ The Peano exploration process with the Peano path $\lambda$ in green,
  primal tree $T$ in blue, and dual tree $T^*$ in red.  When the gray
  quadrilateral is first encountered (left panel), the dual edge $e$ is
  forced to be present (otherwise there would be a primal cycle). 
  This means that $e$ is ``embedding active'', in the sense of~\cite{bernardi-sandpile} (see also~\cite{courtiel-activity}). 
  The Peano curve then explores the map in
  the region enclosed by the blue near-cycle and exits through the
  same (gray) quadrilateral (right panel).  Just before the second time the gray quadrilateral
  is encountered, the most recent quadrilateral encountered exactly once
  is the gray quadrilateral, so $\bar{e}_{k-1}=e$, so $e$ is of active type 
  as defined above.  This characterization of the embedding activity was explained in \cite{shef-burger}.
}
\label{fig:active-sketch}
\end{figure}

Following~\cite{bernardi-sandpile,shef-burger}, a noncrossing Eulerian cycle $\lambda$ based at $\BB e_0$ can be encoded by means of a word $x$ of length $2n$ consisting of elements of $\Theta_0$ with reduced word $\cR(x) = \emptyset$. The symbol $\hb$ (resp.\ $\ho$) corresponds to the first (resp.\ second) time that $\lambda$ crosses an edge of $M$, and the symbol $\cb$ (resp.\ $\co$) corresponds to the first (resp.\ second) time that $\lambda$ crosses an edge of $M^*$. The two times that $\lambda$ crosses a given quadrilateral of $Q$ correspond to a burger and the order which consumes it. With $\ol e_i$ as above, the burger corresponding to the quadrilateral bisected by $\ol e_i$ is the same as the rightmost burger in the reduced word
 $\cR(x_1\cdots x_{i})$;
the edge $\ol e_i$ is undefined if and only if this reduced word is empty.
Therefore edges of active type correspond to orders which consume the most recently added burger that has not yet been consumed, and edges of duplicate type correspond to burgers which are the same type as the the most recently added burger that has not yet been consumed.

\begin{figure}[t!]
\centering
\includegraphics[scale=.8]{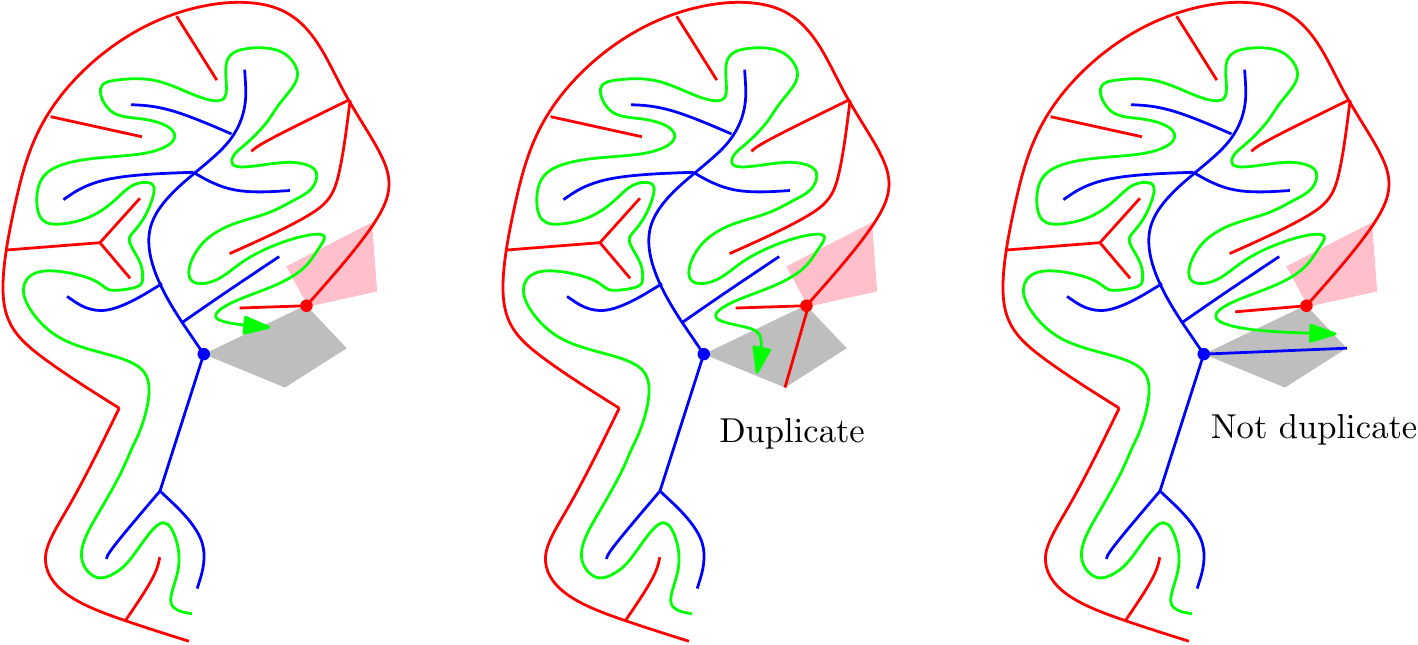}
\caption{Left: the two trees $T$ and $T^*$ and the Peano curve $\lambda$ (in green) run up until step $i-1$. The pink quadrilateral is the most recent one which has been crossed exactly once by $\lambda$ by
time $i-1$ and $\ol e_{i-1}$ is the red edge which bisects this quadrilateral.
The vertices $v_{i-1}^0$ and $v_{i-1}^1$ discussed in Remark~\ref{remark-bending} are shown in red and blue, respectively.
At step $i$, $\lambda$ will either bend away from the red vertex (middle) or toward the red vertex (right). In the former case the edge which bisects the grey quadrilateral belongs to the same tree as $\ol e_{i-1}$, so the edge $\lambda(i)$ is of duplicate type.
 }\label{fig-duplicate}
\end{figure}

\begin{figure}[b!]
\centering
\includegraphics[width=0.55\textwidth]{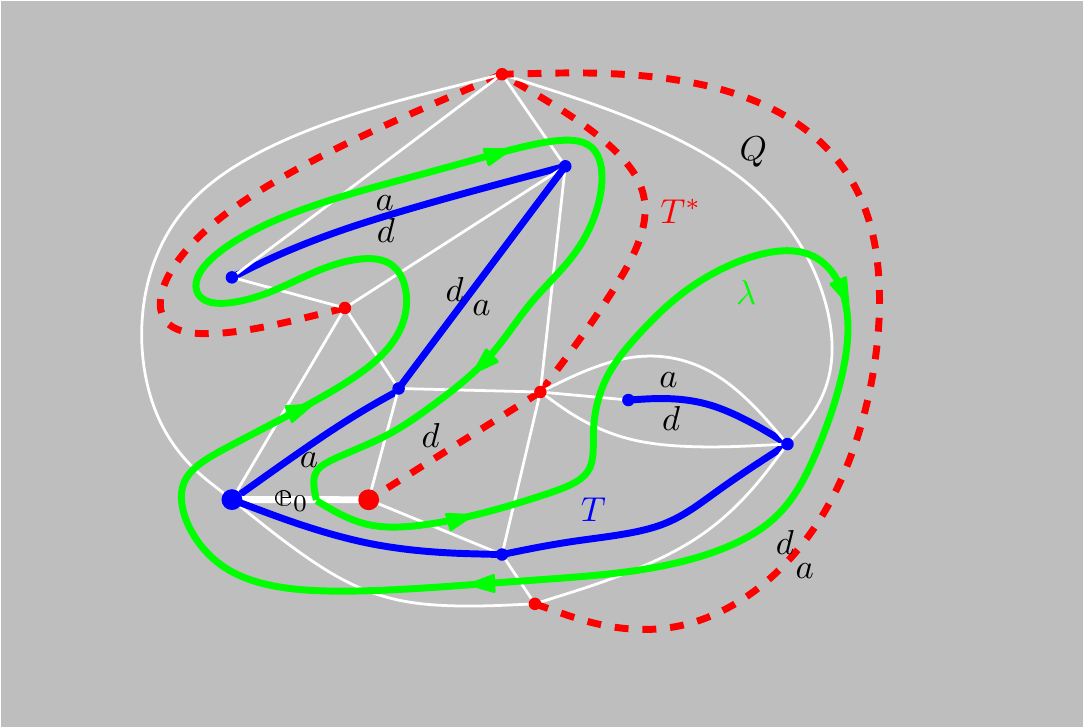}
\caption{The quadrangulation $Q$, the trees $T$ and $T^*$, and the Peano curve $\lambda$ constructed from the triple $(M, e_0, T)$ of Figure~\ref{fig:map} with active (resp.\ duplicate) edges of $T \cup T^*$ indicated with an $a$ (resp.\ a $d$). Edges can be both active and duplicate. The root edge $\BB e_0$ is indicated by a thicker white line. If we allow symbols in~$\Theta$ (rather than just $\Theta_0$), the triple $(M, e_0, T)$ can be encoded by many different words of length~$2n$; more precisely, it can be encoded by any word whose identification is the word shown in Figure~\ref{fig:word}.  The word corresponding to $(M, e_0, T)$ with the smallest possible number of elements of $\Theta_0$ is
$\hb\eb\eb\db\fo\eb\db\so\so\fo\cb\eb\db\db\so\fo\so\fo\so\fo$. In this word, $\fo$ (resp.\ $\so$) symbols correspond to the second time $\lambda$ crosses a quadrilateral of $Q$ bisected by an active (resp.\ inactive) edge and $\db$ (resp.\ $\eb$) symbols correspond to the first time $\lambda$ crosses a quadrilateral of $Q$ bisected by a duplicate (resp.\ non-duplicate) edge. The $\hb$ and $\cb$ symbols correspond to times $i$ for which the edge $\ol e_i$ is not defined.
}\label{fig:active-duplicate-map}
\end{figure}

For a spanning tree $T$ of $M$ rooted at $e_0$, we let $\act(T)$ be the number of active edges and~$\dup(T)$ the number of duplicate edges of its Peano curve $\lambda$.  These quantities depend on the choice of $e_0$. We define the partition function
\begin{equation}
\mathcal{Z}(M,e_0,y,z)=\sum_{\text{spanning tree}\, T}y^{\act(T)}z^{\dup(T)}\,,
\end{equation}
which gives rise, when $y,z\ge 0$, to a probability measure
\begin{equation}\label{eqn-spanning-tree-law}
\PP[T]=\frac{y^{\act(T)}z^{\dup(T)}}{\mathcal{Z}(M,e_0,y,z)}\,,
\end{equation}
on the set of spanning trees $T$ of $M$.
This distribution on spanning trees
satisfies a domain Markov property: for $i\in [1,2n]_\Z$, the conditional law of $\lambda|_{[i+1,2n]_\Z}$ given $\lambda|_{[1,i]_\Z}$ depends only on the set of quadrilaterals and half-quadrilaterals not yet visited by $\lambda$ together with the starting and ending points of the path $\lambda([1,i]_\Z)$. See Figure~\ref{fig:markov} for an illustration of the Markov property of the random decorated map. We call a spanning tree sampled from the above
distribution
an \textit{active spanning tree with bending energy}, for reasons which are explained in the remarks below.

\begin{figure}[t]
\centering
\includegraphics[width=\textwidth/2]{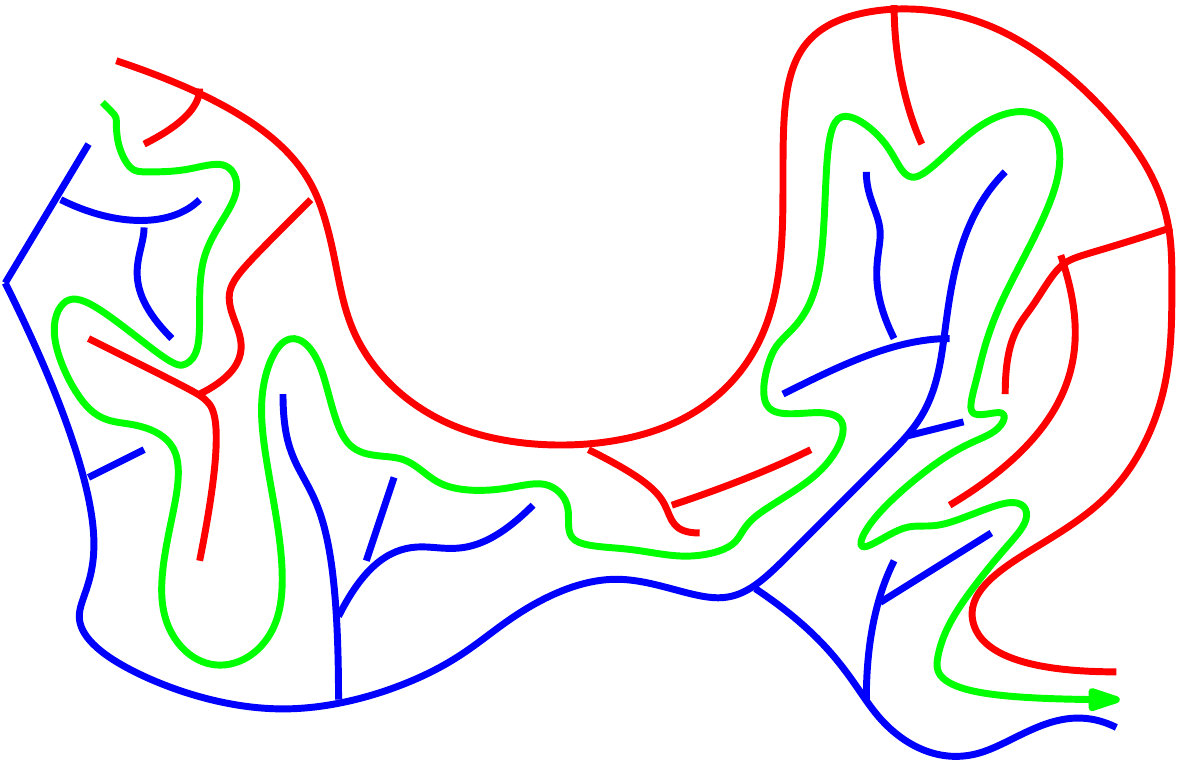}
\caption{Given an initial portion of the exploration process $\lambda$, the set of active and duplicate edges in the remainder of the graph only depends on the initial segment through its boundary.  Consequently,
the law of the decorated random map conditional on the part already drawn only depends on the white region with the boundary components consisting in the red and blue curves only visited on one side by the green curve.  }\label{fig:markov}
\end{figure}

\begin{remark} \label{remark-active-def}
There are other notions of ``active edge'', each of which gives rise to the same Tutte polynomial
\[T_M(x,y) = \sum_{\text{spanning trees $t$ of $M$}} x^{\text{\# internally active edges of $t$}}\,y^{\text{\# externally active edges of $t$}}\,.\]
The embedding activity illustrated in Figure~\ref{fig:active-sketch} differs from Tutte's original definition, but is more natural in this context because it has the domain Markov property, and has a simple characterization in terms of the hamburger-cheesburger model.  The embedding activity is similar to Bernardi's definition \cite[\S~3.1, Def.~3]{bernardi-sandpile}, but with ``maximal'' in place of ``minimal''.
The partition function
$\mathcal{Z}(M,e_0,y,1)=T_M(y,y)$ is the Tutte polynomial of $M$ evaluated at $(y,y)$.  In this case ($z=1$), the partition function is that of the active spanning tree model of~\cite{kassel-wilson-active}, which when $y \geq 1$ coincides with the partition function of the self-dual Fortuin--Kasteleyn (FK) model with parameter $q = (y-1)^2$.
\end{remark}

\begin{remark} \label{remark-bending}
To our knowledge, the notion of edges of duplicate type does not appear elsewhere in the literature. However, this notion can be viewed as a variant of the notion of \textit{bending energies\/} studied in~\cite{bbg-bending} and initially introduced in a different guise in~\cite{DiFrancesco}.
Suppose $(\mcl T, \BB v)$ is a rooted triangulation and $\ell$ is a non-self-crossing  oriented loop in the dual of $\mcl T$, viewed as a cyclically ordered sequence of distinct triangles in $\mcl T$. For each triangle $t$ hit by loop $\ell$, there is a single edge of $t$ which is not shared by the triangles hit by $\ell$ immediately before and after~$t$.
We say that $t$
points
 outward (resp.\ inward) if this edge is on the same (resp.\ opposite) side of the loop $\ell$ as the root vertex $\BB v$. The \textit{bending\/} of $\ell$ is the number of pairs of consecutive triangles which either both point outward or both face inward.
Such a pair of triangles corresponds to a time when loop $\ell$ ``bends around'' a vertex.
If we view the Peano curve $\lambda$ considered above as a loop in the triangulation whose edges are the union of the edges of the quadrangulation $Q$ and the trees $T$ and $T^*$, then the bending of $\lambda$ in the sense of~\cite{bbg-bending} is the number of consecutive pairs of symbols of one of the forms $\hb \hb$, $\ho\ho$, $\hb\ho$, $\ho\hb$, $\cb \cb$, $\co\co$, $\cb\co$, or $\co \cb$ in the identified word which encodes the triple $(M, e_0, T)$ under Sheffield's bijection.

The loops considered in~\cite{bbg-bending} are those arising from variants of the $O(n)$ model, so are expected to be non-space-filling in the limit (in fact they are conjectured to converge to CLE$_\kappa$ loops for $\kappa \in (8/3,8)$~\cite{shef-cle}). For space-filling loops (such as the Peano curve $\lambda$), it is natural to keep track of times when the loop returns to a triangle which shares a vertex with one it hits previously, and then bends toward the set of triangles which it has hit more recently.

Let us now be more precise about what this means. It is easy to see from Sheffield's bijection (and is explained in~\cite[\S~4.2]{chen-fk}) that two edges $\lambda(i)$ and $\lambda(j)$ for $i,j \in [1,2n]_\Z$ share a primal (resp.\ dual) endpoint if and only if the rightmost hamburger (resp.\ cheeseburger) in the reduced words $\cR(x_1\cdots x_i)$ and $\cR(x_1\cdots x_j)$ both correspond to the same burger in the original word $x$, or if these reduced words both have no hamburgers (resp.\ cheeseburgers). Consequently, an edge of duplicate type can be equivalently defined as an edge $\lambda(i)$ such that $\lambda$ crosses a quadrilateral of $Q$ for the first time at time $i$ and the following is true. Let $v_{i-1}^0$ and $v_{i-1}^1$ be the endpoints of $\lambda(i-1)$, enumerated in such a way that $\lambda$ hits an edge which shares the endpoint $v_{i-1}^0$ for the first time before it hits an edge which shares the endpoint $v_{i-1}^1$ for the first time. Then $\lambda$ turns toward $v_{i-1}^1$ at time $i$ (\textit{cf.}\ Figure~\ref{fig-duplicate}).
From this perspective, a time when $\lambda$ crosses a quadrilateral bisected by an edge of duplicate type can be naturally interpreted as a time when $\lambda$ ``bends away from the set of triangles which it has hit more recently''. Hence our model is a probability measure on planar maps decorated by an active spanning tree (in the sense of~\cite{kassel-wilson-active}), weighted by an appropriate notion of the bending of the corresponding Peano curve.
\end{remark}

The generalized burger model of Section~\ref{sec-burger} encodes a random planar map decorated by an active spanning tree with bending energy. The correspondence between the probability vector $\pvec = (p_\fo, p_\so, p_\db, p_\eb) \in \mcl P$ and the pair of parameters $(y,z)$ is given by
\begin{equation} \label{eqn-y-z}
y=\frac{1+p_\fo-p_\so}{1-p_\fo+p_\so}\quad\text{and}\quad z=\frac{1+p_\db-p_\eb}{1-p_\db+p_\eb}\,,
\end{equation}
i.e.
\begin{equation}
p_\fo - p_\so =\frac{y-1}{1+y}\quad\text{and}\quad p_\db-p_\eb=\frac{z-1}{1+z}\,.
\end{equation}
To see why this is the case, let $\dot X$ be a random word of length $2n$ sampled from the conditional law of $X_1 \cdots X_{2n}$ given $\{X(1,2n) = \emptyset\}$, where $X$ is the bi-infinite word from Section~\ref{sec-burger} (in the case when $p_\so = 1$, we allow the last letter of $\dot X$ to be a
flexible order, since a word whose orders are all $\so$'s cannot reduce to the empty word). Let $\dot X' \colonequals \cI(\dot X)$ and let $(M, e_0, T)$ be the rooted spanning-tree-decorated planar map associated with $\dot X'$ under the bijection of~\cite[\S~4.1]{shef-burger}.

\begin{lem} \label{prop-activity}
\begin{enumerate}
\item The law of $(M, e_0, T)$ is that of the uniform measure on edge-rooted, spanning-tree decorated planar maps weighted by $y^{\act(T)} z^{\dup(T)}$, with $y$ and $z$ as in~\eqref{eqn-y-z}. \label{item-activity-law}
\item The conditional law of $T$ given $(M,e_0)$ is given by the law~\eqref{eqn-spanning-tree-law}; and when $z = 1$, the law of $(M, e_0, T)$ is that of an active-tree-decorated planar map (as defined in the introduction). \label{item-activity-cond}
\item If $(M^\infty, e_0^\infty, T^\infty)$ is the infinite-volume rooted spanning-tree-decorated planar map associated with $X$ (by the infinite-volume version of Sheffield's bijection, see the discussion just after~\eqref{eqn-X(a,b)}), then $(M^\infty, e_0^\infty, T^\infty)$ has the law of the Benjamini-Schramm limit~\cite{benjamini-schramm-topology} of the law of $(M,e_0,T)$ as $n\rta\infty$. \label{item-activity-infinite}
\end{enumerate}
\end{lem}
\begin{proof}
Throughout the proof we write $a \propto b$ if $a/b$ is a constant depending only on $n$ and $\pvec$. 
Let $x \in \mcl W(\Theta )$ be a word of length $2n$ which satisfies $\cR(x) = \emptyset$. Note that $x$ must contain $n$ burgers and $n$ orders. Then in the notation of Definition~\ref{def-theta-count},
\begin{multline} \label{eqn-tree-law1}
\PP\!\left(\dot X = x \right) \propto \\
  \left(\frac{2p_\fo}{1-p_\fo-p_\so} \right)^{\cN_{\fo}(x)} \left( \frac{2p_\so}{1-p_\fo-p_\so} \right)^{\cN_{\so}(x)} \left( \frac{2p_\db}{1-p_\db-p_\eb} \right)^{\cN_{\db}(x)} \left(\frac{2p_\eb}{1-p_\db-p_\eb} \right)^{\cN_{\eb}(x)} .
\end{multline} 

Let $A^\ho$ (resp.\ $\wt A^\ho$) be the set of $i\in [1,2n]_\Z$ for which $\dot X_i'$ is a hamburger order matched to a hamburger which is (resp.\ is not) the rightmost burger in $\dot X'(1,i-1)$ (notation as in~\eqref{eqn-X(a,b)}). Let $D^\ho$ (resp.\ $\wt D^\ho$) be the set of $i\in [2,2n]_\Z$ for which $\dot X_i'$ is a hamburger, $\dot X'(1, i-1) \neq\emptyset$, and the rightmost burger in $\dot X'(1,i-1)$ is a hamburger (resp.\ cheeseburger). Define $A^\co$, $\wt A^\co$, $D^\co$, and $\wt D^\co$ similarly but with hamburgers and cheeseburgers interchanged. Then
\eqbn
\act(T) = \# A^\ho + \# A^\co \quad \op{and} \quad \dup(T) = \# D^\ho +\# D^\co.
\eqen
If we condition on $\dot X' $, then we can re-sample $\dot X$ as follows. For each $i \in A^\ho$, independently sample $\dot X_i \in \{\ho, \fo\}$ from the probability measure $\PP( \ho) = (1-p_\fo - p_\so)/(1+ p_\fo -p_\so )$, $\PP(\fo) = 2 p_\fo/(1 + p_\fo-p_\so )$. For each $i \in \wt A^\ho$, independently sample $\dot X_i \in \{\ho, \so\}$ from the probability measure $\PP( \ho) = (1-p_\fo - p_\so)/(1 -p_\fo+p_\so)$, $\PP(\so) = 2 p_\so/(1 -p_\fo+p_\so )$. For each $i \in D^\ho$, independently sample $\dot X_i \in \{\hb, \db\}$ from the probability measure $\PP( \hb) = (1-p_\db - p_\eb)/(1+ p_\db - p_\eb)$, $\PP(\db) = 2 p_\db/(1+p_\db - p_\eb )$. For each $i \in \wt D^\ho$, independently sample $\dot X_i \in \{\hb, \eb\}$ from the probability measure $\PP( \hb) = (1-p_\db - p_\eb)/(1-p_\db +p_\eb)$, $\PP(\eb) = 2 p_\eb/(1- p_\db +p_\eb )$. Then do the same for $A^\co$, $\wt A^\co$, $D^\co$, and $\wt D^\co$ but with hamburgers and cheeseburgers interchanged.  

The above resampling rule implies that with $x$ as above,
\begin{align} \label{eqn-tree-law2}
\PP\!\left(\dot X = x \,|\, \dot X' = \cI(x) \right)
&\propto \left(\frac{1 - p_\fo + p_\so}{1 +p_\fo - p_\so} \right)^{\act(T)} \left(\frac{1 - p_\db + p_\eb}{1 + p_\db - p_\eb} \right)^{\dup(T)} \left(\frac{2p_\fo}{1-p_\fo-p_\so} \right)^{\cN_{\fo}(x)} \notag \\
&\qquad \times \left( \frac{2p_\so}{1-p_\fo-p_\so} \right)^{\cN_{\so}(x)} \left( \frac{2p_\db}{1-p_\db-p_\eb} \right)^{\cN_{\db}(x)} \left(\frac{2p_\eb}{1-p_\db-p_\eb} \right)^{\cN_{\eb}(x)}.
\end{align}
By dividing~\eqref{eqn-tree-law1} by~\eqref{eqn-tree-law2}, we obtain
\eqbn
\PP\!\left( \dot X' = \cI(x) \right) \propto y^{\act(T)} z^{\dup(T)}.
\eqen
Therefore, the probability of any given realization of $(M, e_0, T)$ is proportional to $y^{\act(T)} z^{\dup(T)} $, which gives assertion~\ref{item-activity-law}. Assertion~\ref{item-activity-cond} is an immediate consequence of assertion~\ref{item-activity-law}. Assertion~\ref{item-activity-infinite} follows from the same argument used in~\cite[\S~4.2]{shef-burger} together with the results of Appendix~\ref{sec-prelim}.
\end{proof}

\begin{remark} \label{remark-duality}
The model described in Lemma~\ref{prop-activity} is self dual in the sense that the law of $(M , e_0 , T)$ is the same as the law of $(M^* , e_0^* , T^*)$, where $M^*$ is the dual map of $M$, $e_0^*$ is the edge of $M^*$ which crosses $e_0$, and $T^*$ is the dual spanning tree (consisting of edges of $M^*$ which do not cross edges of $T$). 
This duality corresponds to the fact that the law of the inventory accumulation model of Section~\ref{sec-burger} is invariant under the replacements $\hb \leftrightarrow \cb$ and $\co \leftrightarrow \ho$. It may be possible to treat non-self dual variants of this model in our framework by relaxing the requirement that $\PP(\hb) = \PP(\cb)$ and $\PP(\co) = \PP(\ho)$ in~\eqref{eqn-theta-prob}, but we do not investigate this.  We remark that there are bijections and Brownian motion scaling limit results analogous to the ones in this paper for other random spanning-tree-decorated map models which do not possess this self duality; see, e.g.,~\cite{kmsw-bipolar,lsw-schnyder-wood}. 
\end{remark} 

We end by recording the following corollary of Lemma~\ref{prop-activity}, which says that the law of the identification of the word $X$ (and therefore the law of the associated tree-decorated map) depends on the parmaeter $\pvec$ only via the quantities $y$ and $z$ of~\eqref{eqn-y-z}. 
\begin{cor} \label{prop-identification-law}
Suppose $\pvec = (p_\fo,p_\so,p_\db,p_\eb)$ and $\wt{\PP} = (\wt p_\fo,\wt p_\so, \wt p_\db,\wt p_\eb)$ are two vectors in $\mcl P$ which satisfy $p_\fo-p_\so = \wt p_\fo - \wt p_\so$ and $p_\db - p_\eb = \wt p_\db - \wt p_\eb $. Let $X = \cdots X_{-1} X_0 X_1 \cdots$ (resp.\ $\wt X = \cdots \wt X_{-1} \wt X_0 \wt X_1 \cdots $) be a bi-infinite word such that $\{X_i\}_{i\in\N}$ (resp.\ $\{\wt X_i\}_{i\in\N}$) is a collection of i.i.d.\ samples from the probability measure~\eqref{eqn-theta-prob} with probabilities $\pvec$ (resp.\ with $\wt{\pvec}$).
Then the identifications $\cI(X)$ and $\cI(\wt X)$ agree in law.
\end{cor}
\begin{proof}
It follows from Lemma~\ref{prop-activity} that the infinite-volume tree-decorated planar maps $(M^\infty, e_0^\infty, T^\infty)$ and $(\wt M^\infty, \wt e_0^\infty, \wt T^\infty)$ associated with $\cI(X)$ and $\cI(\wt X)$ agree in law. Since these maps uniquely determine $\cI(X)$ and $\cI(\wt X)$, respectively, via the same deterministic procedure, we infer that $\cI(X) \eqD \cI(\wt X)$.
\end{proof}

\subsection{Statement of main results}
\label{sec-result}

Fix $\pvec = (p_\fo,p_\so,p_\db,p_\eb) \in \mcl P$ and let $X$ be the bi-infinite word from Section~\ref{sec-burger}, whose symbols are i.i.d.\ samples from the probability measure~\ref{eqn-theta-prob}.
Also let $X' = \cdots X_{-1}' X_0' X_1' \cdots = \cI(X)$ be the identification of $X$, as in Definition~\ref{def-X-identification} and recall the notation~\eqref{eqn-X(a,b)}.
For $i \in \Z$, define (in the notation of Definition~\ref{def-theta-count})
\eqbn\hypertop{def-d-z}
\dipt(i) \colonequals \begin{cases}
\dpt( X'(1,i)) \quad &i \geq 1 \\
0 \quad &i = 0 \\
\dpt( X'(i+1,0)) \quad &i \leq -1
\end{cases}
\quad \op{and} \quad
\didt(i) \colonequals \begin{cases}
\ddt(X'(1,i)) \quad &i \geq 1 \\
0 \quad &i = 0 \\
\ddt( X'(i+1,0)) \quad &i \leq -1.
\end{cases}
\eqen
We extend $\dipt$ and $\didt$ to $\R$ by linear interpolation, and define
  $\ddi(t) \colonequals (\dipt(t), \didt(t))$.

For $n\in\N$ and $t\in \R$, let
\eqb \label{eqn-Z^n-def}
U^n(t) \colonequals n^{-1/2} \dipt (n t),\quad V^n(t) \colonequals n^{-1/2} \didt(n  t), \quad Z^n(t) \colonequals (U^n(t), V^n(t) ).
\eqe

It is an immediate consequence of~\cite[Thm.~2.5]{shef-burger} that in the case where $p_\db=p_\eb=p_\so = 0$ and $p_\fo \in (0,1/2)$, the random path $Z^n$ converges in law as $n\rta \infty$ in the topology of uniform convergence on compact intervals to a two-sided two-dimensional correlated Brownian motion $Z = (U,V)$ with $Z(0) =0$ and
\eqb \label{eqn-bm-cov-all-F}
\op{Var}(U(t) ) = \op{Var}(V(t)) = \frac{1}{y+1}|t| \quad \op{and} \quad \op{Cov}(U(t), V(t) ) = \frac{y-1}{2(y+1)} |t|, \quad \forall t\in\R
\eqe
with $y$ as in~\eqref{eqn-y-z}. In the case when $p_\db=p_\eb=p_\so = 0$ and $p_\fo \in [1/2,1]$, the coordinates of~$Z^n$ instead converge in law to two identical two-sided Brownian motions with variance $1/4$.

In light of Corollary~\ref{prop-identification-law} above, the above implies that if $(p_\fo, p_\so, p_\db,p_\eb) \in \mcl P$ with $p_\db = p_\eb = 0$ and $p_\fo - p_\so \geq 0$ (equivalently $y\geq 1$ and $z=1$), then $Z^n$ converges in law as $n\rta\infty$ to a Brownian motion as in~\eqref{eqn-bm-cov-all-F} (resp.\ a pair of identical Brownian motions with variance $1/4$) if $1 \leq y < 3$ (resp.\ $y \geq 3$ and $z=1$).
Our main contribution is to prove that the path $Z^n$ converges to a correlated Brownian motion for additional values of $y$ and $z$.

\begin{thm} \label{thm-all-S}
Let $\pvec = (p_\fo,p_\so,p_\db,p_\eb) \in \mcl P$ with $p_\fo =0$ and $p_\so = 1$ (equivalently, in the notation~\eqref{eqn-y-z}, $y=0$ and $z\ge 0$ is arbitrary). Then with $Z^n$ as in~\eqref{eqn-Z^n-def}, we have $Z^n\rta Z$ in law in the topology of uniform convergence on compacts, where $Z = (U,V)$ is a two-sided correlated Brownian motion with $Z(0) = 0$ and
\eqb \label{eqn-bm-cov-all-S}
\op{Var}( U(t)) = \op{Var}( V(t)) = \frac{(1+z)|t|}{2} \quad \op{and} \quad \op{Cov}( U(t), V(t)) = -\frac{ z |t|}{2 },\quad \forall t \in \R.
\eqe
\end{thm}
We prove Theorem~\ref{thm-all-S} in Section~\ref{sec-all-S}.

\begin{thm} \label{thm-variable-SD}
Let $\pvec =(p_\fo,p_\so,p_\db,p_\eb) \in \mcl P$ with $p_\fo -p_\so \leq 0$ and $p_\eb - p_\db \leq 0$ (equivalently, with $y$ and $z$ as in~\eqref{eqn-y-z}, we have $y \in [0,1]$ and $z\in [1, \infty)$).
There is a parameter $\CHI \in (1,\infty)$, depending only on $y$ and $z$, such that
with $Z^n$ as in~\eqref{eqn-Z^n-def},
$Z^n$ converges in law (in the topology of uniform convergence on compacts) to a two-sided correlated Brownian motion $Z = (U,V)$ with $Z(0) = 0$ and
\begin{equation} \label{eqn-bm-cov-variable-SD}
\begin{aligned}
\op{Var}( U(t)) = \op{Var}( V(t)) &= \frac12 \left( 1 + \frac{(z-y)\CHI }{ (y+1)(z+1)} \right) |t| \quad \op{and} \\
\qquad \op{Cov}( U(t), V(t)) &= - \frac{(z-y)\CHI }{2(y+1)(z+1)}  |t|,\quad\quad \forall t \in \R.
\end{aligned}
\end{equation}
In the case when $z = 1$, we have $\CHI = 2$. When $y=0$, we have $\CHI = z+1$.
\end{thm}

\begin{figure}[htb!]
 \begin{center}
\includegraphics[scale=.7]{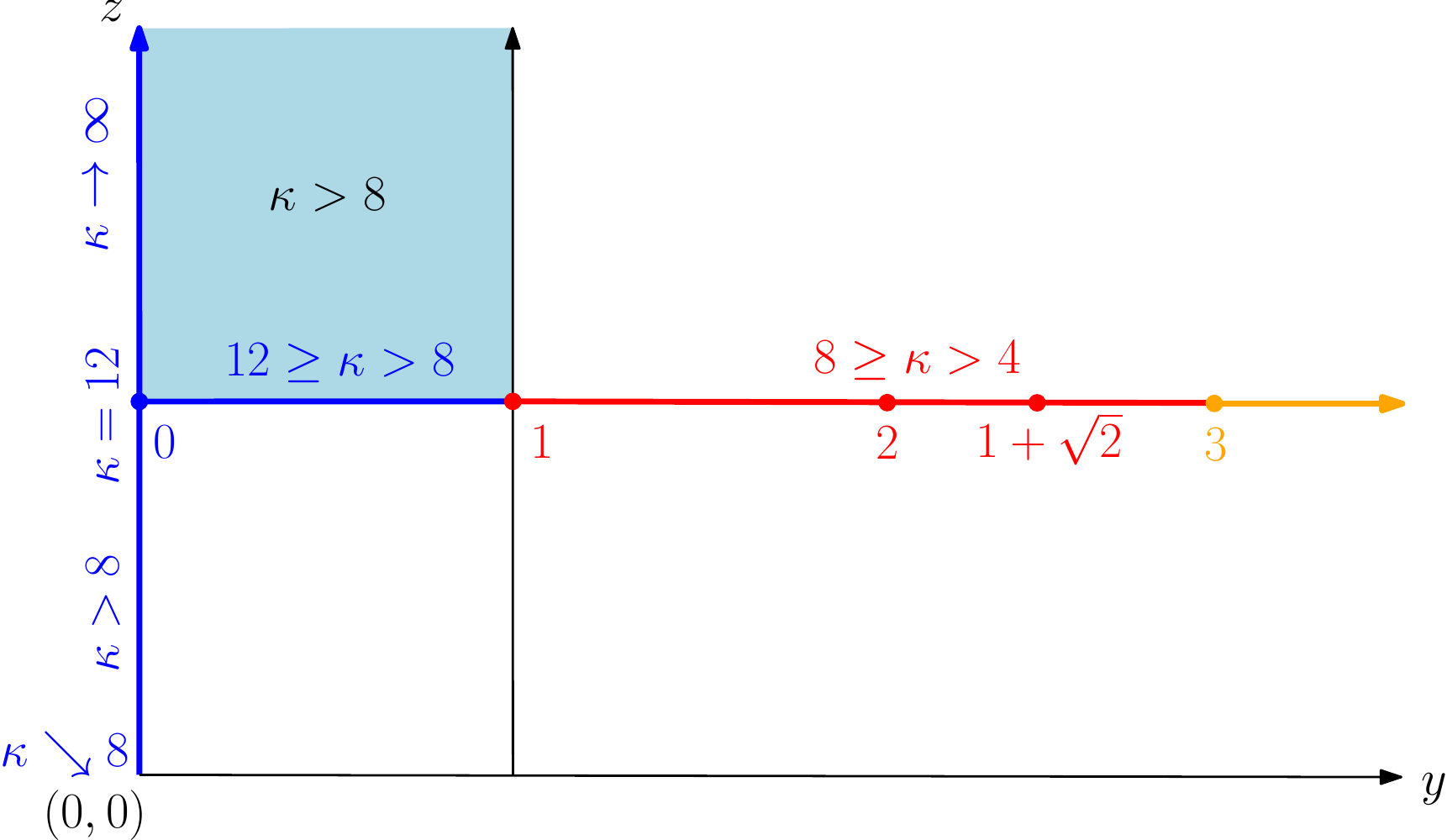}
\caption{A graph of the range of parameter values for which peanosphere scaling limit results for spanning-tree-decorated random planar maps are known, along with the corresponding values of $\kappa$. On the red and orange segments, the path $Z^n$ converges to a non-negatively correlated Brownian motion~\cite{shef-burger}. On the orange segment the correlation is $1$ and the maps are not conjectured to converge to $\SLE_\kappa$-decorated LQG for any $\kappa > 4$. On the red segment, which corresponds to critical FK planar maps for $q\in [0,4)$, peanosphere scaling limit results are known both in the infinite-volume and finite-volume cases~\cite{shef-burger,gms-burger-cone,gms-burger-local,gms-burger-finite}, and several additional results are known~\cite{chen-fk,blr-exponents,gwynne-miller-cle}. The blue and light blue regions are treated in this paper, and give negatively correlated Brownian motions in the scaling limit. The blue line segments are values for which an infinite-volume peanosphere scaling limit result is known and the exact correlation of the limiting Brownian motion (equivalently the limiting values of $\gamma$ and $\kappa$) is known. The horizontal blue segment corresponds to active-tree-decorated planar maps with parameter $y \in [0,1]$ and the vertical segment corresponds to various laws on bipolar-oriented planar maps. The light blue region is the set of parameter values for which the path $Z^n$ is known to converge to a negatively correlated Brownian motion but the exact correlation is unknown. Special parameter values are shown with dots. The case when $(y,z) = (0,1)$ corresponds to a uniform bipolar-oriented random planar map, as studied in~\cite{kmsw-bipolar}. The case when $(y,z) = (1,1)$ corresponds to a random planar map decorated by a uniform spanning tree. The case $(y,z) = (2,1)$ corresponds to the uniform distribution on the underlying planar map $M$, and is the only case where metric scaling limit results are known. The case $(y,z) = (1 + \sqrt 2, 1)$ corresponds to the FK Ising model. }\label{fig-parameter-graph}
\end{center}
\end{figure}

Figure~\ref{fig-parameter-graph} illustrates the range of parameter values for which Theorems~\ref{thm-all-S} and~\ref{thm-variable-SD} (and their analogues elsewhere in the literature) apply. The value of $\CHI$ when $y=0$ follows from Theorem~\ref{thm-all-S}. The value of $\CHI$ when $z=1$ will be obtained in the course of proving Theorem~\ref{thm-variable-SD}.
It remains an open problem to compute $\CHI$ in the case when $z \neq 1$ and $y\neq0$ or to obtain any scaling limit result at all in the case when $z \in [0,1)$ and $y>0$ or when $z \neq1$ and $y \geq 1$.

Theorem~\ref{thm-variable-SD} combined with~\cite[Thm.~1.13]{wedges} and~\cite[Thm.~1.1]{kappa8-cov} tells us that the infinite-volume rooted spanning-tree-decorated random planar map $(M^\infty, e_0^\infty, T^\infty)$ converges in the peanosphere sense, upon rescaling, to a $\gamma$-quantum cone decorated by an independent whole-plane space-filling $\SLE_\kappa$ with $\gamma = 4/\sqrt\kappa$ for some $\kappa \geq 8$. Furthermore, since we know the value of $\CHI$ when $z = 1$, Theorem~\ref{thm-variable-SD} together with~\cite[Thm.~2.5]{shef-burger} and Lemma~\ref{prop-activity} below imply the following.

\begin{cor} \label{prop-active-conv}
Suppose $\pvec \in \mcl P$ is such that $z = 1$ and $y \in [0,3)$. Then $Z^n$ converges in law to a correlated Brownian motion $Z = (U,V)$ with
\eqbn
\op{Var}( U(t)) = \op{Var}( V(t)) = \frac{1}{ y+1 } |t| \quad \op{and} \notag \\
\qquad \op{Cov}( U(t), V(t)) = \frac{y-1}{2(y+1) }  |t|,\quad \forall t \in \R.
\eqen
Hence the scaling limit of an infinite-volume active-tree-decorated planar map with parameter $y\in [0,3)$ in the peanosphere sense is a $\gamma$-quantum cone decorated by an independent whole-plane space-filling $\SLE_\kappa$ with
\eqbn
\frac{y-1}{2} =-\cos\left(\frac{4\pi}{\kappa} \right),\quad \gamma = \frac{4}{\sqrt\kappa},\quad \kappa > 4\,.
\eqen
\end{cor}

\subsection{Outline}
\label{sec-outline}

The remainder of this paper is structured as follows. In Section~\ref{sec-all-S}, we prove Theorem~\ref{thm-all-S}. The key observation in the proof is that if every order in $X$ is an $\so$, then the most recently added burger which has not yet been consumed is the same as the most recently added burger. This allows us to break up the word $X$ into i.i.d.\ blocks of geometric size corresponding to increments of $X$ between the times when the type of the most recently added burger changes. Donsker's theorem applied to the change of $\ddi$ over each of the blocks then concludes the proof.

The proof of Theorem~\ref{thm-variable-SD}, which is given in Section~\ref{sec-variable-SD}, is much more involved than that of Theorem~\ref{thm-all-S}.  Section~\ref{sec-variable-SD} is independent from Section~\ref{sec-all-S}. 

The proof of Theorem~\ref{thm-variable-SD} uses many of the same ideas as the proof of~\cite[Thm.~2.5]{shef-burger}. 
However, the argument used in~\cite{shef-burger} does not suffice for our purposes.
One of the key inputs in the proof~\cite[Thm.~2.5]{shef-burger} is a tail bound for the law of the length of the reduced word $|X(1,n)|$ (see~\cite[Lem.~3.13]{shef-burger}). This tail bound is deduced from the fact that changing a single symbol in the word $X_1\cdots X_n$ changes the value of $\cQ(X(1,n))$, defined as in Definition~\ref{def-theta-count}, by at most 2 (this fact implies that a certain martingale has bounded increments and allows one to apply Azuma's inequality). When we consider words with stale orders and/or duplicate burgers, the above Lipschitz property does not hold.  For example, the reduction of the word
$\hb \hb \hb \cb \so \so \so$
consists of a single $\cb$, but if we change the $\cb$ to an $\hb$, the reduced word has length 7.
We still obtain an analogue of~\cite[Lem.~3.13]{shef-burger} in the setting of Theorem~\ref{thm-variable-SD} (see Proposition~\ref{prop-length-sup} below), but our proof of this result requires analogues of most of the other lemmas in~\cite[\S~3]{shef-burger} as well as some additional estimates.
 
Section~\ref{sec-variable-SD} is structured as follows. 
In Section~\ref{sec-chen-lemma}, we prove a monotonicity result (Lemma~\ref{prop-mean-mono}) which says that for a general choice of $p_\so$ and $p_\db$, the expected number of burgers and the expected number of orders in the reduced word $X(1,n)$ is greater than or equal to the corresponding expectation under the law where $p_\so = p_\db = p_\fo = p_\eb = 0$. Under this latter law, the process $\ddi$ of Definition~\ref{def-theta-count} is a simple random walk on $\Z^2$. 
In fact, this monotonicity holds even if we condition on an event $E$ which depends only on the one-dimensional simple random walk $i\mapsto \cC(X(1,i))$ for $i\in [1,n]_{\Z}$ (Definition~\ref{def-identification}). The proof proceeds by way of a careful analysis of how the length of the reduction of a finite word changes when we replace the rightmost symbol among all of the $\so$ and $\db$ symbols by an element of $\Theta_0$. 

In Section~\ref{sec-few-SD}, we prove a result to the effect that the number of unidentified $\db$'s and $\so$'s in $X(1,n)$ is typically negligible in comparison to the number of unmatched $\ho$'s or $\co$'s (Lemma~\ref{prop-few-SD}). 
Since the $\db$'s and $\so$'s in the reduced word are the only thing which prevents the walk $\ddi$ of Definition~\ref{def-theta-count} from having independent increments, this result tells us that macroscopic increments of $\ddi$ are in some sense ``close'' to being independent. 
This fact will be used frequently in the later subsections.
To prove Lemma~\ref{prop-few-SD}, we use the monotonicity lemma from Section~\ref{sec-chen-lemma} to show that the expected number of unmatched $\ho$'s added to the word between successive times that unidentified $\db$'s and $\so$'s are added is infinite. 

In Section~\ref{sec-J-basic}, we study the time $\J$ which is the smallest $j\in \BB N$ such that $X(-j,-1)$ contains an $\hb$ or $\cb$. The analogue of the time $\J$ also plays a key role in~\cite{shef-burger,gms-burger-cone,gms-burger-local,gms-burger-finite}. 
The importance of $\J$ in our setting is that the burger $X_{-\J}$ determines the identification of the symbol $X_0$. 
We will prove a number of facts about $\J$, the most important of which are Proposition~\ref{prop-J-finite} (which shows that $\CHI := \BB E(|X(-\J,-1)| ) < \infty$) and Lemma~\ref{prop-J-count-mean} (which shows that the expected number of burgers and the expected number of orders in $X(-\J,-1)$ are the same) and Lemma~\ref{prop-J-limit} (a uniform integrability result for $|X(-n,-1)|$ on the event $\{\J> n\}$). 

Section~\ref{sec-var-bound} contains the calculation which leads to the formula for the variances and covariances of the limiting Brownian motions in Theorem~\ref{thm-variable-SD}. This calculation is based on the results of Section~\ref{sec-J-basic} and is similar to~\cite[\S~3.1]{shef-burger}. 

Section~\ref{sec-moment-bound} shows that $\BB E(|X(1,n)|) \asymp n^{1/2}$. The upper bound follows from an analysis of the times at which burgers of a given type are added when we read the word backwards. The upper bound for the number of $\db$'s and $\so$'s in $X(1,n)$ from Lemma~\ref{prop-few-SD} plays an important role in the proof of this estimate since it allows us to avoid worrying about such unmatched symbols. The proof of the corresponding lower bound uses a comparison to a simple random walk on $\Z^2$ based on Lemma~\ref{prop-mean-mono}. 

In Section~\ref{sec-word-length}, we build on the results of Section~\ref{sec-moment-bound} to prove an exponential upper tail bound for $n^{-1/2} |X(1,n)|$ analogous to~\cite[Lem.~3.13]{shef-burger} (Proposition~\ref{prop-length-sup}). 

In Section~\ref{sec-variable-SD-proof}, we use this tail bound to deduce tightness of the law of the re-scaled random walk $Z^n$ in the local uniform topology, then conclude the proof of Theorem~\ref{thm-variable-SD} by using our upper bound for the number of $\db$'s and $\so$'s in $X(1,n)$ to show that any subsequential limiting law must have independent, stationary increments. 

Section~\ref{sec-open-problems} contains some open problems related to the model studied in this paper. Appendix~\ref{sec-prelim} proves some basic facts about the reduction operation $\cR$ and the bi-infinite word $X$.

\bigskip

\noindent \textbf{Acknowledgements.}
We thank the Isaac Newton Institute in Cambridge, UK, where this work was started, for its hospitality.
Part of this work was completed while E.G.\ was an intern with the Microsoft Research Theory group. E.G.\ was partially supported by the U.S.\ Department of Defense via an NDSEG fellowship. When this project was completed, A.K. was supported by ETH Z\"urich and was part of NCCR SwissMAP of the Swiss National Science Foundation. J.M.\ was supported by NSF grant DMS-1204894.
We thank two anonymous referees for helpful comments on an earlier version of this paper.

\section{Scaling limit when all orders are stale}
\label{sec-all-S}

In this section we prove Theorem~\ref{thm-all-S}, which yields the scaling limit of the law of the walk $Z^n$ when all orders are $\so$. Throughout this section we use the notation of Sections~\ref{sec-burger} and~\ref{sec-result} with $p_\fo = 0$ and $p_\so =1$ and to lighten notation, we set
\eqbn
p \colonequals p_\db\quad \op{and}\quad q \colonequals p_\eb.
\eqen
We recall in particular the bi-infinite word $X$ and its identification $X' = \cI(X)$.

The idea of the proof of Theorem~\ref{thm-all-S} is to break up the word $X$ into independent and (almost) identically distributed blocks of random size such that, within each block, the identifications of the symbols $\db$, $\eb$, and $\so$ are determined.  We then apply Donsker's invariance principle to a certain random walk obtained by summing over the blocks.

Let $\iota_0$ be the smallest $i \geq 0$ for which $X_{i } = \hb$. Inductively, if $k\in\N$ and $\iota_{k-1}$ has been defined, let $\iota_k$ be the smallest $i \geq \iota_{k-1}+1$ for which
\eqbn
\begin{dcases}
X_i \in \left\{\eb, \cb\right\} \quad &\text{$k$ odd}, \\
X_i \in \left\{\eb, \hb \right\} \quad &\text{$k$ even}.
\end{dcases}
\eqen
(In other words, the sequence $(\iota_k)_{k\ge 0}$ is the sequence of nonnegative indices which correspond to alternation in the type of burger produced.)

Let
\begin{equation} \label{eqn-inc-def}
\xi_k = \left(\xi_k^\ho,\, \xi_k^\co \right)
\colonequals \begin{dcases}
\left(\cN_{\hb|\db|\eb}\left(X_{\iota_{k-1} } \cdots X_{\iota_k-1 } \right),\, - \cN_{\so}\left(X_{\iota_{k-1} } \cdots X_{\iota_k-1 }\right) \right),\quad &\text{$k$ odd}\\
\left( - \cN_{\so}\left(X_{\iota_{k-1} } \cdots X_{\iota_k-1 }\right),\, \cN_{\cb|\db|\eb}\left(X_{\iota_{k-1} } \cdots X_{\iota_k-1 }\right)\right),\quad &\text{$k$ even}.
\end{dcases}
\end{equation}
(There are no $\eb$ symbols in the subword $X_{\iota_{k-1} } \cdots X_{\iota_k-1 }$ except possibly for $X_{\iota_{k-1}}$.)  Let
\eqb \label{eqn-inc-sum-def}
\Xi_k = (\Xi_k^\ho,\Xi_k^\co) \colonequals \sum_{j=1}^k \xi_j \,.
\eqe

\begin{lem} \label{prop-inc-walk}
In the setting described just above, we have the following.
\begin{enumerate}
\item For each $k\in\N$, we have $\ddi\left( \iota_k-1 \right) - \ddi(\iota_0-1) = \Xi_k$. \label{item-inc-walk-sum}
\item For each odd (resp.\ even) $k \in\N$, we have $\iota_k - \iota_{k-1} = \xi_k^\ho - \xi_k^\co$ (resp.\ $\iota_k - \iota_{k-1} = \xi_k^\co - \xi_k^\ho  $). \label{item-inc-walk-time}\\[-\baselineskip]
\item The random variables $\xi_k$ for $k\in\N$ are independent. \label{item-inc-walk-ind}
\item For each $k\in\N$, the law of $\iota_k - \iota_{k-1}$ is geometric with success probability $(1-p + q)/4$. If $k$ is odd (resp.\ even), then given $\iota_k-\iota_{k-1}$ the symbols of $X_{\iota_{k-1}+1}\cdots X_{\iota_{k}-1}$ are i.i.d., and each is a burger with probability $(1+p-q)/(3+p-q)$. In particular, the conditional law of $\xi_k^\ho -1$ (resp.\ $\xi_k^\co -1$) given $\iota_k - \iota_{k-1}$ is the binomial distribution with parameters $\iota_k - \iota_{k-1}-1$ and $(1+p-q)/(3+p-q)$. \label{item-inc-walk-law}
\end{enumerate}
\end{lem}
\begin{proof}
Since the only orders are of type $\so$, for any $i\in\Z$ the most recently added burger which hasn't yet been consumed is the same as the most recently added burger.  By the definition of the times $\iota_*$, if $\iota_{k-1}\leq i <\iota_k$, then the top burger is of type $\hb$ if $k$ is odd and of type $\cb$ if $k$ is even.  For simplicity we assume throughout the rest of the proof that $k$ is odd; the case when $k$ is even is symmetric.

For $\iota_{k-1}\leq i <\iota_k$, if $X_i$ is a burger, then $X'_i=\hb$, and if $X_i$ is an order, then $X'_i=\co$, which implies $\ddi(\iota_k-1)-\ddi(\iota_{k-1}-1)=\xi_k$. Summing this relation and the analogous relation in the case when $k$ is even gives assertion~\ref{item-inc-walk-sum}.  Since $k$ is assumed to be odd, the total number of burgers and orders in $X_{\iota_{k-1}}\cdots X_{\iota_k}$ is $\xi^\ho_k-\xi^\co_k$, which implies assertion~\ref{item-inc-walk-time}.

Since $\iota_{k-1}$ for $k\in\N$ is a stopping time for the filtration generated by $X_1\cdots X_n$ for $n\in\N$, the
strong Markov property implies
$X_{\iota_{k-1} +1} \cdots X_{\iota_k}$ is independent of
$X_1\cdots X_{\iota_{k-1}}$, which implies
assertion~\ref{item-inc-walk-ind}.

In view of the strong Markov property (and again recalling that $k$ is assumed to be odd), we see that $X_{\iota_{k-1}+1}\cdots X_{\iota_{k}}$ is a string of i.i.d.\ symbols terminated at the first $\cb$ or $\eb$. By~\eqref{eqn-theta-prob}, the terminating symbol occurs with probability $\PP(\eb) + \PP(\cb) = p_\eb/2 +(1-p_\eb-p_\db)/4 =(1-p+q)/4$, which implies the geometric law for $\iota_k-\iota_{k-1}$.  Given the length of the string $X_{\iota_{k-1}+1}\cdots X_{\iota_{k}}$, each symbol except the last is a burger independently with probability
\eqbn
\frac{ \PP(\hb) + \PP(\db)  }{ \PP(\hb) + \PP(\db) + \PP(\so)  } 
= \frac{(1+p-q)/4}{(3+p-q)/4},
\eqen
which finishes proving assertion~\ref{item-inc-walk-law}.
\end{proof}

\begin{prop}
For odd $k\in\N$,
\begin{equation} \label{eqn-inc-walk-var}
\begin{aligned}
\E\left(\iota_k - \iota_{k-1} \right) &= \frac{4}{1-p+q} \,,& \op{Var}\left(\iota_k - \iota_{k-1} \right) &= \frac{4(3+p-q)}{(1-p+q)^2} \,, \\
\E\left(\xi_k^\ho \right) &= \frac{2}{1-p+q}\,, & \op{Var}\left(\xi_k^\ho \right) &= \frac{2 (1 + p - q)}{(1 - p + q)^2} \,,\\
\E\left(\xi_k^\co \right) &= -\frac{2}{1-p+q}\,,\quad& \op{Var}\left( \xi_k^\co \right) &= \frac{2(3-p+q)}{(1-p+q)^2} \,,\\
&&\op{Cov}\left(\xi_k^\ho, \xi_k^\co\right) &= -\frac{2(1+p-q)}{ (1-p+q)^2} \,.
\end{aligned}
\end{equation}
For even $k\in\N$, the same holds with $\xi_k^\ho$ and $\xi_k^\co$ interchanged.
\end{prop}
\begin{proof}
Let $Z_i$ be the indicator random variable for the word $X_{\iota_{k-1}+1}\cdots X_{\iota_k-1}$ having length at least $i$ and having a burger in position~$i$, and let $Z^\so_i$ be the indicator variable for the word having length at least $i$ and having an order in position~$i$. For odd $k$, $\xi_k^\ho=1+\sum_{i=1}^\infty Z_i$ and $\xi_k^\co=-\sum_{i=1}^\infty Z^\so_i$, and vice versa for even $k$. Assertions~\ref{item-inc-walk-time} and~\ref{item-inc-walk-law} of Lemma~\ref{prop-inc-walk} yield $\E[Z_i]$, $\E[Z^\so_i]$, $\E[Z_i Z_j]$, $\E[Z^\so_i Z^\so_j]$, and $\E[Z_i Z^\so_j]$, from which \eqref{eqn-inc-walk-var} follows from a short calculation.
\end{proof}

\begin{proof}[Proof of Theorem~\ref{thm-all-S}]
For $k \in \N \cup \{0\}$ let $\Xi_k^\ho$ and $\Xi_k^\co$ be as in~\eqref{eqn-inc-sum-def}. Extend $\Xi^\ho$ and $\Xi^\co$ from $\N \cup \{0\}$ to $[0,\infty)$ by linear interpolation.
For $t \geq 0$ and $n\in\N$, let
\eqbn
\wh U^n(t) \colonequals n^{-1/2}\, \Xi_{2n t}^\ho, \quad \wh V^n(t) \colonequals n^{-1/2}\, \Xi_{2n t}^\co, \quad \op{and} \quad \wh Z^n(t) \colonequals \left(\wh U^n(t), \wh V^n(t) \right).
\eqen

It follows from~\eqref{eqn-inc-walk-var} that for each $k\in\N$,
\begin{align*}
&\E\left(\xi_{2k-1}^\ho + \xi_{2k}^\ho \right) = \E\left(\xi_{2k-1}^\co + \xi_{2k}^\co \right) = 0 \,,\\
&
\op{Var}\left( \xi_{2k-1}^\ho + \xi_{2k}^\ho\right) = \op{Var}\left( \xi_{2k-1}^\co + \xi_{2k}^\co \right) = \frac{8}{(1-p+q)^2}\,, \\
&\op{Cov}\left( \xi_{2k-1}^\ho + \xi_{2k}^\ho, \xi_{2k-1}^\co + \xi_{2k}^\co\right) = -\frac{4(1+p-q)}{ (1-p+q)^2} \,.
\end{align*}
By Lemma~\ref{prop-inc-walk}, the pairs $(\xi_{2k-1}^\ho + \xi_{2k}^\ho, \xi_{2k-1}^\co + \xi_{2k}^\co)$ for each $k \in \N$ are i.i.d.
By Donsker's invariance principle (see~\cite[Thm.~4.3.5]{whitt-limits-book} for a statement in general dimension), $\wh Z^n$ converges in law as $n\rta \infty$ in the topology of uniform convergence on compacts to a pair $\wh Z = (\wh U, \wh V)$ of correlated Brownian motions with $\wh Z(0) = 0$
\eqb \label{eqn-hat-bm-cov-all-S}
\op{Var}(\wh U(t)) = \op{Var}(\wh V(t)) = \frac{8t }{(1-p+q)^2} \quad \op{and} \quad \op{Cov}(\wh U(t), \wh V(t)) = -\frac{4(1+p-q) t}{ (1-p+q)^2},\quad \forall t \geq 0.
\eqe
By the law of large numbers, a.s.\
\begin{equation}
\label{eqn-time-convergence}
\lim_{k\rta\infty} k^{-1} \iota_{\lfloor t k \rfloor} = \frac{4 t}{1-p+q},\quad \forall t\in \Q.
\end{equation}
By the Skorokhod representation theorem, we can find a coupling of a sequence of words $(X^n)$, each with the same law as $X$, with the correlated Brownian motion $\wh Z$ such that (with $\wh Z^n$ and $\iota_k^n$ defined with respect to the word $X^n$) we a.s.\ have $\wh Z^n \rta \wh Z$ and $k^{-1} \iota_{\lfloor t k \rfloor}^n \rta 4t/(1-p+q)$ for each $t\in\Q$.
Combining~\eqref{eqn-time-convergence} with the fact that each coordinate of $Z^n$ is monotone between subsequent renewal times, and the continuity of Brownian motion, we obtain that
\eqbn
(t \mapsto Z^n(t)) \xrightarrow{n\rta\infty} \left( t \mapsto \wh Z\left(\frac{1-p+q}{8} t \right) \right)
\eqen
in the topology of uniform convergence on compacts of $[0,\infty)$. By~\eqref{eqn-hat-bm-cov-all-S}, $t \mapsto \wh Z\left(\frac{1-p+q}{8} t \right)$ is a Brownian motion with variances and covariances as in~\eqref{eqn-bm-cov-all-S}. We thus obtain $Z^n|_{[0,\infty)} \rta Z|_{[0,\infty)}$ in law, with $Z$ as in the theorem statement.
Since the law of the bi-infinite word $X$ is translation invariant, we also have that $\left(Z^n - Z^n(s_0)\right)|_{[s_0,\infty)} \rta \left(Z - Z (s_0)\right)|_{[s_0,\infty)}$ in law for each $s_0\in\R$. Since $Z^n(0) = Z(0) = 0$ for each $n\in\N$, for $s_0 < 0$ and $t\geq s_0$, 
\eqbn
Z^n(t) = \left( Z^n(t) - Z^n(s_0) \right)  - \left( Z^n(0) - Z^n(s_0) \right)
\eqen
and the analogous relation holds for $Z$. From this we infer that $Z^n|_{[s_0, \infty)} \rta Z|_{[s_0, \infty)}$ in law for each $s_0 \in \R$.
Since $s_0$ can be be made arbitrarily negative, we infer that $Z^n \rta Z$ in law in the topology of uniform convergence on compact subsets of $\BB R$. 
\end{proof}

\section{Scaling limit with stale orders and duplicate burgers}
\label{sec-variable-SD}

In this section we prove Theorem~\ref{thm-variable-SD}.  Since the paths $Z^n$ are deterministic functions of the identified word $X' = \cI(X)$, Corollary~\ref{prop-identification-law} implies that we only need to prove Theorem~\ref{thm-variable-SD} in the case when $p_\fo = p_\eb = 0$.

Throughout this section, we fix $p\in [0,1)$ and $q\in [0,1)$ and let $\PP^{p,q}$ denote the law of the bi-infinite word $X$ whose symbols are i.i.d.\ samples from the law~\eqref{eqn-theta-prob} with $p_\fo = p_\eb = 0$, $p_\so = p$, and $p_\db=q$. Let $\E^{p,q}$ denote the corresponding expectation. When there is no ambiguity (i.e.\ only one pair $(p,q)$ is under consideration) we write $\PP = \PP^{p,q}$ and $\E = \E^{p,q}$.

Since we think of $p_\so$ and $p_\db$ as being fixed, we abuse notation and allow ``constants'' to depend on $p_\so$ and $p_\db$,
including the implicit constants in asymptotic notation.

\subsection{Comparison of expected lengths of reduced words}
\label{sec-chen-lemma}

The following lemma is
one of our main tools for estimating expectations of quantities associated with the word $X$.

\begin{lem} \label{prop-mean-mono}
Suppose we are in the setting described at the beginning of this section.
Let $n\in\N$ and let $E$ be an event which is measurable with respect to the $\sigma$-algebra generated by $\left\{\cC(X(1,i))\right\}_{i \in [1,n]_\Z}$ (where here $\cC$ is as in Definition~\ref{def-theta-count}, i.e., $\cC=\cB-\cO$).
For each $n\in\N$ and $(p,q) \in [0,1] \times [0,1)$, we have (in the notation of Definition~\ref{def-theta-count})
\eqb \label{eqn-mean-mono-B}
\E^{p,q}\left( \cB\left(X(1,n)\right) \one_E \right) \geq \E^{0,0}\left( \cB\left(X(1,n)\right) \one_E \right)
\eqe
and 
\eqb \label{eqn-mean-mono-O}
\E^{p,q}\left( \cO\left(X(1,n)\right) \one_E \right) \geq \E^{0,0}\left( \cO\left(X(1,n)\right) \one_E \right).
\eqe
\end{lem}

The intuitive reason why we expect Lemma~\ref{prop-mean-mono} to be true is that it is ``harder'' for a $\db$ or $\so$ to find a match than it is for an element of $\Theta_0$ to find a match, since the $\db$ or $\so$ has to be identified, then matched. So, replacing $\db$'s and $\so$'s by elements of $\Theta_0$ should tend to reduce the number of burgers and orders in the word. 

To prove the lemma, we will iteratively replace the rightmost symbol amongst all of the $\db$'s and $\so$'s in $X_1\dots X_n$ by an $\hb$ or $\cb$ with equal probability (if it is a $\db$) or by an $\ho$ or $\co$ with equal probability (if it is an $\so$) and argue that each of these replacements reduces the expected number of burgers and orders in $X(1,n)$. The key tool in the proof is Lemma~\ref{prop-word-flip} below.

\begin{remark}
The proof of Lemma~\ref{prop-mean-mono} is based on an argument of Linxiao Chen which appears in the proof of Lemma~5 in the original arXiv version of~\cite{chen-fk}.  Chen's argument does not in fact yield the stochastic domination statement claimed in his Lemma~5, but does prove the analogue of Lemma~\ref{prop-mean-mono} in the setting where $p_\so=p_\db=p_\eb=0$ and $p_\fo \in [0,1]$.
\end{remark}

Define an involution $\theta\mapsto \theta^\dagger$ on $\Theta$ by
\begin{equation} \label{eqn-involution}
\begin{split}
&\hb^\dagger = \cb \quad \cb^\dagger = \hb \quad \ho^\dagger = \co \quad \co^\dagger = \ho \\
&\db^\dagger = \db \quad \eb^\dagger = \eb \quad \fo^\dagger =\fo \quad \so^\dagger =\so.
\end{split}
\end{equation}
For a word $x = x_1 \cdots x_{|x|}$ consisting of elements of $\Theta $, we write $x^\dagger = x_1^\dagger \cdots x_{|x|}^\dagger$.

For such a word $x$, we write $s(x)$ for the index of the rightmost symbol among all of the $\db$ or $\so$ symbols in $x$ (or $s(x) =0$ if no such $x$ exists).
We define
\[
x^\Hc = \begin{cases}
 \text{word obtained from $x$ by replacing $x_{s(x)}$ with $\ho$} & \text{$s(x)>0$ and $x_{s(x)} = \so$}\\
 \text{word obtained from $x$ by replacing $x_{s(x)}$ with $\cb$} & \text{$s(x)>0$ and $x_{s(x)} = \db$}\\
 x &  s(x) = 0\,,
\end{cases}
\]
and we define $x^\Ch$ similarly but with $\co$ and $\hb$ in place of $\ho$ and $\cb$.

We write $r(x)$ for the largest $k \in [1,s(x)-1]_\Z$ for which $x_k=\hb$ or $x_k=\cb$ and $x_k$ has
no match in $x_{1} \cdots x_{s(x)-1}$ (or $r(x) = 0$ if no such $k$ exists).
We define an involution
\eqbn
\Psi(x) = \begin{cases}
\left( x_1 \cdots x_{r(x)-1} \right)^\dagger x_{r(x)} \cdots x_{s(x)} \left( x_{s(x)+1} \cdots x_{|x|} \right)^\dagger & r(x)>0\\
x_1 \cdots x_{s(x)} \left( x_{s(x)+1} \cdots x_{|x|} \right)^\dagger & r(x)=0.
\end{cases}
\eqen

We make the following elementary observations about the above operations.
\begin{enumerate}
\item Involution commutes with reduction, i.e.\ $\cR(x^\dagger) = \cR(x)^\dagger$ for all words $x$.
\item $s(\Psi(x)) = s(x)$ and $r(\Psi(x)) = r(x)$ for all words $x$, and hence $\Psi(\Psi(x)) = x$.
\end{enumerate}

\begin{lem} \label{prop-word-flip}
Let $x = x_1\cdots x_{|x|}$ be a word consisting of elements of $\Theta_0 \cup \left\{\db, \so \right\}$.
If
\eqb \label{eqn-word-flip-hyp}
\cB\left( \cR( x^\Hc )\right)  > \cB( \cR( x) )
\eqe
then
\eqb \label{eqn-word-flip-conc}
\cB\left( \cR( x^\Hc) \right) = \cB( \cR( x ) ) + 1 \quad \op{and} \quad \cB\left( \cR( \Psi(x)^\Hc ) \right) = \cB( \cR( \Psi(x) ) ) - 1.
\eqe
\end{lem}

To prove Lemma~\ref{prop-word-flip}, we first explain why~\eqref{eqn-word-flip-hyp} implies that $x_{s(x)}$ is identified in $x$ (i.e., $r(x) > 0$) and that the word $\cR(x_{r(x)} \dots x_{s(x)-1})$ must take the form $\co^n \hb^m$ for some $m\geq 1$ and $n\geq 0$ (where here $\co^n$ denotes the word which is a concatenation of $n$ $\co$'s, etc.); see~\eqref{eqn-only-hC}. 
By means of~\eqref{eqn-reduced-decomp}, we then reduce to the case when $\cR( x_1\dots x_{r(x)-1}) $ (resp.\ $\cR( x_{s(x)} \dots x_{|x|})$) contains only $\hb$'s and $\cb$'s (resp.\ $\ho$'s and $\co$'s). 
This reduction together with~\eqref{eqn-only-hC} will allow us to write down explicit expressions for the quantities in~\eqref{eqn-word-flip-conc} in terms of $n$ and $m$. Comparing these expressions will yield~\eqref{eqn-word-flip-conc}.

\begin{proof}[Proof of Lemma~\ref{prop-word-flip}]
If \eqref{eqn-word-flip-hyp} holds, then $x^\Hc \neq x$, so the word $x$ contains at least one $\db$ or $\so$.
Since $x_{s(x)}$ is the rightmost $\db$ or $\so$ in the word $x$,
replacing $x_{s(x)}$ by $\cb$ or $\ho$ does not change the identification of any symbol in $x_{s(x)+1} \cdots x_{|x|}$.

We first argue that~\eqref{eqn-word-flip-hyp} implies that $x_{s(x)}$ is identified in $x$,
and hence that $r(x)>0$.
Indeed, suppose $x_{s(x)}$ is not identified in the word $x$.
Then the reduced word $x(1,s(x)-1)$ contains no $\hb$ or $\cb$ symbols, since the presence of any such symbol would identify $x_{s(x)}$ (recall~\eqref{eqn-theta-relations'}).
In this case, the reduced words $\cR(x)$ and $\cR(x^\Hc)$ would have the same set of symbols except
the symbol coming from position $s(x)$, and possibly an order in $x_{s(x)+1} \cdots x_{|x|}$ which may consume $x^\Hc_{s(x)}$ if it
is a burger.  But then $\cB\left(\cR(x^\Hc)\right) \leq \cB(\cR(x))$,
contradicting~\eqref{eqn-word-flip-hyp}.   

Henceforth assume that~\eqref{eqn-word-flip-hyp} holds, which implies (by the preceding paragraph) that $x_{s(x)}$ is identified in $x$.
If $r(x)<k<s(x)$ and $x_k$ is a burger, then
by definition of $r(x)$, either
$x_k\in\{\hb,\cb\}$ but is consumed by an order in $x_{r(x)}\cdots x_{s(x)}$, or else $x_k=\db$.
If $x_k=\db$ and is identified as $x_{r(x)}^\dagger$, consider the first such $k$.
By definition of $r(x)$, the burger that identifies $x_k$ is consumed
in $x_{r(x)}\cdots x_{s(x)}$, at a time by which $x_k$ must therefore also have been consumed.
Thus the only burgers in $\cR(x_{r(x)}\cdots x_{s(x)})$ are $\db$'s that are identified as $x_{r(x)}$ and the burger $x_{r(x)}$ itself.

If $x_{s(x)} = \so$, then since $\cR(x^\Hc) \neq \cR(x)$, it must be that
$x_{s(x)}$ corresponds to a $\co$ symbol in the identification $\cI(x)$,
which in turn implies $x_{r(x)} = \hb$.
If on the other hand
$x_{s(x) } = \db$, then $x_{s(x)}$ must be identified by an $\hb$ in the word $x$,
which again implies
$x_{r(x)} = \hb$, since from the previous paragraph we know that all potential intermediate burgers would be $\db$'s.
Since the burger $x_{r(x)}=\hb$ is not consumed in $\cR(x_{r(x)}\cdots x_{s(x)})$, any order in
this reduced word is identified and must be of type $\co$.
Regardless of $x_{s(x)}$,
\eqb \label{eqn-only-hC}
\cR\left(x_{r(x)}\cdots x_{s(x)-1}\right) = \co^{n} \hb^{m} \quad\quad\quad\text{with $m\geq1$ and $n\geq0$.}
\eqe

Write $x(1, r(x)-1) = Uu$ and $x(s(x)+1,|x|) = Vv$, where $U$ and $V$ are words consisting of only orders and $\db$'s, and $u$ and $v$ are words consisting of only $\hb$'s and $\cb$'s.  By definition of $s(x)$, $V$ contains no $\so$ or $\db$.
Let $\alpha$ denote the identification of $x_{s(x)}$ in $x_{r(x)}\cdots x_{s(x)}$, which is either $\co$ or $\hb$.
By the relation $\mcl R(\mcl R(x)\mcl R(y)) = \mcl R(xy)$ (Lemma~\ref{prop-associative}) and the commutativity of $\hb$ with $\co$,
\begin{equation} \label{eqn-reduced-decomp}
\begin{aligned}
\cR(x) &= U \cR\left( u \hb^m \co^n \alpha  V \right)  v,
& \cR\left(x^\Hc \right) &= U \cR\left( u \hb^m \co^n \alpha^\dagger  V \right)  v\,, \\
\cR(\Psi(x)) &= U^\dagger \cR\left( u^\dagger \hb^m \co^{n} \alpha  V^\dagger \right)  v^\dagger,\quad
& \cR\left( \Psi(x)^\Hc \right) &= U^\dagger \cR\left( u^\dagger \hb^m \co^n \alpha^\dagger V^\dagger \right)  v^\dagger \,.
\end{aligned}
\end{equation}
From~\eqref{eqn-reduced-decomp} we see that changing $U$ and $v$ while leaving the other words fixed does not change $\cB\left( \cR( x^\Hc) \right) - \cB( \cR( x ) )$ or $\cB\left( \cR( \Psi(x)^\Hc ) \right) - \cB( \cR( \Psi(x) ) )$, so we assume without loss of generality that $U = v = \emptyset$.

Under this assumption, the words $\cR(x)$ and $\cR(\Psi(x))$ both take the form $\cR(y Y)$, where $y$ is a word with only hamburgers and cheeseburgers and $Y$ is a word with only hamburger orders and cheeseburger orders.  If $\alpha$ is an order, then $\cR(x^\Hc)$ and $\cR(\Psi(x)^\Hc)$ also take the form $\cR(y Y)$, but if $\alpha$ is a burger and $n>0$, then $\cR(x^\Hc)$ and $\cR(\Psi(x)^\Hc)$ take the form $\cR(y \co^n \cb Y)$ (where in both cases, as above, $y$ denotes a word with only $\cb$'s and $\hb$'s, and $Y$ a word with only $\co$'s and $\ho$'s).

For convenience we define
\begin{align*}
\Delta_\hb &\colonequals \cN_{\hb}\left(u \right) - \cN_{\ho}\left(V \right) \\
\Delta_\cb &\colonequals \cN_{\cb}\left(u \right) - \cN_{\co}\left(V \right) \,.
\end{align*}

Suppose first $\alpha=\co$.  From \eqref{eqn-reduced-decomp} we see
\begin{equation*}
\begin{split}
\cB(\cR(x)) &= \left( \Delta_\hb + m \right) \vee 0 + \left(\Delta_\cb - n-1 \right) \vee 0 \\
\cB(\cR(x^\Hc)) &= \left( \Delta_\hb + m -1\right) \vee 0 + \left(\Delta_\cb - n \right) \vee 0
\end{split}
\end{equation*}
Since $\cB(\cR(x^\Hc)) > \cB(\cR(x))$ it follows that $\Delta_\hb \leq -m$ and $\Delta_\cb \geq n+1$,
and hence $\cB(\cR(x^\Hc)) = \cB(\cR(x))+1$, as claimed.
From \eqref{eqn-reduced-decomp} together with $\Delta_\hb\leq -1$ and $\Delta_\cb\geq 1$, we see
\begin{equation*}
\begin{split}
\cB(\cR(\Psi(x))) &= \left( \Delta_\cb + m \right) \vee 0 + \left(\Delta_\hb - n-1 \right) \vee 0 = (\Delta_\cb+m)+0\\
\cB(\cR(\Psi(x^\Hc))) &= \left( \Delta_\cb + m -1\right) \vee 0 + \left(\Delta_\hb - n \right) \vee 0 = (\Delta_\cb+m-1)+0
\end{split}
\end{equation*}
so $\cB(\cR(\Psi(x^\Hc))) = \cB(\cR(\Psi(x)))-1$, as claimed.

Suppose next $\alpha=\hb$.  From \eqref{eqn-reduced-decomp} we see
\begin{equation*}
\begin{split}
\cB(\cR(x)) &= \left( \Delta_\hb + m + 1 \right) \vee 0 + \left(\Delta_\cb - n \right) \vee 0 \\
\cB(\cR(x^\Hc)) &= \left( \Delta_\hb + m \right) \vee 0 +
[((\cN_\cb(u)-n)\vee 0)+1-\cN_\co(V)]\vee 0
\end{split}
\end{equation*}
The nested-$\vee$ expression arises because $\cR(x^\Hc)$ takes the form $\cR(y \co^n \cb Y)$.
Since $\Delta_\cb=\cN_\cb(u)-\cN_\co(V)$ and $\cB(\cR(x^\Hc)) > \cB(\cR(x))$, it follows (by a short argument by contradiction due to the nested-$\vee$ expression) that $\Delta_\hb\leq-m-1$ and
\[
(\cN_\cb(u)-n)\vee 0 \geq \cN_\co(V),
\]
which in turn implies either $\cN_\co(V)=0$ or $\Delta_\cb\geq n$.  In either case, $\Delta_\cb\geq 0$.
We also see $\cB(\cR(x^\Hc)) = \cB(\cR(x))+1$, as claimed.
Referring to \eqref{eqn-reduced-decomp} again, and using from above that $\Delta_\hb\le -2$ and $\Delta_\cb\ge 0$, we see
\begin{align*}
\cB(\cR(\Psi(x))) &= \left( \Delta_\cb + m + 1 \right) \vee 0 + \left(\Delta_\hb - n \right) \vee 0 = (\Delta_\cb+m+1)+0\\
\cB(\cR(\Psi(x^\Hc))) &= \left( \Delta_\cb + m \right) \vee 0 +
[((\cN_\hb(u)-n)\vee 0)+1-\cN_\ho(V)]\vee 0 \\
& = (\Delta_\cb+m)+[((\Delta_\hb-n)\vee (-\cN_\ho(V)))+1]\vee 0
\intertext{Since $\cN_{\hb}\left(u \right) - \cN_{\ho}\left(V \right)=\Delta_\hb\leq-m-1\leq -2$, it follows that $\cN_\ho(V) \geq 2$, and so}
\cB(\cR(\Psi(x^\Hc)))
& =  \Delta_\cb+m \,,
\end{align*}
so in this case as well $\cB(\cR(\Psi(x^\Hc))) = \cB(\cR(\Psi(x)))-1$, as claimed.
\end{proof}

\begin{proof}[Proof of Lemma~\ref{prop-mean-mono}]
The law of $n\mapsto \cC(X(1,n))$ is that of one-dimensional simple random walk, regardless of $p$ and $q$. Therefore $\PP^{p,q}(E) = \PP^{0,0}(E)$, so to prove~\eqref{eqn-mean-mono-B} it suffices to show
\eqb \label{eqn-mean-mono}
\E^{p,q} \left( \cB(X(1,n )) \,|\, E \right) \geq \E^{0,0}\left( \cB(X(1,n )) \,|\, E \right),\quad \forall n \in \N.
\eqe
To this end, let $X^0 = X_{1}^0 \cdots X_{n}^0$ be a word whose law is that of $X_{1} \cdots X_{n}$ under $\PP^{p,q}$. Let $\{\xi_k\}_{k\in [1,n]_\Z}$ be i.i.d.\ Bernoulli random variables with parameter $1/2$, independent from $X^0$. For $k\in [1,n]_\Z$ inductively define
\eqbn
X^k = \begin{dcases}
(X^{k-1} )^\Hc \quad &\op{if} \: \xi_k = 0\\
(X^{k-1})^\Ch \quad &\op{if} \: \xi_k = 1.
\end{dcases}
\eqen
Since $\cN_{\db|\so}(X^k)=0\vee(\cN_{\db|\so}(X^{k-1})-1)$, and the word $X^n$ is obtained from $X^0$ by replacing each $\db$ symbol in $X^0$ with an independent random symbol which is uniformly distributed on $\left\{\hb,\cb\right\}$ and each $\so$ symbol in $X$ with an independent random symbol which is uniformly distributed on $\left\{\ho,\co\right\}$, the law of $X^n$ is that of $X_{1} \cdots X_{n}$ under $\PP^{0,0}$.

We next argue that
\eqb \label{eqn-flip-law}
\Psi(X^k) \eqD X^k \quad \forall \: k\in [0,n]_\Z.
\eqe
To see this, let $k\in[1,n]_\Z$ and let $j_k$ be the $k^{\text{th}}$ largest $j\in [1,n]_\Z$ for which $X^0_{j} \in \left\{\db,\so\right\}$, or $j_k = 0$ if no such $j$ exists. Also let $j_k'$ be the largest $j \in [1, j_k-1]_\Z$ for which the reduced word $X^0(j, j_k-1)$ contains an $\hb$ or $\cb$, or $j_k' = 0$ if no such $j$ exists.
Then $j_k$ and $j_k'$ are stopping times for the filtration generated by $X^0$, read from right to left. By the strong Markov property, the conditional law of $X^0_{1} \cdots X^0_{ j_k'-1}$ given $X^0_{ j_k'}\cdots X^0_{n}$ is a string of $(j_k'-1)\vee 0$ i.i.d.\ symbols sampled from the law $\PP^{p,q}$.
Hence given $j_k'$, $X^0_{1} \cdots X^0_{ j_k'-1}$ is conditionally independent from $X^0_{ j_k + 1} \cdots X^0_{n}$ and $X^0_{ j_k'} \cdots X^0_{ j_k}$.

By the above description of the conditional law of $X^0_{1} \cdots X^0_{ j_k'-1}$ given $j_k'$ and $X^0_{ j_k'}\cdots X^0_{n}$ and the symmetry between hamburgers and cheeseburgers, we infer that this conditional law is invariant under involution. 
Since the definition of $j_k$ is invariant under involution, we infer that also the conditional law of $X^0_{ j_k+1} \cdots X^0_{n}$ given $j_k$ is invariant under involution. Since $j_k$ is a stopping time for $X^0$, read backwards, it follows that the joint conditional law of $X^0_{1} \cdots X^0_{ j_k'-1}$ and $X^0_{ j_k+1} \cdots X^0_{n}$ given $j_k$, $j_k'$ and $X^0_{ j_k'} \cdots X^0_{ j_k}$ is invariant under involution.
In particular,
\eqb \label{eqn-flip-law-whole}
X^0\eqD \left( X^0_{1} \cdots X^0_{ j_k'-1} \right)^\dagger X^0_{ j_k'}\cdots X^0_{j_k} \left( X^0_{ j_k+1} \cdots X^0_{n} \right)^\dagger.
\eqe
The word $X^k$ (resp.\ $\Psi(X^k)$) is obtained from the word on the left (resp.\ right) side of~\eqref{eqn-flip-law-whole} by replacing its $k$ rightmost $\db$ or $\so$ symbols with independent random symbols sampled uniformly from $\left\{\hb,\cb\right\}$ or $\left\{\ho,\co\right\}$ respectively.  We thus obtain~\eqref{eqn-flip-law}.

Now let $E$ be an event as in the statement of the lemma, defined with the word $X^0_1 \dots X_0^n$ in place of the word $X_1\dots X_n$. 
The operations $x\mapsto x^\Hc$, $x\mapsto x^\Ch$, and $x\mapsto \Psi(x)$ replace burgers with burgers and orders with orders in the word $x$, so the sequence $\cC(x(1,i))_{i=1,\dots,n}$ is the same for each $x \in \left\{ X^k,\Psi(X^k)\right\}$ and $k \in [0,n]_\Z$.
Since the event $E$ is determined by $\cC(X^0(1,i))_{i=1,\dots,n}$,
we see that the definition of $ E$ is unaffected if we replace $X^0$ with $X^k$ or $\Psi(X^k)$ for any $k\in [0,n]_\Z$. From this observation, we deduce the following:
\begin{enumerate}
\item The conditional law of $X^0$ given $E$ is the same as the conditional law of $X_1\cdots X_n$ given $E$ under $\PP^{p,q}$. \label{item-E-start-law}
\item The conditional law of $X^n$ given $E$ is the same as the conditional law of $X_1\cdots X_n$ given $E$ under $\PP^{0,0}$. \label{item-E-end-law}
\item $E$ is independent from the Bernoulli random variables $\{\xi_k\}_{k\in [1,n]_\Z}$. \label{item-E-ind}
\item By~\eqref{eqn-flip-law}, for each $k\in [1,n]_\Z$, the conditional laws of $X^k$ and $\Psi(X^k)$ given $E$ agree. \label{item-E-flip}
\end{enumerate}
By combining these observations with Lemma~\ref{prop-word-flip}, we find that for each $k\in [1,n]_\Z$,
\begin{align} \label{eqn-diff-compare}
&\PP\!\left( \cB\left( \cR( X^k ) \right) > \cB\left( \cR(X^{k-1} ) \right) \,|\, E \right) \notag \\
&\qquad = \tfrac12 \PP\!\left( \cB\left( \cR( (X^{k-1})^\Hc ) \right) > \cB\left( \cR(X^{k-1}) \right) \,|\, E \right)
+ \tfrac12 \PP\!\left( \cB\left( \cR( (X^{k-1})^\Ch ) \right) > \cB\left( \cR( X^{k-1} ) \right) \,|\, E \right) \notag \\
&\qquad\leq \tfrac12 \PP\!\left(\cB\left( \cR( \Psi(X^{k-1})^\Hc ) \right) = \cB\left( \cR( \Psi(X^{k-1}) ) \right) - 1 \,|\, E \right) \notag\\
&\qquad\quad+\tfrac12\PP\!\left(\cB\left( \cR( \Psi(X^{k-1})^\Ch ) \right) = \cB\left( \cR( \Psi(X^{k-1}) ) \right) - 1 \,|\, E \right) \notag \\
&\qquad= \PP\!\left( \cB\left(\cR( X^k ) \right) = \cB\left(\cR( X^{k-1} ) \right) - 1 \,|\, E \right).
\end{align}
We used observation~\ref{item-E-ind} above in the first equality and observation~\ref{item-E-flip} in the last equality. Lemma~\ref{prop-word-flip} implies that $\cB( \cR( X^k )  ) = \cB(\cR( X^{k-1} ) ) + 1$ whenever $\cB(\cR( X^k )  ) > \cB(\cR( X^{k-1} ) )$, so~\eqref{eqn-diff-compare} implies 
\alb
&\E\left( \cB\left( \cR( X^k ) \right)  -  \cB\left( \cR( X^{k-1}) \right) \,|\, E \right) \notag \\
&\qquad \leq \PP\!\left( \cB\left( \cR( X^k ) \right) > \cB\left( \cR(X^{k-1} ) \right) \,|\, E \right) - \PP\!\left( \cB\left(\cR( X^k ) \right) = \cB\left(\cR( X^{k-1} ) \right) - 1 \,|\, E \right)  \leq 0 ,
\ale 
whence
\eqbn
\E\left( \cB\left( \cR( X^k ) \right) \,|\, E \right) \leq \E\left( \cB\left( \cR( X^{k-1}) \right) \,|\, E \right) \quad \forall k\in [1,n]_\Z.
\eqen
Therefore
\eqb \label{BRXn<=BRX0}
\E\left( \cB\left( \cR( X^n ) \right) \,|\, E \right) \leq \E\left( \cB\left(\cR( X^0 ) \right) \,|\, E \right).
\eqe
By observations~\ref{item-E-start-law} and~\ref{item-E-end-law} above, we obtain~\eqref{eqn-mean-mono-B}.

The bound~\eqref{eqn-mean-mono-O} follows observations~\ref{item-E-start-law} and~\ref{item-E-end-law} above,
\eqref{BRXn<=BRX0} and
\[
\cB(\cR(X^0))-\cO(\cR(X^0))
= \cB(\cR(X^n))-\cO(\cR(X^n)) \,. \qedhere
\]
\end{proof}

\subsection{Bound on the number of unidentified symbols}
\label{sec-few-SD}

In the next three subsections we prove analogues of various results found in~\cite[\S~3]{shef-burger} in the setting of Theorem~\ref{thm-variable-SD}.
Throughout, we assume we are in the setting described just above the statement of Theorem~\ref{thm-variable-SD} for fixed $(p,q) \in [0,1] \times [0,1)$. 

The main purpose of this section is to prove the following more quantitative analogue of~\cite[Lem.~3.7]{shef-burger}.

\begin{lem} \label{prop-few-SD}
For each $\ep>0$, there are positive numbers $c_0,c_1>0$ such that, for each $n\in\N$ and $A>0$,
the event
\eqb \label{eqn-few-SD-event}
F_n(\ep,A) \colonequals \left\{ \frac{\cN_{\db|\so}(X(1,n))}{ \cN_{\ho}(X(1,n)) \vee A } \geq \ep \right\} 
\eqe
occurs with probability
\eqb \label{eqn-few-SD}
\PP\left( F_n(\ep,A) \right) \leq c_0 e^{-c_1 A }.
\eqe
\end{lem}

Lemma~\ref{prop-few-SD} will be an important tool in what follows since it allows us in many cases to ignore the (potentially quite complicated) manner in which the $\db$'s and $\so$'s are identified. When we apply the lemma, we will typically take $\ep$ to be a small fixed parameter and $A$ to be a small positive power of $n$ (so that $\PP(F_n(\ep,A)  )$ decays faster than any negative power of $n$). 
We expect that an even stronger statement than Lemma~\ref{prop-few-SD} is true, namely, that $ \cN_{\db|\so}(X(1,\infty)) < \infty$ a.s.\ and that $\cN_{\db|\so}(X(1,\infty))$ is stochastically dominated by a geometric distribution. 
The reason for this is explained in Remark~\ref{remark-I-infinite}.

To prove Lemma~\ref{prop-few-SD}, we first observe that if $i \in [1,n]_{\Z}$ is such that $X_i$ is a $\db$ or $\so$ which is not identified in $X_1\dots X_n$, then the word $X(1,i-1)$ must contain no hamburgers or cheeseburgers (such a hamburger or cheeseburger would identify $X_i$). 
We will prove that the expected number of unmatched $\ho$'s added to the word between the successive times when $X(1,j)$ contains no burgers is infinite (Lemma~\ref{prop-time-to-SD-infty}). 
By Hoeffding's inequality and the fact that the increments of the word $X$ between these successive times are i.i.d., this will tell us that the number of $\db$'s and $\so$'s in $X(1,n)$ is typically negligible compared to the number of $\ho$'s. 

To start off, we consider the time
\begin{equation}
K = \min\left\{i \in\N : \cC(X(1,i)) = -1\right\}
\end{equation}
(here $i\mapsto \cC(X(1,i))$ is the simple random walk as in Definition~\ref{def-theta-count}).

\begin{lem} \label{prop-mean-infty}
We have
\eqb \label{eqn-mean-infty}
\E \left( \left| X(1,K)\right| \right) = \infty.
\eqe
Furthermore, if we let $P$ be the smallest $j\in \N$ for which $\cC(X(-j,-1)) =1$, then
\eqb \label{eqn-mean-infty'}
\E \left( \left| X(-P,-1)\right| \right) = \infty.
\eqe
\end{lem}
\begin{proof}
For each $n\in\N$, the event $\{K=n\}$ depends only on $\cC(X(1,i))$ for $i\in [1,n]_\Z$. By Lemma~\ref{prop-mean-mono}, we find
\eqbn
\E \left( \left| X(1,K) \right| \times \one_{(K=n)} \right) \geq \E^{0,0}\left( \left| X(1,K) \right| \times \one_{(K=n)} \right),\quad \forall n \in \N
\eqen
where here $\E^{0,0}$ denotes the law of $X$ with $p = q= 0$. By summing over all $n$, we obtain
\eqb \label{eqn-mean-infty-compare}
\E \left( \left| X(1,K) \right| \right) \geq \E^{0,0}\left( \left| X(1,K) \right| \right).
\eqe

By standard estimates for one-dimensional simple random walk, $\PP^{0,0}\left( K = n \right) \asymp n^{-3/2}$.
Under $\PP^{0,0}$, if we condition on $\{K = n\}$, then the conditional law of the walk $\ddi = (\dipt, \didt)$ restricted to $[0,n ]_{\Z}$ is that of a two-dimensional simple random walk conditioned to first exit the diagonal half plane $\{x + y \geq 0\}$ at time $n $. With uniformly positive probability under this conditioning, it holds that
\eqbn
\dipt(n) - \inf_{i \in [0,n]_{\Z}} \dipt(i) \geq n^{1/2},
\eqen
in which case $\cN_{\hb}(X(1,n)) \geq n^{1/2}$.
Therefore,
\eqbn
\E^{0,0}\left( \left|  X(1,K) \right| \times \one_{(K=n)} \right) \succeq n^{-1 }.
\eqen
By summing over all $n\in\N$ we obtain $\E^{0,0}\left( \cB( X(1,K)) \right) = \infty$ and hence~\eqref{eqn-mean-infty}.

We similarly obtain~\eqref{eqn-mean-infty'}.
\end{proof}

\begin{lem} \label{prop-time-to-SD-infty}
Let $I_1$ be the smallest $i\in\N$ for which $X(1,i)$ contains no hamburgers or cheeseburgers. Then $\E\left(\cN_{\ho}(X(1,I_1))\right) = \infty$ (here we take $\cN_{\ho }(X(1,I_1))= \infty$ if $I_1 =\infty$).
\end{lem}

\begin{remark}  \label{remark-I-infinite}
It is possible that $I_1=\infty$ with positive probability. 
In fact, we expect (but do not prove) that this is the case since the coordinates of the re-scaled walk $Z^n$ in~\eqref{eqn-Z^n-def} should be close to attaining a simultaneous running infimum at time $I_1$; and the coordinates of the negatively correlated Brownian motion $Z$ in Theorem~\ref{thm-variable-SD} a.s.\ do not have any simultaneous running infima (this follows by applying a linear transformation and using that an uncorrelated two-dimensional Brownian motion a.s.\ has no $\theta$-cone times for $\theta < \pi/2$~\cite{shimura-cone,evans-cone}). Note that if $I_1 = \infty$ with positive probability, then a.s.\ there are only finitely many times in $\BB N$ for which $X(1,i)$ contains no $\hb$'s or $\cb$'s, and hence only finitely many unidentified $\db$'s and $\so$'s in $X(1,\infty)$. We note, by way of comparison, that in the setting when $p_\so=p_\db=p_\eb=0$ and $p_\fo \in (0,1]$, the word $X(1,\infty)$ a.s.\ contains infinitely many $\fo$'s; see~\cite[Lemma 3.7]{shef-burger} in the case $p_\fo > 1/2$ and~\cite[Proposition 3.5]{gms-burger-cone} in the case $p_\fo < 1/2$ (the same proof works for $p_\fo =1/2$). 
\end{remark} 

\begin{proof}[Proof of Lemma~\ref{prop-time-to-SD-infty}]
The statement of the lemma is obvious if $I_1 = \infty$ with positive probability, so we can assume that $I_1 < \infty$ a.s.
If $I_1 > 1$ and $X(1,I_1)$ contains a $\db$ or $\so$ symbol, then $X(1,i)$ would have to contain no hamburgers or cheeseburgers for some $i\leq I_1-1$ (corresponding to the index of the $\db$ or $\so$ in question), which contradicts the definition of $I_1$.  Thus either $I_1=1$ or the word $X(1,I_1)$ contains no unidentified $\db$'s or $\so$'s.

If $I_1>1$, since every burger in $X(1,I_1)$ is identified, by definition of $I_1$, it must be that
$X(1,I_1)$ contains no burgers.  Thus if $I_1>1$, the word $X(2,I_1)$ contains more orders than burgers.

Now let $K_2$ be the smallest $i \geq 2$ for which $\cC(X(2,i)) \leq -1$. Then $X_2 \cdots X_{K_2}$ is independent from $X_1$ and agrees in law with $X_1\cdots X_K$.  On the event $\left\{X_1 = \hb\right\}$, we have $I_1 \geq K_2$. Therefore, every order appearing in $X(2, K_2)$ except possibly one also appears in $X(1, I_1)$.  It follows that
\eqbn
\E\left(\cN_{\ho|\co}(X(1,I_1 ) \right)
\geq \frac{1-q}{4} \E\left( |X(1, K)| - 1 \right)
=\infty.
\eqen
By symmetry between $\ho$ and $\co$, we also have $\E\left(\cN_{\ho }(X(1,I_1))\right) = \infty$.
\end{proof}

\begin{proof}[Proof of Lemma~\ref{prop-few-SD}]
Let $I_0 = 0$ and for $m \in \N$, let $I_m$ be the $m^{\text{th}}$ smallest $i\in\N$ for which $X(1, i)$ contains no hamburgers or cheeseburgers.  The definition of $I_1$ is the same as that given in Lemma~\ref{prop-time-to-SD-infty}. Furthermore, if $i\in\N$ and $X_i$ is a $\db$ or a $\so$ which is not identified in $X_1 X_2 \cdots$, then $i$ must be one of the times $I_m$ for $m\in\N$.

For each $m \in\N$, the time $I_m$ is a stopping time for the filtration generated by $X$, read forward. Furthermore, for $m\in\N$ and $i\geq I_{m-1} + 1$, the word $X(1, i)$ contains no hamburgers or cheeseburgers if and only if $X(I_{m-1} + 1, i)$ contains no hamburgers or cheeseburgers. By the strong Markov property, the words $X_{I_{m-1}+1} \cdots X_{I_m}$ for $m\in\N$ are i.i.d.

For $m\in\N$, let
\[
\xi_m \colonequals \cN_{\ho}(X(I_{m-1}+1,I_m)),
\]
so that the random variables $\xi_m$ for $m\in\N$ are i.i.d.  None of the $\ho$'s in $X(I_{m-1} + 1, I_m)$ have a match in $X_1 X_2\cdots$, so for each $m\in\N$
\eqbn
\cN_{\ho}(X(1,I_m)) = \sum_{k=1}^m \xi_k.
\eqen
By Lemma~\ref{prop-time-to-SD-infty}, for each $\ep > 0$ we can find an $R > 0$ such that
\eqbn
\E\left( \xi_1 \wedge R \right) \geq 2\ep^{-1}.
\eqen
By Hoeffding's inequality for sums of i.i.d.\ bounded random variables,
for each $m\in\N$,
\begin{align} \label{eqn-SD-increment-to-infty}
\PP\!\left( \cN_{\ho }(X(1,I_m)) \leq \ep^{-1} m \right)
&\leq \PP\!\left( \frac1m \sum_{k=1}^m (\xi_k \wedge R) \leq \ep^{-1} \right) \notag \\
&\leq \exp\left(-\frac{2 m}{\ep^2  R^2} \right).
\end{align}

Given $n\in\N$, let $M_n$ be the largest $m\in\N$ for which $I_m \leq n$. Then $\cN_{\ho }(X(1,n))\geq \cN_{\ho }(X(1,I_{M_n}))$ and $\cN_{\db|\so}(X(1,n))\leq M_n$.
By~\eqref{eqn-SD-increment-to-infty},
\alb
\PP(F_n(\ep,A))
&\leq \PP\!\left( \frac{ M_n }{ \cN_{\ho }(X(1,I_{M_n})) } \geq \ep,\; M_n \geq \ep A \right)\\
& = \sum_{m = \lceil \ep A \rceil}^\infty \PP\!\left( \cN_{\ho }(X(1,I_m)) \leq \ep^{-1} m,\; M_n=m\right) \\
& \leq \sum_{m = \lceil \ep A \rceil}^\infty \exp\left(-\frac{2 m}{\ep^2 R^2} \right)
\ale
so we take $c_1=2/(\ep R^2)$ and $c_0=1/(1-e^{-c_1/\ep})$.
\end{proof}

\subsection{Renewal times in the word}
\label{sec-J-basic}

For the bi-infinite word $X$,
let $J$ be the age of the freshest (unconsumed) non-duplicate burger, as seen from the present:
\eqb \label{def-J}\hypertop{def-J}
 \J\colonequals\min\big\{j\in\N: \cN_{\hb|\cb}(X(-j,-1))>0\big\},
\eqe
and more generally we define a sequence of backward renewal times $\J_m$ by
\eqb \label{def-J_m}
 \J_m\colonequals\begin{cases} 0 & m=0 \\ \min\big\{j\in\N: \cN_{\hb|\cb}(X(-j,-J_{m-1}-1))>0\big\} & m\in\N. \end{cases}
\eqe
We also define
\eqb \label{eqn-chi-def}\hypertop{def-chi}
\CHI \colonequals \E\left(|X(-\J,-1)| \right),
\eqe
In the case $p=q=0$, $\E[\J]=\infty$, so
\textit{a priori\/} we could have $\CHI = \infty$,
but we will prove that $\CHI$ is finite in Proposition~\ref{prop-J-finite} below.

In this subsection we carry out a careful study of the time $\J$ and related quantities.
These results are needed for the variance calculation in the next subsection.
We start by recording some basic properties of $\J$ (which follow easily from the definition) in Lemma~\ref{prop-J-basic} and an alternative definition of $\J_m$ in Lemma~\ref{prop-J-af}. 
In Lemma~\ref{prop-J-moment}, we show that $\J$ has finite moments up to order $1/2$. The idea of the proof is to bound $\J$ above by a time associated with the simple random walk $j\mapsto \cC(X(-j,-1))$. 
Using this and Lemma~\ref{prop-few-SD}, we prove in Proposition~\ref{prop-J-finite} that $\CHI := \BB E( |X(-\J, -1)|)$ is finite and that $\BB E( \cC(X(-\J, -1)) )$ is non-negative.
We then show that in fact this latter expectation is 0 using a generalization of the proof of~\cite[Lem.~3.5]{shef-burger}. 
Since $|X(-\J , -1)| =  2  -  \cC(X(-\J, -1))$ whenever $X(-\J,-1)$ contains no $\db$'s, this shows in particular that $\CHI  = 2$ when $q= 0$, i.e., $z=1$ (which is why we get an exact expression for the variances and covariances in Theorem~\ref{thm-variable-SD} in this case). 
The last main result of this subsection is Lemma~\ref{prop-J-limit}, which shows that $\BB E( |X(-n,-1)| \BB 1_{(\J < n)} ) \rta 0$ as $n\rta \infty$, and is an easy consequence of the earlier results in this subsection and a dominated convergence argument.

\begin{lem} \label{prop-J-basic}
With $\J$ as in \eqref{def-J},
\begin{enumerate}
\item $\J$ is a.s.\ finite. \label{item-J-finite}
\item $X_{-\J} \in \left\{\hb, \cb \right\}$. \label{item-J-burger}
\item The symbol $X_{-\J}$ does not have a match in $X_{-\J} \cdots X_{-1}$. \label{item-J-match}
\item The reduced word $X(-\J, -1)$ consists of only hamburgers and cheeseburger orders (if $X_{-\J} = \hb$) or cheeseburgers and hamburger orders (if $X_{-\J} = \cb$). \label{item-J-reduced}
\end{enumerate}
\end{lem}
\begin{proof}
Assertion~\ref{item-J-finite} follows from Lemma~\ref{prop-identification-exists}.
By definition of $\J$, the word $X(-\J+1,-1)$ contains no $\hb$ or $\cb$ symbols, so assertion~\ref{item-J-burger} follows from Lemma~\ref{prop-associative} (applied with $x=X_{-\J}$ and $y = X_{-\J+1}\dots X_{-1}$).

Suppose $k\in[1,\J-1]_\Z$.  By definition of $\J$, the word $X(-k,-1)$ contains no $\hb$ or $\cb$.  If $X(-\J,-k-1)$ contained
no burger, then $\cR(X(-\J,-k-1)X(-k,-1))=X(-\J,1)$ would contain no $\hb$ or $\cb$, contrary to the definition of $\J$.
So $X(-\J,-k)$ contains a burger.

We argue by induction on $\J-k\in[0,\J-1]_\Z$ that each symbol in $X_{-\J}\cdots X_{-k}$ is identified in this word.
Since $X_{-\J}\in\{\hb,\cb\}$, this is true for $k=\J$.  If the claim is true for $k$, then since $X(-\J,-k)$ contains a burger,
each of which by induction is identified, it follows that $X_{-k+1}$ is identified in $X_{-\J}\cdots X_{-k+1}$, completing the
induction.

Every burger in $X(-\J+1, -1)$ is a $\db$.  Since each burger in $X(-\J,-1)$ is identified, it must be that they are
identified to $X_{-\J}$.

Suppose that $X_{-\J}$ is matched to an order $X_{\phi(-\J )}$ for $\phi(-\J ) \in [-\J +1, -1]_\Z$.
We assume without loss of generality that $X_{-\J} = \hb$.
Consequently, $X(-\J, -1)$ contains no $\cb$.
Since $X_{-\J}=\hb$ is consumed, the reduced word $X(-\J, \phi(-\J))$ consists of only $\cb$'s and $\co$'s.
Since $X(\phi(-\J) +1,-1)$ contains no $\hb$ or $\cb$, each $\db$ in $X(\phi(-\J) +1,-1)$
is identified by a $\cb$ in $X_{-\J} \cdots X_{-1}$.
Consequently, $X(-\J,-1)$ contains no $\hb$.
We have already shown above that $X(-\J,-1)$ contains no $\cb$, so we contradict the definition of $\J$. We thus obtain assertion~\ref{item-J-match}.

Since each burger in $X(-\J,-1)$ is identified to $X_{-\J}$, and $X_{-\J}$ is not consumed, it must be that each order in $X(-\J,-1)$
is for the opposite burger type, which proves assertion~\ref{item-J-reduced}.
\end{proof}

Our next lemma is an analogue of~\cite[Lem.~A.7]{gms-burger-cone} in the setting where we read the word backward, rather than forward, and is proven in a similar manner.

\begin{lem} \label{prop-J-af}
The time $ \J_m$ from~\eqref{def-J_m} is the $m^{\text{th}}$ smallest $j\in \N$ such that $\cN_{\hb|\cb}(X(-j, -k))>0$ for all $k\in [1,j]_\Z$.
\end{lem}
\begin{proof}
Let $\wt J_0 = 0$ and for $m\in\N$, let $\wt J_m$ be the $m^{\text{th}}$ smallest $j\in\N$ such that $X(-j, -k)$ contains a hamburger or a cheeseburger for each $k\in [1,j]_\Z$.  We show by induction that $\wt J_m = \J_m$ for each $m\in \N$.  The base case $m=0$ is trivial. Suppose $m\in\N$ and $\wt J_{m-1} = \J_{m-1}$. By assertion~\ref{item-J-match} of Lemma~\ref{prop-J-basic}, the word $X(-\J_m, - k)$ contains a hamburger or a cheeseburger (namely $X_{-\J_m}$) for each $k\in [\wt J_{m-1}+1, \J_m]_\Z$. By definition of $\wt J_{m-1}$, the word $X(-\wt J_{m-1}, -k)$ (and hence the word $X(-\J_m,-k)$) contains a hamburger or a cheeseburger for each $k\in [ 1, \wt J_{m-1}]_\Z$.  Thus $\J_m$ is one of the $\wt J_{m'}$'s, and hence $\J_m \geq \wt J_m$.  On the other hand, the word $X(-\wt J_m, -\J_{m-1}-1)$ contains a hamburger or cheeseburger by the inductive hypothesis and the definition of $\wt J_m$, so $\J_m \leq\wt J_m$, so in fact $\wt J_m = \J_m$.
\end{proof}

We next prove that $\J$ has finite moments up to order $1/2$ (actually we prove something a little stronger, which will be needed for technical reasons below).

\begin{lem} \label{prop-J-moment}
Let $M$ be the smallest $m\in\N$ for which $\cC(X(-\J_m,-1)) \geq 1$.
Almost surely $M<\infty$, and
for each $\zeta \in (0,1/2)$, we have $\E(\J_M^\zeta) < \infty$.
\end{lem}
\begin{proof}
Let $P_0 = 0$ and for $m\in\N$, let $P_m$ be the smallest $j \in\N$ for which $\cC(X(-j,-1)) = m$, as in Lemma~\ref{prop-backward-burger}. Also let $\wt M$ be the smallest $m\in\N$ for which $X_{-P_m} \in \left\{\hb, \cb\right\}$. By Lemma~\ref{prop-backward-burger}, the word $X(-P_{\wt M},-n)$ contains either a hamburger or a cheeseburger for each $n\in [1,P_m]_\Z$. Therefore, Lemma~\ref{prop-J-af} implies that $P_{\wt M} = \J_{\wt m}$ for some $\wt m\in\N$. Since $\cC(X(-P_{\wt M}, -1)) = \wt M \geq 1$, we have $M \leq \wt m$. Therefore $\J_M \leq P_{\wt M}$.

For $\zeta\in (0,1/2)$, the function $t\mapsto t^\zeta$ is concave, hence subadditive. Thus, for $m\in\N$
\eqbn
P_m^\zeta \leq \sum_{k=1}^m (P_k - P_{k-1})^\zeta.
\eqen
Since $j\mapsto \cC(X(-j,-1))$ is a simple random walk, $\E\left( P_1^\zeta \right) < \infty$ for $\zeta \in (0,1/2)$. By the strong Markov property, for each $m \in \N$, it holds with conditional probability $1-q$ given $X_{-P_{m-1}} \cdots X_{-1}$ that $X_{-P_m} \in \left\{\hb, \cb\right\}$.  Therefore, the law of $\wt M$ is geometric with success probability $1-q$, and in particular $\E(\wt M) < \infty$.
By Wald's equation, it holds for each $\zeta \in (0,1/2)$ that $\E(P_{\wt M}^\zeta) < \infty$, and hence also $\E(\J_M^\zeta ) < \infty$.
\end{proof}

We are now ready to prove that the quantity $\CHI$ of~\eqref{eqn-chi-def} is finite.

\begin{prop} \label{prop-J-finite}
\eqb \label{eqn-J-mean-finite}
\CHI = \E\left(|X(-\J,-1)|\right) < \infty
\eqe
and
\eqb \label{eqn-count-mean-pos}
 \E\left( \cC(X(-\J, -1) ) \right) \geq 0.
\eqe
\end{prop}
\begin{proof}
Fix $\ep, \zeta \in (0,1/2)$ and for $n\in\N$, let $F_n=F_n(\ep,n^\zeta)$ be defined as in~\eqref{eqn-few-SD-event} but with $X(-n,-1)$ in place of $X(1,n)$.
Let
\eqb \label{eqn-no-F-sum}
\Xi \colonequals \sum_{n=1}^\infty n \one_{F_n}.
\eqe
By Lemma~\ref{prop-few-SD} and translation invariance, $\E(\Xi) < \infty$.

For $n\in\N$, if $F_n$ occurs, then
\eqb\label{C<=Xi}
\cC(X(-n,-1)) \leq |X(-n,-1)| \leq n \leq \Xi.
\eqe

For $n\in\N$, if $n<\J$ then every burger in $X(-n,-1)$ is a $\db$. 
If $n<\J$ and furthermore $F_n$ does not occur,
then $\cN_{\db|\so}(X(-n,-1)) \leq \ep   \cN_{\ho|\co}(X(-n,-1)) + \ep n^\zeta$ since $\ep  <1$, so
\begin{align}
\cC(X(-n,-1))&=\cN_\db(X(-n,-1))-\cN_{\ho|\co|\so}(X(-n,-1))\notag\\
&\leq\cN_{\db|\so}(X(-n,-1))-\ep\cN_{\ho|\co}(X(-n,-1))\leq\ep n^\zeta\,. \label{C<=nzeta}
\end{align}

For $n\in\N\cup\{0\}$, let
\[Y_n = \cC(X(-(\J\wedge n), -1) )\,.\]
Whether or not $F_{(\J\wedge n)-1}$ occurs, from \eqref{C<=Xi} and \eqref{C<=nzeta} applied to $(\J\wedge n)-1$, we have
\eqb \label{eqn-count-upper}
Y_n \leq 1 + \cC(X(-(\J\wedge n)+1, -1) )\leq 1 + \ep \J^\zeta + \Xi\,.
\eqe
Note that the $\Xi$ comes from the possibility that $F_{(\J\wedge n)-1}$ does not occur.
Since $\cC(X(-n,-1))$ is a martingale, the optional stopping theorem implies $\E[Y_n]=0$.
Let $R=1+\J^\zeta+\Xi$.
By Lemma~\ref{prop-J-moment} (note that $\J \leq \J_M$) and since $\E(\Xi)  <\infty$, we have $\E(R) < \infty$.
Since $0\leq R-Y_n$ and $Y_n\to \cC(X(-\J,-1))$, Fatou's lemma implies
\eqb \label{eqn-count-mean}
\E\left(R - \cC(X(-\J, -1) ) \right) \leq \liminf_n \E(R-Y_n)=\E(R).
\eqe
This in particular implies $\E(\cC(X(-\J,-1)))\geq 0$, i.e., ~\eqref{eqn-count-mean-pos}.

Since every burger in $X(-n,-1)$ is a $\db$ when $n < \J$, 
\eqb \label{eqn-J-word-split}
|X(-n, -1)| \BB 1_{\{n < \J\}} = - \cC(X(-n, -1)) + 2\cN_{\db}\left(X(-n, -1)\right).
\eqe
If $n<\J$ and $F_n$ does not occur, then
\begin{align*}
\cN_\db(X(-n,-1)) &\leq \ep\cN_{\ho|\co}(X(-n,-1)) + \ep n^\zeta \\
(1-\ep)\cN_\db(X(-n,-1)) &\leq -\ep\, \cC(X(-n,-1)) + \ep n^\zeta .
\end{align*}
Note that in the second inequality, we use that $\cN_\db(X(-n,-1)) \leq \ep n^\zeta$ and that $  -\ep \cN_\db(X(-n,-1)) \leq -\ep \cC(X(-n,-1))$ since every burger in $X(-n,-1)$ is a $\db$. 
Combining the above inequalities with \eqref{eqn-J-word-split} gives
\eqb \label{eqn-length-on-F}
|X(-n, -1)| \BB 1_{\{ n < \J \} \cap F_n^c}
\leq - \left(1 + \frac{2\ep}{1-\ep} \right) \cC(X(-n, -1)) + \frac{2\ep}{1-\ep} n^\zeta\,.
\eqe

We combine~\eqref{C<=Xi} and \eqref{eqn-length-on-F}, applied to $n = \J-1$, to obtain
\eqbn
|X(-\J, -1)|
\leq 1 - \left(1 + \frac{2\ep}{1-\ep} \right) \cC(X(-\J, -1)) + \frac{2\ep}{1-\ep} \J^\zeta + \Xi.
\eqen
Since the expectation of each term on the right side of this last inequality is finite, we obtain~\eqref{eqn-J-mean-finite}.
\end{proof}

\begin{lem} \label{prop-overshoot-finite}
With $M$ as in Lemma~\ref{prop-J-moment}, we have
$\E(\cC(X(-\J_M,-1))) < \infty$.
\end{lem}
\begin{proof}
By definition of $M$ and the times $\J_m$,  
\[
 1\leq \cC(X(-\J_M,-1)) \leq \cC(X(-\J_M, -\J_{M-1}-1)) \leq \cN_{\db}(X(-\J_M+1,-\J_{M-1}-1)) + 1
\, ; \] 
in the second inequality, we use that $ \cC(X(-\J_{M-1} , - 1)) \leq  0$.
Since every burger in $X(-\J_M + 1, -\J_{M-1}-1)$ is a $\db$, and 
 \eqbn
\cC(X(-\J_M+1, -\J_{M-1}-1))\geq \cC(X(-\J_M+1, -1)) \geq \cC(X(-\J_M , -1))  - 1    \geq  0  , 
\eqen 
we have
\eqbn
\cN_{\db}\left(X(-\J_M + 1, -\J_{M-1}-1) \right) \geq \cN_{\ho |\co} \left(X(-\J_M + 1, -\J_{M-1}-1) \right).
\eqen

Now fix $\zeta \in (0,1/2)$, and for $m\in\N$ let
\eqbn
E_m \colonequals \left\{ \cN_{\db}\left(X(-\J_m + 1, -\J_{m-1}-1) \right) \geq \cN_{\ho |\co}\left(X(-\J_m + 1, -\J_{m-1}-1) \right) \vee m^\zeta \right\}.
\eqen
Either $\cN_{\db}\left(X(-\J_M + 1, -\J_{M-1}-1) \right) < M^\zeta$ or $E_M$ occurs.
Therefore,
\alb
\cC \left(X(-\J_M, -1 ) \right)
&\leq \cN_{\db}\left(X(-\J_M + 1, -\J_{M-1}-1) \right) + 1 \\
&\leq  M^\zeta + \cN_{\db}\left(X(-\J_M + 1, -\J_{M-1}-1) \right) \one_{E_M} + 1 \\
&\leq \J_M^\zeta  + \sum_{m=1}^\infty \cN_{\db}\left(X(-\J_m + 1, -\J_{m-1}-1) \right) \one_{E_m}  +1.
\ale
By Lemma~\ref{prop-J-moment} we know $\E(\J_M^\zeta) <\infty$,
so to complete the proof it suffices to show
\eqb \label{eqn-overshoot-tail-event}
\sum_{m=1}^\infty \E\left( \cN_{\db}\left(X(-\J_m + 1, -\J_{m-1}-1) \right) \one_{E_m} \right) <\infty.
\eqe

Recall that the words $X_{-\J_m} \cdots X_{-\J_{m-1}-1}$ are i.i.d.\ with the same law as $X_{-\J} \cdots X_{-1}$.
For $B>0$, Lemma~\ref{prop-few-SD} and a union bound over all $n \in [1, B]_\Z$ yields
\eqbn
\PP\left( \cN_{\db}\left(X(-\J+1,-1) \right) \geq \cN_{\ho|\co}\left(X(-\J+1,-1)\right) \vee A,\; \J \leq B \right) \leq c_0 B e^{-c_1 A}
\eqen
for constants $c_0, c_1 > 0$ depending only on $p$ and $q$. Lemma~\ref{prop-J-moment} and the Chebyshev inequality together imply that $\PP(\J>B) = \PP(\J^\zeta>B^\zeta) \leq \E(\J^\zeta) B^{-\zeta}$.  Thus
\eqbn
\PP\left( \cN_{\db}\left(X(-\J+1,-1) \right) \geq \cN_{\ho|\co}\left(X(-\J+1,-1)\right) \vee A \right) \leq c_0 B e^{-c_1 A} + \E(\J^\zeta) B^{-\zeta}\,,
\eqen
and since $B$ was arbitrary, we choose $B=\exp[c_1 A/(1+\zeta)]$.  Then
\begin{align*}
\E&\left( \cN_{\db}\left(X(-\J_m + 1, -\J_{m-1}-1) \right) \one_{E_m} \right)\\
&=
\E\left( \cN_{\db}\left(X(-\J+1, -1) \right) \one\left\{ \cN_{\db}\left(X(-\J+1, -1) \right) \geq \cN_{\ho |\co} \left(X(-\J+1,-1) \right) \vee m^\zeta \right\} \right) \\
&= \sum_{k\geq m^\zeta} k\times\PP\left( k=\cN_{\db}\left(X(-\J+1, -1) \right) \geq \cN_{\ho |\co} \left(X(-\J+1,-1) \right) \vee m^\zeta \right)\\
&\leq \sum_{k\geq m^\zeta} k\times\PP\left( \cN_{\db}\left(X(-\J+1, -1) \right) \geq \cN_{\ho |\co} \left(X(-\J+1,-1) \right) \vee k \right)\\
&\leq \sum_{k\geq m^\zeta} k\times(c_0+\E[\J^\zeta]) \times\exp[-(c_1 \zeta/(1+\zeta)) k]\\
&\leq (m^\zeta+\text{const}) \times \text{const} \times \exp[-\text{const}\times m^\zeta]\,,
\end{align*}
which is summable in $m$, establishing \eqref{eqn-overshoot-tail-event}.
\end{proof}

The next two lemmas correspond to \cite[Lem.~3.5]{shef-burger}.
However, slightly more work is needed to prove Lemma~\ref{prop-J-count-mean} below in our setting
because the word $X(-\J,-1)$ can contain more than one burger,
so with $\J_M$ as in Lemma~\ref{prop-J-moment}, we might have $\cC(X(-\J_M,-1)) > 1$.

\begin{lem} \label{E[M]}
Let $M$ be the smallest $m\in\N$ for which $\cC\left(X(-\J_m,-1) \right) \geq 1$, as in Lemma~\ref{prop-J-moment}.
Then $\E[M]=\infty$.
\end{lem}
\begin{proof}
With $P$ as in Lemma~\ref{prop-mean-infty}, i.e., the smallest $j\in\N$ for which $\cC(X(-j,-1))=1$,
\eqb \label{eqn-X(-P,-1)}
|X(-P,-1)| = 2\cN_{\hb|\cb|\db}\left(X(-P,-1)\right) + 1\,.
\eqe 
For $\ep, \zeta \in (0,1/2)$ and the events $F_n = F_n(\ep , n^\zeta)$ and the random variable $\Xi$ in~\eqref{eqn-no-F-sum},
\[
\cN_{\db|\so}\left(X(-P,-1)\right) \leq
\begin{cases}
 \ep |X(-P,-1)| + \ep P^\zeta & \text{if $F_P$ does not occur}\\ \Xi & \text{if $F_P$ occurs}\,.\end{cases}
\]
By this and~\eqref{eqn-X(-P,-1)},
\eqb \label{eqn-X(-P,-1)-again}
 |X(-P,-1)| 
 \leq 2 \cN_{\hb|\cb}\left(X(-P,-1)\right) + 2 \ep P^\zeta + 2 \ep |X(-P,-1)| + 2\Xi + 1 .
\eqe 
Since $\cC(X(-\J_M, -1)) \geq 1$, we have $P \leq \J_M$.
Since $\E(\Xi) < \infty$ and $\E(P^\zeta) \leq \E(\J_M^\zeta) < \infty$,
$\E(|X(-P,-1)|) = \infty$ by Lemma~\ref{prop-mean-infty}, and $\ep  <1/2$, 
we deduce from~\eqref{eqn-X(-P,-1)-again} that
\[ \E\left(\cN_{\hb|\cb}\left(X(-P,-1)\right)\right) = \infty\,.
\]
Since $P \leq \J_M$,
\[\cN_{\hb|\cb}(X(-P,-1)) \leq \cN_{\hb|\cb}(X(-\J_M,-1))\,.\]
Since each symbol in $X(-\J_m, -\J_{m-1}-1)$ is identified,
\[ \cN_{\hb|\cb}(X(-\J_M,-1)) \leq \sum_{m=1}^M \cN_{\hb|\cb}(X(-\J_m, -\J_{m-1}-1)\,.\]
The summands are i.i.d., and have finite expectation by Proposition~\ref{prop-J-finite}.
But the left hand side has infinite expectation, so by Wald's equation, $\E[M]=\infty$.
\end{proof}

\begin{lem} \label{prop-J-count-mean}
\eqbn
\E\left( \cC(X(-\J,-1)) \right) = 0\,.
\eqen
\end{lem}

\begin{proof}
Write $\alpha = \E\left(\cC(X(-\J,-1) \right)$.
Observe that by Proposition~\ref{prop-J-finite},
\[
0 \leq \alpha \leq \E\left(|\cC(X(-\J,-1)| \right) \leq \E\left(|X(-\J,-1)| \right) < \infty\,.
\]
The strong Markov property implies that the words $X_{-\J_m} \cdots X_{-\J_{m-1}-1}$ for $m\in\N$ are i.i.d., and each has the same law as $X_{-\J} \cdots X_{-1}$. By Lemma~\ref{prop-J-basic}, none of the reduced words $X(-\J_m, -\J_{m-1}-1)$ contains an unidentified $\db$ or $\so$. By definition of $\alpha$, we find that
\eqbn
A_m \colonequals \cC\left(X(-\J_m,-1) \right) - \alpha m = \sum_{k=1}^m \cC\left(X(-\J_k, -\J_{k-1}-1)\right) - \alpha m
\eqen
is a martingale in $m$.

Let $M$ be the smallest $m\in\N$ for which $\cC\left(X(-\J_m,-1) \right) \geq 1$, as in Lemma~\ref{prop-J-moment}.
By the optional stopping theorem, for each $n \in \N$ we have $\E\left(A_{M\wedge n} \right) = 0$. Since $A_{M\wedge n} \leq \cC(X(-\J_M, -1))$ and the latter quantity has finite expectation by Lemma~\ref{prop-overshoot-finite}, it follows from Fatou's lemma that
\eqbn
0 \leq \E(A_M) \leq \E\left(\cC(X(-\J_M, -1) )\right).
\eqen
In particular $\E(A_M)\geq 0$ implies
\eqbn
\alpha \E(M) \leq \E\left(\cC(X(-\J_M,-1)) \right).
\eqen
By Lemma~\ref{prop-overshoot-finite} $\E\left(\cC(X(-\J_M,-1)) \right)<\infty$ and by Lemma~\ref{E[M]} $\E(M) = \infty$,
so $\alpha\leq 0$.  We already showed in Proposition~\ref{prop-J-finite} that $\alpha\geq 0$, so in fact $\alpha=0$.
\end{proof}

The following corollary is the reason why we know the variance and covariance of $Z$ in Theorem~\ref{thm-variable-SD} in the case when $z=1$. 

\begin{cor}
If $q = 0$ then $\CHI = 2$.
\end{cor}
\begin{proof}
When $q = 0$ the word $X(-\J,-1)$ contains exactly one burger. Hence in this case $|X(-\J,-1)| = 2- \cC(X(-\J,-1))$. Therefore Lemma~\ref{prop-J-count-mean} implies $\CHI = 2$ in this case.
\end{proof}

\begin{lem} \label{prop-J-limit}
\eqbn
\lim_{n\rta\infty} \E\left(|X(-n,-1)| \times \one_{(\J>n)} \right) = 0.
\eqen
\end{lem}
\begin{proof}
By the optional stopping theorem, for each $n\in\N$,
\eqb \label{eqn-J-stop-decomp}
0 = \E\left( \cC(X(-\J \wedge n,- 1)) \right) = \E\left( \cC(X(-\J,-1)) \one_{(\J\leq n)} \right) + \E\left( \cC(X(-n,-1)) \one_{(\J >  n)} \right).
\eqe
Since
\[|\cC(X(-\J,-1)) \one_{(\J\leq n)}| \leq |\cC(X(-\J,-1))| \leq |X(-J,-1)|\,,
\]
and by Proposition~\ref{prop-J-finite} $\E(|X(-J,-1)|)<\infty$,
by dominated convergence (and Lemma~\ref{prop-J-count-mean}),
\eqbn
\lim_{n\rta\infty} \E\left( \cC(X(-\J,-1)) \one_{(\J\leq n)} \right) = \E\left(\cC(X(-\J,-1)) \right) = 0\,.
\eqen
It therefore follows from~\eqref{eqn-J-stop-decomp} that
\eqb \label{eqn-C-on-J}
\lim_{n\rta \infty} \E\left( \cC(X(-n,-1)) \one_{(\J >  n)} \right) = 0.
\eqe
Now fix $\ep,\zeta \in (0,1/2)$ and let $F_n = F_n(\ep, n^\zeta)$ be as in Lemma~\ref{prop-few-SD} with $X_{-n}\cdots X_{-1}$ in place of $X_1\cdots X_n$, as in the proof of Proposition~\ref{prop-J-finite}.
By~\eqref{eqn-length-on-F} and since $|X(-n,-1)|  \leq n$, 
\eqb \label{eqn-length-on-J-decomp}
|X(-n,-1)| \one_{(\J > n)} \leq - \left(1 + \frac{2\ep}{1-\ep}\right) \cC(X(-n,-1)) \one_{(\J >  n)} + \frac{2\ep}{1-\ep} n^\zeta \one_{(\J > n)} +   n \one_{F_n}.
\eqe
By~\eqref{eqn-C-on-J}, the expectation of the first term on the right in~\eqref{eqn-length-on-J-decomp} tends to 0 as $n\rta\infty$.
By Lemma~\ref{prop-few-SD}, $\lim_{n\rta\infty} n \PP(F_n) = 0$. By Lemma~\ref{prop-overshoot-finite}, for each $\zeta'\in (\zeta,1/2)$ we have $\E(\J^{\zeta'}) \leq \E(\J_M^{\zeta'}) < \infty$, so by Chebyshev's inequality $\PP(\J > n) \leq \E(\J_M^{\zeta'})/n^{\zeta'}$.
By combining these observations with~\eqref{eqn-length-on-J-decomp}, we obtain the statement of the lemma.
\end{proof}

\subsection{Variance of the discrepancy between burger types}
\label{sec-var-bound}

In this subsection we obtain an asymptotic formula for $\op{Var}\cQ(X'(1,n))$, where here $\cQ$ is as in Definition~\ref{def-theta-count} and $X'$ is as in Definition~\ref{def-X-identification}.
This formula will be used to obtain the variance and covariance for the limiting Brownian motion in Theorem~\ref{thm-variable-SD}.
In particular, we prove Proposition~\ref{prop-var-limit} below.
The proof is similar to the argument found in~\cite[\S~3.1]{shef-burger}, but unlike in~\cite[\S~3.1]{shef-burger}, all of the assumptions needed to make the argument work have already been proven.
Recall from Proposition~\ref{prop-J-finite} that $\CHI$ is finite.

\begin{prop} \label{prop-var-limit}
Let $\CHI$ be as in~\eqref{eqn-chi-def}. Then
\eqbn
\lim_{n\rta\infty} n^{-1} \op{Var}\left(\cQ(X'(-n,-1) ) \right) = 1 + (p+q) \CHI.
\eqen
\end{prop}

\begin{proof}
By Lemma~\ref{prop-J-basic}, the word $X(-\J,-1)$ is equal to $X'(-\J,-1)$ and consists of either $\hb$'s and $\co$'s (if $X_{-\J} = \hb$) or $\cb$'s and $\ho$'s (if $X_{-\J} = \cb$). Therefore,
\eqb \label{eqn-J-discrep}
 \cQ\left(X(-\J,-1)\right) = \cQ\left(X'(-\J,-1)\right) = \pm |X(-\J,-1)|
\eqe
where the sign is positive if $X_{-\J} = \hb$ and negative if $X_{-\J} = \cb$. We observe that $X_0$ is independent from $X(-\J,-1)$, and that $X'_0$ is determined by $X_0$ on the event $\left\{X_0 \not\in\left\{\db, \so\right\} \right\}$. Therefore,
\[
\E\left(\cQ(X'_0) \cQ(X(-\J,-1)) \one_{\left( X_0 \not\in\left\{\db, \so\right\} \right)} \right) = 0.
\]
If on the other hand $X_0 \in\left\{\db, \so\right\}$, if $X_{-\J} = \hb$ then $X_0'\in\{\hb,\co\}$, and if $X_{-\J} = \cb$ then $X_0'\in\{\cb,\ho\}$.
Therefore, if $X_0 \in\left\{\db, \so\right\}$ then $\cQ(X'_0)$ has the same sign as $\cQ(X(-\J,-1))$, so
\begin{align} \label{eqn-discrep-mean-J}
\E\left(\cQ(X'_0) \cQ(X(-\J,-1)) \right)
&=\E\left(\cQ(X'_0) \cQ(X(-\J,-1)) \one_{(X_0\in\{\db,\so\})} \right) \notag\\
 &= \PP\!\left( X_0 \in \left\{\db, \so\right\} \right) \E\left(|X(-\J,-1)| \right) = \frac{\CHI (p+q)}{2}.
\end{align}

We next observe that $X_0'$ is determined by $X_{-\J} \cdots X_{-1}$ and $X_0$, so by the strong Markov property, for each $n\in\N$ it holds that $X_0'$ is conditionally independent from $X'_{-n} \cdots X'_{-\J-1}$ given $X'_{-\J}\cdots X'_{-1}$ (here we set $X(-n,-\J-1) = \emptyset$ if $n \leq \J$, so that the assertion holds vacuously in this case).  By symmetry $\cQ(X'(-n,-\J-1))$ has zero conditional mean given $X_{-\J} \cdots X_{-1}$, so
\eqb \label{eqn-after-J}
\E\left( \cQ(X_0') \cQ(X'(-n,-\J-1)) \;|\; X_{-\J} \cdots X_{-1}\right)
= 0.
\eqe
Therefore,
\eqb \label{eqn-discrep-mean-n}
\E\left( \cQ(X_0') \cQ(X'(-n, -1)) \right)
= \E\left( \cQ(X_0') \cQ(X(-\J, -1)) \one_{\J\leq n} \right) + \E\left( \cQ(X_0') \cQ(X'(-n, -1)) \one_{\J >  n} \right).
\eqe
By~\eqref{eqn-J-discrep},~\eqref{eqn-discrep-mean-J}, and dominated convergence (with $|X(-\J,-1)|$ as the dominator; recall Proposition~\ref{prop-J-finite}) we find that the first term on the right in~\eqref{eqn-discrep-mean-n} tends to $\CHI(p+q)/2$ as $n\rta\infty$. The absolute value of the second term is at most $\E\left(|X(-n,-1)| \one_{(\J > n)} \right)$, which tends to 0 by Lemma~\ref{prop-J-limit}. By translation invariance, we therefore have
\alb
\op{Var}\left(\cQ(X'(1,n)) \right)
&= \E\left( \cQ(X'(-n,-1))^2 \right) \\
&= \sum_{i=1}^n \E\left( \cQ(X_i')^2 \right) + 2 \sum_{i=2}^n \E\left(\cQ(X_i') \cQ(X'(1,i-1))\right) \\
&= n + 2 \sum_{i=2}^n \E\left(\cQ(X_0') \cQ(X'(-i+1, -1))\right) \\
&= \ \left(1 + \CHI (p+q) \right) n + o(n).\qedhere
\ale
\end{proof}

\subsection{Expected length of the reduced word}
\label{sec-moment-bound}

In this subsection we estimate the expectations of several quantities related to the reduced words $X(1,n)$ and $X'(1,n)$ for $n\in\N$ (recall~\eqref{eqn-X(a,b)}).
As one might expect due to the diffusive scaling for $Z^n$ in~\eqref{eqn-Z^n-def}, these quantities will typically be of order $n^{1/2}$. 
We first prove in Lemma~\ref{prop-length-mean-upper'} an upper bound for the length of the latter word, which may be shorter than $|X(1,n)|$ since there could be $\db$'s in $X_1\dots X_n$ which are identified by burgers in $\dots X_{-1} X_0$ but matched to orders in $X_1\dots X_n$. 
In Lemma~\ref{prop-length-mean-upper}, we transfer this to an upper bound for $|X(1,n)|$ using Lemma~\ref{prop-few-SD}. 
We then use a comparison to simple random walk on $\Z^2$ (via Lemma~\ref{prop-mean-mono}) to prove a corresponding lower bound for the expected number of burgers and orders in $X(1,n)$ (Lemma~\ref{prop-H-mean}).

\begin{lem} \label{prop-length-mean-upper'}
For $n\in\N$, we have (using the notation $\preceq$ from Section~\ref{sec-basic}),
\eqbn
\E\left(|X'(1,n)| \right) \preceq n^{1/2}\,.
\eqen
\end{lem}

\begin{proof}
By the symmetry between hamburgers and cheeseburgers, $\E\left(\cQ(X'(1,n))\right)=0$, so by Proposition~\ref{prop-var-limit} and translation invariance, for each $n\in\N$ we have $\E\left( \cQ(X'(1,n))^2\right)= \op{Var}\left(\cQ(X'(-n,-1) ) \right)\preceq n$.
Since $n\mapsto \cC( X'(1,n))$ is a simple random walk, $\E\left( \cC(X'(1,n))^2 \right) = n$. With $\dpt(X'(1,n))$ as in Definition~\ref{def-theta-count}, $\dpt(X'(1,n)) = \frac12 \left( \cQ(X'(1,n)) + \cC(X'(1,n))\right)$.
By a union bound and the Chebyshev inequality, we infer
\eqb \label{eqn-d-tail}
\PP\!\left(|\dpt(X'(1,n))| \geq k \right) \preceq n/k^2,\quad\quad \forall n,k \in \N.
\eqe

For $k\in\N$, let $K_k$ be the smallest $i \in \N$ for which $X(-i,-1)$ contains at least $k$ hamburgers.
Then $X_{-K_k}$ is a $\hb$ without a match in $X_{-K_k} \dots X_{-1}$, so each $\db$ or $\so$ in $X(-K_k +1 , -1)$ must be identified and there are no hamburger orders in $X(-K_k+1,-1)$.
Consequently, the word $X(-K_k, -1)$ contains at least $k$ hamburgers, no unidentified $\db$'s or $\so$'s, and no orders other than cheeseburger orders. Therefore,
\eqbn
\dpt\left(X(-K_k, -1) \right) = \dpt\left(X'(-K_k, -1) \right) \geq k.
\eqen
It follows that if $K_k \leq n$, then either
\eqbn
\dpt\left(X'(-n, -K_k- 1) \right) \leq -k/2 \quad \op{or}\quad \dpt\left(X'(-n,-1) \right) \geq k/2.
\eqen
Since $K_k$ is a backward stopping time for the word $X$, we infer from the strong Markov property and translation invariance that the conditional law of $\dpt\left(X'(-n, -K_k- 1) \right)$ given $X_{-K_k} \cdots X_{-1}$ is the same as the law of $\dpt\left(X'(1, n - K_k) \right)$. By~\eqref{eqn-d-tail} and the union bound,
\eqbn
\PP\!\left( K_k \leq n \right) \preceq n/k^2\,,
\eqen
and hence
\eqb \label{eqn-H-tail}
\PP\!\left(\cN_{\hb}\left(X'(-n,-1) \right) \geq k \right) \preceq n/k^2,\quad \forall k,n \in \N.
\eqe
By combining~\eqref{eqn-d-tail} and~\eqref{eqn-H-tail} and noting that $\cN_{\hb|\ho}(x) \leq 2 \cN_{\hb}(x)  + |\dpt( x )|$ for every word $x$, we get
\eqbn
\PP\!\left(\cN_{\hb|\ho}\left(X'(-n,-1) \right) \geq k \right) \preceq n/k^2,\quad \forall k,n \in \N.
\eqen
By symmetry, the analogous estimate holds with $\cb$ and $\co$ in place of $\hb$ and $\ho$.
Since $|X'(-n,-1)|=\cN_{\hb|\cb|\ho|\co}(X'(-n,-1))$, a union bound therefore implies
\eqbn
\PP\!\left( \left| X'(-n,-1) \right| \geq k\right) \preceq n/k^2,\quad \forall k,n \in \N.
\eqen
Hence
\eqbn
\E\left( \left|X'(-n,-1)\right| \right)
= \sum_{k=1}^\infty \PP\!\left( \left| X'(-n,-1) \right| \geq k\right)
\preceq \int_1^\infty (1 \wedge (n/k^2)) \, dk \preceq n^{1/2}\,,
\eqen
which finishes the proof in view of translation invariance.
\end{proof}

We now estimate the expectation of $|X(1,n)|$, which may be larger than the expectation of $|X'(1,n)|$ since some duplicate burgers with no match in $X_1 \cdots X_n$ may correspond to hamburgers or cheeseburgers in $X'$ which have a match in $X'_1 \cdots X'_n$.

\begin{lem} \label{prop-length-mean-upper}
\eqb \label{eqn-length-mean-upper}
\E\left(|X(1,n)|\right) \preceq n^{1/2}\,.
\eqe
and
\eqb \label{eqn-SD-mean-upper}
\E\left( \cN_{\db|\so}(X(1,n))\right) = o(n^{1/2})\,,
\eqe
as $n\in\N$ tends to infinity.
\end{lem}
\begin{proof}
If $i\in [1,n]_\Z$ is such that $X_i$ does not have a match in $X_1\cdots X_n$ but $X_i'$ has a match in $X_1'\cdots X_n'$, then either $X_i =\db$ or $X_i$ is matched to a $\db$ in the word $X_1\cdots X_n$. Therefore,
\eqb \label{eqn-reduced-compare}
|X'(1,n)| \leq |X(1,n)| \leq |X'(1,n)| + 2 \cN_{\db}(X(1,n))\,.
\eqe

Now fix $\ep, \zeta \in (0,1/2)$ and for $n\in\N$ let $F_n = F_n(\ep, n^\zeta)$ be the event defined in~\eqref{eqn-few-SD-event}.
On the event $F_n^c$, we have
\eqbn
\cN_{\db|\so}(X(1,n)) \leq \ep |X(1,n)|   + \ep n^\zeta  \leq \ep |X'(1,n)| + 2 \ep \cN_{\db|\so}(X(1,n))  + \ep n^\zeta\,,
\eqen
where we used~\eqref{eqn-reduced-compare} in the second inequality.
After re-arranging this inequality, and considering also the possibility that $F_n$ occurs, we get
\eqb \label{eqn-SD-length-compare}
  \cN_{\db|\so}(X(1,n)) \leq \frac{\ep}{1-2\ep} |X'(1,n)|     + \frac{\ep}{1-2\ep}  n^\zeta + n \one_{F_n}.
\eqe
Combining~\eqref{eqn-SD-length-compare}, the bound $\E(|X'(1,n)|)\preceq n^{1/2}$ from Lemma~\ref{prop-length-mean-upper'}, the exponential decay of $\E(n\one_{F_n})$ from Lemma~\ref{prop-few-SD}, and the fact that $\ep>0$ can be made arbitrarily small, we easily obtain~\eqref{eqn-SD-mean-upper}.
We obtain~\eqref{eqn-length-mean-upper} from~\eqref{eqn-SD-mean-upper}, \eqref{eqn-reduced-compare} and Lemma~\ref{prop-length-mean-upper'}.
\end{proof}

\begin{lem} \label{prop-H-mean}
For $n\in\N$,
\eqb \label{eqn-H-mean}
\E\left( \cN_{\hb}\left(X (1,n) \right) \right) \asymp \E\left( \cN_{\ho}\left(X (1,n) \right) \right) \asymp  n^{1/2}\,.
\eqe
\end{lem}
\begin{proof}
The upper bounds for both expectations in~\eqref{eqn-H-mean} follow from Lemma~\ref{prop-length-mean-upper}, so we only need to prove the lower bounds.

Recall that $\PP^{0,0}$ denotes the law of $X$ with $p = q = 0$ and $\E^{0,0}$ is the corresponding expectation.
By Lemma~\ref{prop-mean-mono},
\eqb \label{eqn-burger-mean-compare}
\begin{aligned}
\E \left(\cN_{\hb|\cb|\db}(X(1,n)) \right) &\geq \E^{0,0}\left(\cN_{\hb|\cb}(X(1,n))\right)\\
\E \left(\cN_{\ho|\co|\so}(X(1,n)) \right) &\geq \E^{0,0}\left(\cN_{\ho|\co}(X(1,n))\right)\,.
\end{aligned}
\eqe
If all symbols in $X_{-n}\cdots X_{-1}$ are identified, then
\begin{align*}
 \cN_{\hb}(X(-n,-1)) &= \max_{1\leq i\leq n} \dpt(X(-i,-1)) \\
 \cN_{\ho}(X(-n,-1)) &= \max_{1\leq i\leq n+1} -\dpt(X(-n,-i))\,.
\end{align*}
Under $\PP^{0,0}$, the maps $i\mapsto\dd(X(-i,-1))$ and $i\mapsto\dd(X(-n,-i))$ are two-dimensional simple random walks,
so we deduce (using
e.g., Donsker's invariance principle and Fatou's lemma together with the fact that Brownian motion has a well-defined running supremum process which is positive at any given time)
\begin{equation} \label{E00>=n^1/2}
\begin{aligned}
 \E^{0,0}(\cN_{\hb}(X(-n,-1))) &\succeq n^{1/2} \\
 \E^{0,0}(\cN_{\ho}(X(-n,-1))) &\succeq n^{1/2}\,.
\end{aligned}
\end{equation}
By symmetry
 $\E\left(\cN_{\hb}(X(1,n))\right) = \E\left( \cN_{\cb}(X(1,n))\right)$ and
 $\E\left(\cN_{\ho}(X(1,n))\right) = \E\left( \cN_{\co}(X(1,n))\right)$,
 and by \eqref{eqn-SD-mean-upper} of Lemma~\ref{prop-length-mean-upper}
$\E\left(\cN_{\db|\so}(X(1,n))\right) = o(n^{1/2})$,
which combined with \eqref{eqn-burger-mean-compare} and \eqref{E00>=n^1/2}
gives the lower bounds
 $\E \left(\cN_{\hb}(X(1,n)) \right) \succeq n^{1/2}$ and
 $\E \left(\cN_{\ho}(X(1,n)) \right) \succeq n^{1/2}$.
\end{proof}

\subsection{Tail bound for the length of the reduced word}
\label{sec-word-length}

In this subsection we prove the following analogue of~\cite[Lem.~3.13]{shef-burger}, which will be used to prove tightness of the sequence of paths $Z^n$ defined in~\eqref{eqn-Z^n-def} in the proof of Theorem~\ref{thm-variable-SD}.

\begin{prop} \label{prop-length-sup}
There are constants $a_0,a_1> 0$ such that for each $n\in\N$ and $r > 0$,
\eqb  \label{eqn-length-max-all}
\PP\!\left( \max_{\substack{i,j\in [1,n]_\Z\\1\leq i\leq j\leq n}} |X(i,j)| > r n^{1/2} \right) \leq a_0 e^{-a_1 r}
 \,.
\eqe
\end{prop}

To prove Proposition~\ref{prop-length-sup}, we will study the times at which unmatched hamburgers are added when we read the word backwards. 
The increments of $X$ between these times are i.i.d., and the number of $\db$'s which are identified at each of these times (some of which also correspond to unmatched hamburgers in our word) can be bounded using Lemma~\ref{prop-few-SD} (c.f.\ Lemma~\ref{prop-D-count-finite}). Using a lower bound for the probability that a reduced word of length $n$ contains no hamburgers (Lemma~\ref{prop-J^H-tail}) and Chernoff's inequality, we get an upper tail bound for the number of hamburgers in $X(-n,-1)$.
By symmetry, we also have an analogous bound for the number of cheeseburgers in $X(-n,-1)$. 
Since the difference $\cC(X(-n,-1))$ between the number of burgers and the number of orders in $X(-n,-1)$ evolves as a simple random walk on $\Z$ and by another application of Lemma~\ref{prop-few-SD}, this will be enough to prove Proposition~\ref{prop-length-sup}.

\begin{lem} \label{prop-J^H-tail}
Let $\Jh$ be the smallest $j\in\N$ for which $X(-j,-1)$ contains a hamburger. Then
\eqb \label{eqn-J^H-tail}
\PP\!\left(\Jh > n \right) \asymp   n^{-1/2}
\eqe
with the implicit constant depending only on $p$.
\end{lem}
\begin{proof}
For $n\in\N\cup\{0\}$, let $E_n$ be the event that $X(1,n)$ contains no hamburgers (recall that $X(1,0) = \emptyset$). By translation invariance,
\eqb \label{eqn-J^H-event-compare}
\PP\!\left( E_n \right) = \PP\!\left( \Jh > n \right)\,.
\eqe
In particular, $n\mapsto \PP(E_n)$ is non-increasing.

Suppose $i \in [1,n]_\Z$.  If $X_i$ identifies to $\ho$ in $X_1\cdots X_n$ and has no match in $X_1 \dots X_{i-1}$, then $E_{i-1}$ occurs and $X_i\in\{\ho,\so\}$.  On the other hand, if $E_{i-1}$ occurs, then by independence of the symbols of $X$, it holds with conditional probability $\frac{1-p}{4}$ that $X_i = \ho$, in which case $X_i$ does not have a match in $X_1 \cdots X_i$.
Therefore,
\eqbn
 \E\left(\cN_{\ho}(X(1,n)) \right) \leq \sum_{i=0}^{n-1} \PP\!\left( E_i \right)  \leq \frac{4}{1-p} \E\left(\cN_{\ho}(X(1,n)) \right) .
\eqen
By Lemma~\ref{prop-H-mean} we can find a constant $C>1$ such that for each $n\in\N$
\eqb \label{eqn-J^H-sum}
C^{-1} n^{1/2} \leq \sum_{i=1}^{n-1} \PP\!\left( E_i \right) \leq C n^{1/2}.
\eqe
By monotonicity of $\PP(E_n)$, we immediately obtain
\eqbn
n \PP(E_n) \leq   \sum_{i=1}^{n-1} \PP( E_i)  \leq C n^{ 1/2} .
\eqen
Furthermore,
\eqbn
 4 C^2 n  \PP(E_n) \geq   \sum_{i=n}^{\lceil 4 C^2 n \rceil-1} \PP(E_i) \geq  2 C n^{1/2} - C n^{1/2} = C n^{1/2} .
\eqen
Combining these two relations with~\eqref{eqn-J^H-event-compare} yields~\eqref{eqn-J^H-tail}.
\end{proof}

 \begin{lem} \label{prop-D-count-finite}
Let $\Jh$ be the smallest $j \in \N$ for which $X(-j,-1)$ contains a hamburger.
There are constants $a_0 , a_1 > 0$ depending only on $p$ such that for each $m \in\N$, we have
\eqb \label{eqn-D-count-finite}
\PP\!\left( \cN_{\db|\so}\left(X(-\Jh+1,-1)\right) > m \right) \leq a_0 e^{-a_1 m} .
\eqe
\end{lem}
\begin{proof}
We observe that $\cN_{\db}\left(X(-\Jh+1,-1)\right) \geq \cN_{\ho}\left(X(-\Jh+1,-1)\right)$; indeed, otherwise it is not possible for all of the $\ho$'s in $X(-\Jh+1,-1)$ to be fulfilled in $X_{-\Jh} \dots X_{-1}$ while still leaving a leftover $\hb$. Now let $c_0 , c_1 > 0$ be as in Lemma~\ref{prop-few-SD} with $\ep = 1$. By that lemma and a union bound,
\eqbn
\PP\!\left( \cN_{\db|\so}\left(X(-\Jh+1,-1)\right) \geq m , \, \Jh \leq   e^{c_1 m/2} \right) \leq c_0 e^{-c_1 m/2} .
\eqen
On the other hand, by Lemma~\ref{prop-J^H-tail} we have
\eqbn
\PP\! \left( \Jh >  e^{c_1 m/2} \right) \preceq e^{-c_1 m/4} .
\eqen
Combining these estimates yields~\eqref{eqn-D-count-finite}.
\end{proof}

\begin{proof}[Proof of Proposition~\ref{prop-length-sup}]
Let $\Jh_0 = 0$ and for $m\in\N$ inductively let $\Jh_m$ be the smallest $j \geq \Jh_{m-1}$ for which $X(-j,-\Jh_{m-1}-1)$ contains a hamburger. Then $\Jh_1$ is the same as the time $\Jh$ from Lemma~\ref{prop-J^H-tail} and by the strong Markov property the increments $X_{-\Jh_m} \cdots X_{-\Jh_{m-1}-1}$ for $m\in\N$ are i.i.d.  For $m\in\N$, let
\eqbn
H_m \colonequals \cN_{\hb}\left(X(-\Jh_m, -\Jh_{m-1}-1)\right) = 1 + \cN_{\db}\left(X(-\Jh_m +1, -\Jh_{m-1}-1)\right).
\eqen
Since none of the reduced words $X(-\Jh_m,-\Jh_{m-1}-1)$ contain $\ho$'s,
\eqb \label{eqn-H-sum}
\cN_{\hb}\left( X(-\Jh_m, -1) \right) = \sum_{k=1}^m H_k \,.
\eqe

By Lemma~\ref{prop-D-count-finite}, for some positive number $\beta>0$ (depending only on $p$) $\E(e^{\beta H_k})<\infty$,
and since the $H_k$'s are i.i.d., Chernoff's bound
implies that there are positive numbers $\wt c_0, \wt c_1  > 0$ such that for each $m \in \N$,
\eqb \label{eqn-H-sum-tail}
\PP\!\left( \sum_{k=1}^m H_k \geq \wt c_0 m \right) \leq  e^{-\wt c_1 m}.
\eqe

By Lemma~\ref{prop-J^H-tail}, we can find a constant $c>0$ such that for each $n, m \in\N$,
\eqbn
\PP\!\left( \Jh_m - \Jh_{m-1} > n \right) \geq c n^{-1/2}.
\eqen
Since the increments $\Jh_m - \Jh_{m-1}$ are i.i.d., we infer that for each $n, m \in \N$,
\eqb \label{eqn-J-exp-bound}
\PP\!\left( \Jh_m \leq n \right) \leq
\PP\!\left( \Jh_k - \Jh_{k-1} \leq n,\, \forall k \leq m \right) \leq \left(1 - c n^{-1/2} \right)^m
\leq \exp[-c m/n^{1/2}]\,.
\eqe

Recall that  $\cN_{\hb}(X(-j,-1))$ is monotone increasing in $j$.
If $\Jh_{m} \geq n$ and $\cN_{\hb}\left( X(-\Jh_{m}, -1) \right) \leq \wt c_0 m$, then $\cN_{\hb}(X(-j,-1)) \leq \wt c_0 m$ for each $j \in [1,n]_\Z$. By taking $m = \lfloor r n^{1/2} / \wt c_0\rfloor$ and applying~\eqref{eqn-H-sum}, ~\eqref{eqn-H-sum-tail}, and~\eqref{eqn-J-exp-bound}, we find that for each $n\in\N$,
\eqb \label{eqn-burger-tail}
\PP\!\left( \max_{j\in [1,n]_\Z} \cN_{\hb}(X(-j,-1)) >  r n^{1/2} \right) \leq c_0 e^{-c_1 r}
\eqe
for appropriate $c_0, c_1 > 0$ independent of $r$ and $n$.
By symmetry, the analogous estimate holds with $\cb$ in place of $\hb$.

Since $j\mapsto \cC(X(-j,-1))$ is a simple random walk, we have (see e.g.~\cite[Prop.~2.1.2b]{lawler-limic-walks})
\eqb \label{eqn-net-count-tail}
\PP\!\left( \max_{j \in [1,n]_\Z} |\cC(X(-j,-1))| > r n^{1/2} \right) \leq b_0 e^{-b_1 r^2}
\eqe
for universal constants $b_0, b_1 > 0$. By Lemma~\ref{prop-few-SD} (applied with $\ep = \frac12$ and $A = \text{const}\times r n^{1/2}$) and the union bound, except on an event of probability $\leq \exp(-\Theta(r))$,
\eqb \label{eqn-D-tail}
\cN_{\db}\left(X(-j, -1) \right) \leq \frac12 \text{const} \times r n^{1/2}  - \frac12 \cC(X(-j,-1)) + \frac12 \cN_{\hb|\cb|\db}(X(-j,-1))\,,  \quad \forall j \in [1, n]_\Z .
\eqe
Re-arranging gives
\eqb \label{eqn-D-tail'}
\cN_{\db}\left(X(-j, -1) \right) \leq  \text{const}\times r n^{1/2}  -  \cC(X(-j,-1))  +   \cN_{\hb|\cb }(X(-j,-1))\,,    \quad \forall j \in [1, n]_\Z .
\eqe
By writing $|X(-j,-1)|= 2\cN_{\db}(X(-j,-1))+2\cN_{\hb|\cb}(X(-j,-1))-\cC(X(-j,-1))$, using the bound~\eqref{eqn-D-tail'},
and the bounds~\eqref{eqn-burger-tail} and~\eqref{eqn-net-count-tail}, we obtain
\eqb  \label{eqn-length-sup-forward}
\PP\!\left( \max_{j\in [1,n]_\Z} |X(-j,-1)| > r n^{1/2} \right) \leq \text{const} \times e^{-\text{const}\times r}
 \,.
\eqe

We now observe that for $1\leq i\leq j\leq n$, each order and each unidentified $\so$ or $\db$ in $X(i,j)$ also appears in $X(i,n)$; and each $\hb$ or $\cb$ in $X(i,j)$ either appears in $X(i,n)$ or is consumed by a unique order in $X(j+1,n)$.  Thus
\eqb  \label{eqn-reduced-length-compare}
|X(i,j)| \leq |X(i,n)|  +  |X(j+1,n)|\,.
\eqe
This bound~\eqref{eqn-reduced-length-compare} together with \eqref{eqn-length-sup-forward} imply \eqref{eqn-length-max-all}.
\end{proof}

\subsection{Convergence to correlated Brownian motion}
\label{sec-variable-SD-proof}

We are now ready to conclude the proof of Theorem~\ref{thm-variable-SD}. We first establish tightness.

\begin{lem} \label{prop-variable-SD-tight}
Suppose we are in the setting of Theorem~\ref{thm-variable-SD}.
The sequence of laws of the paths $Z^n$ for $n\in\N$ is tight in the topology of uniform convergence on compacts of $\R$.
\end{lem}
\begin{proof}
Fix $T \geq 1$ and $\ep > 0$. For $N\in\N$, we cover the time interval $[0,T]$ by $N$ blocks
of the form $[k T/N,(k+2) T/N]$ for $k\in[0,N-1]_\Z$. Note that successive blocks overlap.
Within each block, the path $Z^n$ has (up to rounding error) $2 n T/N$ steps.
Any pair of times $s,t\in[0,T]$ with $|s-t|<T/N$ lie in some common block,
and if $s,t\in\Z/n$ and $s<t$, $\|Z^n(s)-Z^n(t)\|_1$ is bounded by $|X(ns,nt)|$.
Thus Proposition~\ref{prop-length-sup} together with the union bound implies that there exist constants $a_0, a_1 > 0$,
such that for any $n\geq N$ (here we take $n\geq N$ to avoid worrying about rounding error),
\eqbn
\PP\!\left(\sup_{\substack{s,t\in[0,T]\\|s-t|\leq T/N}} \|Z^n(t) - Z^n(s)\|_1 \geq 2^{-m} \right) \leq 2N a_0 \exp\left( - a_1 T^{-1/2} N^{1/2} 2^{-m} \right) \,.
\eqen
By choosing $N=N_{T,\ep,m}$ sufficiently large, depending on $T$, $\ep$, and $m$,
we can make this probability at most $\ep 2^{-m}$ for all $n\geq N_{T,\ep,m}$.
By starting with $\delta_m = T/N_{T,\ep,m}$, and then possibly shrinking $\delta_m$,
we can arrange that
\eqbn
\PP\!\left(\sup_{\substack{s,t\in[0,T]\\|s-t|\leq\delta_m}} \|Z^n(t) - Z^n(s)\|_1 \geq 2^{-m} \right) \leq \ep 2^{-m}
\eqen
for all $n\in\N$, not just $n\geq N_{T,\ep,m}$. By the union bound, we obtain that for each $n\in\N$, it holds except on an event of probability at most $\ep$ that, whenever $m\in\N$ and $s,t \in [0,T]$ with $|t-s| \leq \delta_m$, we have $\|Z^n(t) -Z^n(s)\|_1 < 2^{-m}$. By the Arzel\'a-Ascoli theorem, we obtain tightness of the paths $Z^n|_{[0,\infty)}$ in the topology of uniform convergence on compacts. Tightness of the sequence of the full processes (defined on $\R$) follows from translation invariance.
\end{proof}

\begin{proof}[Proof of Theorem~\ref{thm-variable-SD}]
By Lemma~\ref{prop-variable-SD-tight} and Prokhorov's theorem, for any sequence of $n$'s tending to infinity, there exists a subsequence $n_k$ and a random continuous path $Z = ( U, V) : \R \rta \R^2$ such that, as $k$ tends to infinity, $Z^{n_k}|_{[0,\infty)}$ converges to $Z$ in law in the topology of uniform convergence on compacts.

Next we show that the law of $Z$ is uniquely determined
(independently of the subsequence).
Consider any subsequence $n_k$ for which $Z^{n_k}|_{[0,\infty)}$ converges in law (in the topology of uniform convergence on compacts).
By the Skorokhod representation theorem, we can find a coupling of a sequence of random words $(X^{n_k})$, each with the law of $X$, such that if we define $Z^{n_k}$ with $X^{n_k}$ in place of $X$, then a.s.\ as $k$ tends to infinity, $Z^{n_k}$ converges to $Z$ uniformly on compact subsets of $[0,\infty)$.

Fix real numbers $t_0 < t_2 < \cdots < t_N$. For $j \in [1,N]_\Z$ and $k\in\N$, let
\eqbn
\Upsilon_j^{n_k} \colonequals n_k^{-1/2} \dd \big( X^{n_k}(\lfloor t_{j-1} n_k \rfloor + 1, \lfloor t_j n_k \rfloor ) \big).
\eqen
Observe that $\Upsilon_j^{n_k}$ differs from $Z^{n_k}(t_j) - Z^{n_k}(t_{j-1})$ in either coordinate by at most
\[2 n_k^{-1/2} + n_k^{-1/2} \cN_{\db|\so}(X^{n_k}\left(\lfloor t_{j-1} n_k \rfloor + 1, \lfloor t_j n_k \rfloor \right)).\]
By Lemma~\ref{prop-few-SD} and Proposition~\ref{prop-length-sup}, the latter quantity tends to $0$ in probability as $k$ tends to infinity, and since by the Skorokhod coupling $Z^{n_k}\rta Z$,
in fact $\Upsilon_j^{n_k} \rta  Z(t_j) - Z(t_{j-1})$ a.s.\ for each $j\in [1,N]_\Z$. The random variables $(\Upsilon_j^{n_k} : j\in [1,N]_\Z ) $ are independent, and by translation invariance of the law of $X$ together with our above observation about $\Upsilon_j^{n_k}$, the law of each $\Upsilon_j^{n_k}$ converges as $k$ tends to infinity to the law of $Z(t_j ) - Z(t_{j-1})$. Hence the increments $Z(t_j) - Z(t_{j-1})$ are independent and each has the same law as $Z(t_j - t_{j-1})$,
i.e., $Z$ has independent stationary increments.

By Proposition~\ref{prop-length-sup} and the Vitali convergence theorem, we find that for each $t\geq 0$, the first and second moments of the coordinates of $Z^{n_k}(t)$ converge to the corresponding quantities for $Z(t)$.
Convergence of the expectations implies that $\E( Z(t)) = 0$ for each $t \geq 0$,
and convergence of variances implies with Proposition~\ref{prop-length-sup} implies $Z(t)$ has finite variance.
Thus $Z$ is a continuous L\'evy process with independent stationary mean-zero increments, so $Z$ must be a two-dimensional Brownian motion with $Z(0)= 0$ and some variances, covariance, and zero drift.

Since $\cC(X'(1,n))$ is a simple random walk,
\eqbn
\lim_{k \rta \infty} \op{Var}\left(n_k^{-1/2} \cC(X'(1,n_k)) \right) = 1\,,
\eqen
and by Proposition~\ref{prop-var-limit},
\eqbn
\lim_{k \rta \infty} \op{Var}\left(n_k^{-1/2} \cQ(X'(1,n_k)) \right) = 1 + (p+q) \CHI\,.
\eqen
Furthermore, the conditional law of $X$ given $\cC(X(1,m))$
 for all $m\in\N$ is invariant under the involution operation~\eqref{eqn-involution} and this operation changes the sign of $\cQ(X'(1,m))$, so
\eqbn
\op{Cov}\left( \cC(X'(1,m)), \cQ(X'(1,m)) \right) = 0,\quad \forall m \in \N\,.
\eqen
Equivalently,
\alb
\op{Var}\left( U(t) + V(t) \right) &=1,\\
 \op{Var}\left( U(t) -  V(t) \right) &= 1 + (p+ q) \CHI,\\
 \op{Cov}\left( U(t) + V(t), U(t) - V(t) \right) &= 0.
\ale
Recalling the formula~\eqref{eqn-y-z}, this implies that $Z$ must be as in~\eqref{eqn-bm-cov-variable-SD}.

If the full sequence $\{Z^n\}_{n\in \N}$ failed to converge uniformly on compact subsets of $[0,\infty)$ to the law of $Z$, then there would be a subsequence
bounded away from the law of $Z$.  But by Prokhorov's theorem and the argument above,
there would be a subsubsequence converging in law to $Z$, a contradiction.
Hence the full sequence $\{Z^n\}_{n\in \N}$ converges uniformly on compact subsets of $[0,\infty)$, and thus on compact subsets of $\R$ by translation invariance.

The statement that $\CHI =2$ when $q = 0$ is established in Lemma~\ref{prop-J-count-mean}. We thus obtain the statement of the theorem when $p_\so = p$, $p_\db = q$, and $p_\fo = p_\eb=0$. By Corollary~\ref{prop-identification-law} we obtain the statement of the theorem in general.
\end{proof}

\section{Open problems}
\label{sec-open-problems}

Here we list some open problems related to the model studied in this paper, some of which were mentioned in the text.

\begin{enumerate}
\item Compute the value of the constant $\CHI$ in Theorem~\ref{thm-variable-SD} when $p_\db - p_\eb \neq0$ ($z\neq1$).
Figure~\ref{fig:chi-kappa} shows computer simulations of the value of $\CHI$ and the corresponding value of $\kappa$ in terms of $y$ and $z$.
\begin{figure}[h!]
 \begin{center}
\includegraphics[width=\textwidth/3]{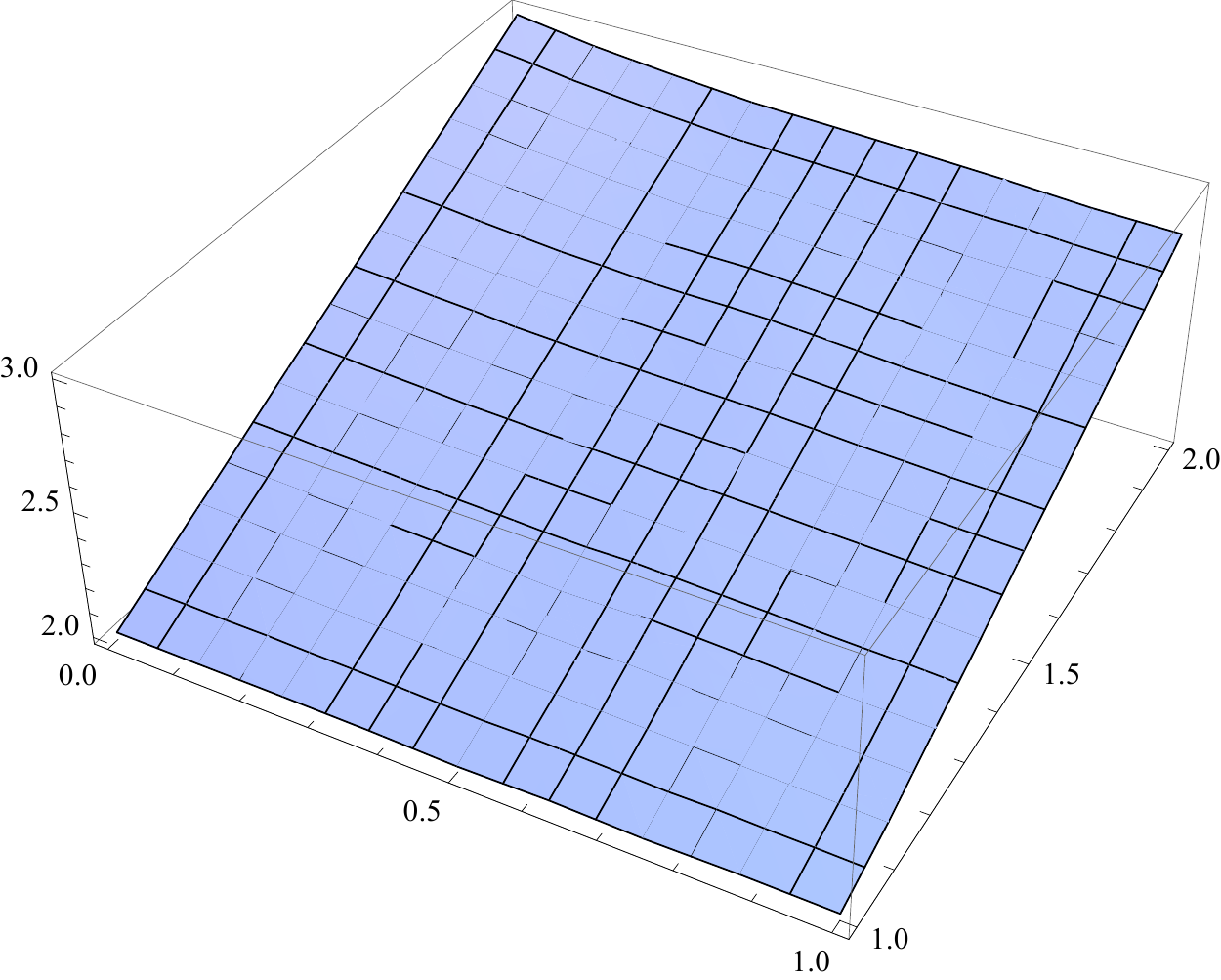}
\includegraphics[width=\textwidth/3]{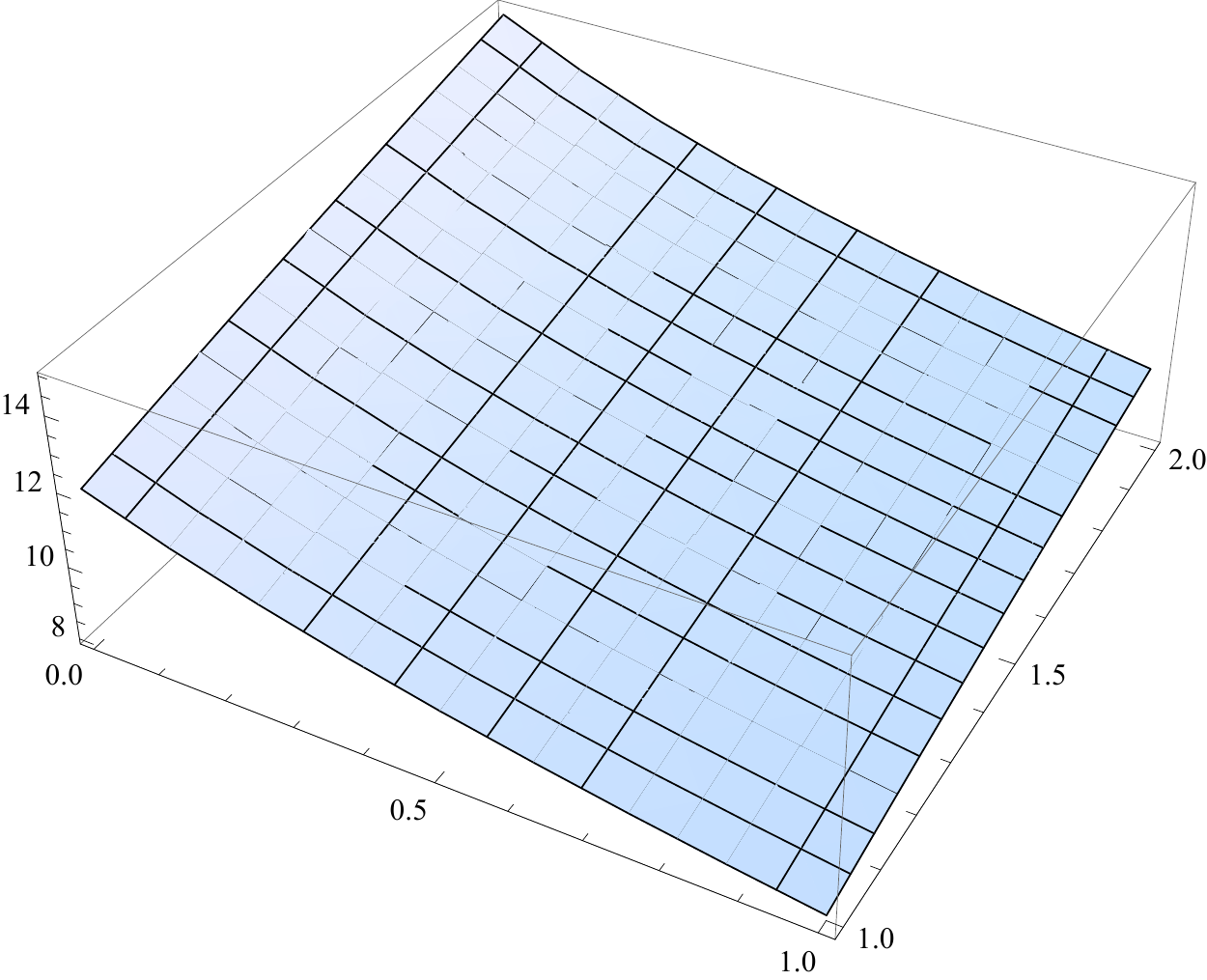}
\end{center}
\caption{Experimental plots for $\CHI$ and $\kappa$ as a function of $(y,z)\in[0,1]\times [1,2]$.}
\label{fig:chi-kappa}
\end{figure}
\item Prove an infinite-volume peanosphere scaling limit result similar to Theorems~\ref{thm-all-S} and~\ref{thm-variable-SD} in the case when $p_\so \neq1$ and $p_\db-p_\eb < 0$ ($y > 0$ and $z \in (0,1)$) or when $p_\fo - p_\so > 0$ and $p_\db-p_\eb \neq0$ ($y >1$ and $z\neq1$).
\item Prove a scaling limit result for the walk $Z^n|_{[0,2]}$ conditioned on the event that the reduced word $X(1,2n) =\emptyset$ (which encodes a finite-volume spanning-tree-decorated random planar map), possibly just in the case when $p_\db = p_\eb = p_\fo = 0$. See~\cite[Thm.~1.8]{gms-burger-finite} for an analogous result in the case when $p_\db = p_\eb = p_\so= 0$ and $p_\fo \in [0,1/2)$.
\item Prove a scaling limit result for the bending loop model of Remark~\ref{remark-bending}. In particular is there an encoding of this model in terms of a model on words analogous to the one studied in this paper?
\item For many statistical mechanics models on random planar maps which converge in the scaling limit to $\SLE_\kappa$-decorated LQG for some value of $\kappa > 0$, it is expected that the same model on a Euclidean lattice converges in the scaling limit to $\SLE_\kappa$.
Recall that for peanosphere scaling limit results, the correlation of the Brownian motion is given by $-\cos(4\pi/\kappa)$.
In light of Lemma~\ref{prop-activity} and Theorems~\ref{thm-all-S} and~\ref{thm-variable-SD}, it is therefore natural to make the following conjecture, which expands the conjecture in~\cite{kassel-wilson-active}.
\begin{conj}
 Let $\Lambda$ be either the triangular, hexagonal, or square lattice and suppose that either $y = 0$ and $z > 0$ or $y \in [0,1]$ and $z\in [1,\infty)$. Let $T$ be a spanning tree on $\Lambda$ sampled according to the law~\eqref{eqn-spanning-tree-law} (defined, e.g., by taking a limit of the law~\eqref{eqn-spanning-tree-law} on large finite sub-graphs of $\Lambda$) and let $\lambda$ be its associated Peano curve. Then $\lambda$ converges in law in the scaling limit to $\SLE_\kappa$, where $\kappa \geq 8$ is chosen so
\eqbn
- \cos\left(\frac{4\pi}{\kappa} \right) =
\begin{dcases}
-\frac{z}{1+z},\quad &y = 0 \\
-\frac{(z-y) \CHI}{(y+1)(z+1) + (z-y)\CHI} \quad &(y,z) \in [0,1] \times  [1,\infty),
\end{dcases}
\eqen
where $\CHI$ (depending on $y$ and $z$) is as in Theorem~\ref{thm-variable-SD}.
\end{conj}
Prove this conjecture.  The case when $y = z =1$ corresponds to the uniform spanning tree and has been treated in~\cite{lsw-lerw-ust}.  The case $(y,z)=(1+\sqrt{2},1)$ corresponds to the FK--Ising model and has recently been addressed in~\cite{kemp-smirnov-fk-bdy}.
\end{enumerate}

\appendix

\section{Basic properties of the burger model}
\label{sec-prelim}

Recall that a word in $\mcl W(\Theta)$ is called \textit{reduced\/} if all of its orders (i.e., elements of $\{\co,\ho,\fo,\so\}$), $\db$'s, and $\eb$'s lie to the left of all of its $\hb$'s and $\cb$'s.  

\begin{lem} \label{prop-reduction}
The reduction operation of Definition~\ref{def-reduce} is well-defined, i.e., for every finite word $x \in \mcl W(\Theta)$ there is a unique reduced word $\cR(x)$ which is equivalent to $x$ modulo the relations~\eqref{eqn-theta-relations} and~\eqref{eqn-theta-relations'}.
\old{ Moreover, $\cR(x)$ can be obtained from $x$ inductively by the formula
\eqb \label{eqn-reduction}
\cR(x) = \cR\left( \cR(x_1\dots x_{|x|-1}) x_{|x|} \right) .
\eqe }
\end{lem}
\begin{proof}
The proof follows the same argument as~\cite[Prop.~2.1]{shef-burger}.
To prove existence, we define $\cR(x)$ by induction on the length $|x|$ of the word, as follows. We set $\cR(x) =\emptyset$ if $|x|  = 0$ (i.e., $x = \emptyset$). Suppose now that $n\in\BB N$, a reduced word $\cR(x)$ equivalent to $x$ has been constructed for all words $\wt x $ with $|\wt x| \leq n-1$, and we are given a word $x=x_1\dots x_n$ with $|x| = n$. 

If we set $\wt x = x_1\dots x_{n-1}$, then the reduced word $\cR(\wt x)$ consists of a (possibly empty) word $ U$ consisting of orders, $\db$'s, and $\eb$'s followed by a (possibly empty) word $u$ consisting of $\hb$'s and $\cb$'s. If either $u = \emptyset$ or $x_{n} \in \{\hb, \cb\}$, then we set $\cR(x) = \cR(\wt x) x_{n}$ to get a reduced word equivalent to $x$. If $u\not=\emptyset$ and $x_{n} \in \{\db , \eb\}$, we replace $x_{n}$ by an $\hb$ or $\cb$ using~\eqref{eqn-theta-relations'}, then append this burger to the end of $\cR(\wt x)$ to get $\cR(x)$. If $u\not=\emptyset$ and $x_{n} = \ho$, then using~\eqref{eqn-theta-relations}, we either match $x_{n}$ to the rightmost $\hb$ in $u$ or (if $u$ has no $\hb$'s) we move $x_{n}$ to the position between $U$ and $u$ to form a reduced word $\cR(x)$ equivalent to $x$. We define $\cR(x)$ similarly if $x_{n} = \co$. Finally, if $u\not=\emptyset$ and $x_{n} \in \{\fo, \so\}$, we replace $x_{n}$ by an $\ho$ or $\co$ (depending on the rightmost burger in $u$) using~\eqref{eqn-theta-relations} then proceed as in the case $x_{n} \in \{\ho,\co\}$ to once again get a reduced word equivalent to $x$. 

To prove uniqueness, we observe that $\cR(x)$, as defined above, equals $x$ if $x$ is already reduced.
The relations~\eqref{eqn-theta-relations} and~\eqref{eqn-theta-relations'} list several pairs $(w, w')$ of two-letter or zero-letter words which are defined to be equivalent.
If $(w,w')$ is one of these pairs of words (e.g., $(w,w') = (\hb \so , \hb \co)$) then the above construction shows that replacing an instance of the word $w$ which appears anywhere in the word $x$ with the word $w'$ in the same position will have no effect on the reduced word $\cR(x)$. 
If $x$ and $x'$ are two words which are equivalent modulo the relations~\eqref{eqn-theta-relations} and~\eqref{eqn-theta-relations'}, then $x'$ can be obtained from $x $ by performing finitely many replacement operations of the above form. 
Consequently, $\cR(x) = \cR(x')$ whenever $x$ and $x'$ are equivalent. In particular, if $x$ and $x'$ are both reduced, then $x = x'$.  
\end{proof}

We make frequent use of the following consequence of uniqueness:
\begin{lem} \label{prop-associative} 
For any finite words $x , y \in \mcl W(\Theta)$, 
\eqb \label{eqn-reduce-associative}
\cR( \cR(x) \cR(y)) = \cR(xy) .
\eqe
\end{lem}
\begin{proof}
If $x,y,x',y' \in \mcl W(\Theta)$ are finite words such that $x$ is equivalent to $x'$ and $y$ is equivalent to $y'$ modulo the relations~\eqref{eqn-theta-relations} and~\eqref{eqn-theta-relations'}, then $xy$ is equivalent to $x'y'$ modulo these relations (since one can apply the relations to convert $x$ to $x'$, then apply the relations to convert $y$ to $y'$ separately). Consequently, $ \cR(x) \cR(y) $ and hence also $\cR(\cR(x) \cR(y))$ is equivalent to $xy$ modulo our relations. Since $\cR(\cR(x) \cR(y))$ is a reduced word,~\eqref{eqn-reduce-associative} follows from the uniqueness statement in Lemma~\ref{prop-reduction}. 
\end{proof}

The above lemmas were purely combinatorial, but the following ones also use the stochastic properties of the burger model. 
Fix $\pvec =(p_\fo,p_\so,p_\db,p_\eb) \in
\mcl P$ and let $X$ be a bi-infinite word whose symbols
$\{X_i\}_{i\in\Z}$ are i.i.d.\ with the law~\eqref{eqn-theta-prob}, as
in Section~\ref{sec-burger}.  We prove that a.s.\ each symbol in $X$
has a match and that the bi-infinite identification $X' = \cI(X)$ is
well-defined and contains only symbols in $\Theta_0$, i.e., that the
objects of Definitions~\ref{def-X-identification} and~\ref{def-match}
are well-defined.

\begin{lem} \label{prop-backward-burger} Let $m\in\N$ and let $P_m $
  be the smallest $j\in\N$ for which the net burger count satisfies
  $\cC(X(-j,-1)) = m$. If $X_{-P_m } \in \left\{\hb, \cb\right\}$,
  then for any $n \in [1,P_m]_\Z$, every symbol in $X_{-P_m}\cdots
  X_{-n}$ is identified, and the reduced word
  $X(-P_m,-n)$ contains at least one burger.
\end{lem}
\begin{proof}
A.s.\ $P_m$ is well-defined and finite.
Suppose $n\in[1,P_m]_\Z$.  By definition of $P_m$,
\old{\alb
\cC\left(X(-P_m, -1) \right) = \cC\left(X(-P_m, -n\right)+ \cC\left(X(-n+1, -1) \right) &= m \\
\cC\left(X(-n+1, -1) \right)  &\leq m-1\,,
\ale
}
$\cC\left(X(-P_m, -1) \right)=m$ and $\cC\left(X(-n+1, -1) \right) \leq m-1$, but we also have
\alb
\cC\left(X(-P_m, -1) \right) & = \cC\left(X_{-P_m} \cdots X_{-1}\right)\\
& = \cC\left(X_{-P_m} \cdots X_{-n}\right)+\cC\left(X_{-n+1} \cdots X_{-1}\right)\\
&= \cC\left(X(-P_m, -n)\right)+ \cC\left(X(-n+1, -1) \right)\,,
\ale
so $\cC(X(-P_m,-n))\geq 1$ and hence the reduced word $X(-P_m,-n)$ contains at least one burger.

We now show by induction on $P_m-n\in[0,P_m-1]_\Z$ that every symbol in $X_{-P_m}\cdots X_{-n}$ is identified.  Since $X_{-P_m}\in\{\hb,\cb\}$, the claim is true for $n=P_m$.  If the claim is true for $P_m\geq n\geq 2$, then since $X(-P_m,-n)$ contains a burger, each of which is by induction identified, the next symbol $X_{-n+1}$ becomes identified, completing the induction.
\end{proof}

\begin{cor}
Almost surely each symbol in $X$ is identified, i.e., the bi-infinite identified word $X'=\cI(X)$ is well-defined and contains only symbols in $\Theta_0$.
\end{cor}
\begin{proof}
For $m\in\N$, by the strong Markov property, conditional on the stopping time $P_m$, the word $X_{-P_{m+1}}\cdots X_{-P_m-1}$ is independent of $X_{-P_m}\cdots X_{-1}$, and in particular $\PP(X_{-P_{m+1}}\in\{\hb,\cb\}) = 1-p_\db-p_\eb > 0$.  Almost surely for some $m$ this event occurs, in which case, by Lemma~\ref{prop-backward-burger}, the symbol $X_0$ is identified.  By translation invariance, a.s.\ every symbol in $X$ is identified.
\end{proof}

\begin{lem} \label{prop-identification-exists}
Almost surely,
\eqb \label{eqn-burger-to-infty}
\lim_{n\rta\infty} \cN_{\hb}(X(-n,-1)) = \lim_{n\rta\infty} \cN_{\cb}(X(-n,-1)) = \infty.
\eqe
\end{lem}

\begin{proof}
Suppose $r\in\N$.  Let $\ell_0=0$, and for $k\in\N$ inductively define
\eqbn
\ell_k \colonequals \min \big\{t\in\N: \cC(X_{-t}\cdots X_{-\ell_{k-1}-1}) = 2 \ell_{k-1} + 2 r\big\}.
\eqen
Since $j \mapsto \cC(X(-j,-1))$ is a simple random walk, a.s.\ each of the times $\ell_k$ is finite.

By the strong Markov property, conditional on $\ell_{k-1}$, the word $X_{-\ell_k}\cdots X_{-\ell_{k-1}-1}$ is independent of
of $X_{-\ell_{k-1}}\cdots X_{-1}$.  With probability $1-p_\db-p_\eb$ the symbol $X_{-\ell_k}$ is $\hb$ or $\cb$,
in which case, by Lemma~\ref{prop-backward-burger}, each symbol of $X_{-\ell_k}\cdots X_{-\ell_{k-1}-1}$
is identified.  Conditional on this event, since $X_{-\ell_k}\cdots X_{-\ell_{k-1}-1}$ contains at least $2\ell_{k-1}+2 r$ burgers,
by symmetry it must be that with probability at least $1/2$ it contains at least $\ell_{k-1}+r$ hamburgers.  But
$X(-\ell_{k-1},-1)$ contains no more than $\ell_{k-1}$ hamburger orders, so $X(-\ell_k,-1)$ contains at least $r$ hamburgers.
So regardless of $X_{-\ell_{k-1}}\cdots X_{-1}$, with probability at least $(1-p_\db-p_\eb)/2$, the reduced word $X(-\ell_k,-1)$ contains at least $r$ hamburgers.  Almost surely this event will occur for some $k$.  Since $\cN_\hb(X(-n,-1))$ is monotone increasing in $n$, almost surely $\liminf_{n\to\infty} \cN_\hb(X(-n,-1)) \geq r$.
Since $r\in\N$ was arbitrary, \eqref{eqn-burger-to-infty} follows.
\end{proof}

\begin{lem} \label{prop-match}
Almost surely, for each $i\in \Z$ there is a unique $\phi(i) \in \Z$, called the \emph{match} of~$i$, such that $X_{\phi(i)}$ cancels out $X_i$ in the reduced word $X(\phi(i), i)$ if $\phi(i)< i$ or $X(i,\phi(i))$ if $i < \phi(i)$.
\end{lem}
\begin{proof}
If $X'_i$ is an order, from Lemma~\ref{prop-identification-exists} we see that a.s.\ it has a match in $X_{-n}\cdots X_i$
for some sufficiently large $n$ (using translation invariance).

Next we use the Mass-Transport Principle (see e.g.~\cite[Chapter~8]{Lyons-Peres}).  We let each order send one unit of mass to the burger that it consumes.
Since each order consumes a burger by the above paragraph and each letter is an order with probability $1/2$, the expected mass out of a letter is $1/2$.  Since this mass transport rule
is translation invariant, by the Mass-Transport Principle,
the expected mass into a letter is $1/2$.  Thus any given letter is with probability $1/2$ a burger that is consumed.
But it is a burger with probability $1/2$, so a.s.\ all burgers are consumed.
\end{proof}

\def\@rst #1 #2other{#1}
\renewcommand\MR[1]{\relax\ifhmode\unskip\spacefactor3000 \space\fi
  \MRhref{\expandafter\@rst #1 other}{#1}}
\renewcommand{\MRhref}[2]{\href{http://www.ams.org/mathscinet-getitem?mr=#1}{MR#1}}

\renewcommand{\arXiv}[1]{\href{http://arxiv.org/abs/#1}{arXiv:#1}}
\renewcommand{\arxiv}[1]{\href{http://arxiv.org/abs/#1}{#1}}

\phantomsection
\pdfbookmark[1]{References}{bib}

\def\cprime{$'$}

\end{document}